\RequirePackage{fix-cm}
\documentclass[smallextended]{svjour3}       
\smartqed  
\usepackage{graphicx}
\usepackage{amsmath,amsfonts,latexsym,amssymb,amsbsy}
\usepackage{mathtools, bm}

\usepackage{algorithm}
\usepackage[noend]{algpseudocode}

\algrenewcommand\textproc{\textsf}

\usepackage{booktabs}
\usepackage{lscape}
\usepackage{multicol,multirow}

\usepackage{enumerate}
\usepackage[inline]{enumitem}

\usepackage{pgf, tikz}
\usepackage{pgfplots}
\usetikzlibrary{arrows, automata, shapes}
\usetikzlibrary{positioning}
\pgfplotsset{scaled ticks=false}

\usepackage[colorinlistoftodos]{todonotes}

\usepackage{hyperref}
\hypersetup{
	colorlinks,
	linkcolor={red!50!black},
	citecolor={blue!50!black},
	urlcolor={blue!80!black}
}

\newcommand{\R}{\mathbb{R}}

\newcommand{\B}{\{0,1\}}
\newcommand{\prob}{\mathbb{P}}

\DeclareMathOperator*{\argmax}{arg\,max}

\newcommand{\maxwidth}{\mathcal{W}}
\newcommand{\width}{w}
\newcommand{\bdd}{\mathcal{B}}
\newcommand{\bddr}{\bdd^\mathsf{r}}
\newcommand{\nodes}{\mathcal{N}}
\newcommand{\arcs}{\mathcal{A}}
\newcommand{\nodesr}{\nodes^\mathsf{r}}
\newcommand{\arcsr}{\arcs^\mathsf{r}}

\newcommand{\rootnode}{\mathbf{r}}
\newcommand{\terminalnode}{\mathbf{t} }
\newcommand{\incoming}{\arcs^{\mathsf{in} }}
\newcommand{\outgoing}{\arcs^{\mathsf{out}}}
\newcommand{\source}{s}
\newcommand{\target}{t}
\newcommand{\arcLabel}{v}
\newcommand{\pathbdd}{p}
\newcommand{\pathvar}{\bm{x}}
\newcommand{\rt}{$\rootnode-\terminalnode$}
\newcommand{\sol}{X}
\newcommand{\paths}{\mathcal{P}}
\newcommand{\pathsr}{\mathcal{P}^\mathsf{r}}
\newcommand{\arclength}{\ell}

\newcommand{\costTop}{L^\downarrow}
\newcommand{\costBottom}{L^\uparrow}

\newcommand{\infoMinDown}{Q^{\downarrow\mathsf{min}}}
\newcommand{\infoMinUp}{Q^{\uparrow\mathsf{min}} }
\newcommand{\infoMaxDown}{Q^{\downarrow\mathsf{max}}}
\newcommand{\infoMaxUp}{Q^{\uparrow\mathsf{max}}}

\newcommand{\ds}{\mathrm{\lambda}}
\newcommand{\dszero}{\ds^0}
\newcommand{\dsone}{\ds^1}

\newcommand{\face}{F}
\newcommand{\slacks}{S}
\newcommand{\varIndex}{I}
\newcommand{\bx}{\bm{x}}
\newcommand{\by}{\bm{y}}
\newcommand{\bw}{\bm{w}}
\newcommand{\bQ}{\bm{Q}}
\newcommand{\ba}{\bm{a}}
\newcommand{\bd}{\bm{d}}
\newcommand{\blambda}{\bm{\lambda}}

\newcommand{\bomega}{\bm{\omega}}
\newcommand{\balpha}{\bm{\alpha}}
\newcommand{\bpi}{\bm{\pi}}

\newcommand{\bnu}{\bm{\nu}}
\newcommand{\bbeta}{\bm{\eta}}
\newcommand{\btheta}{\bm{\theta}}
\newcommand{\flow}{\mbox{NF}}
\newcommand{\flowjoint}{\mbox{JNF}}
\newcommand{\flowcap}{\mbox{CNF}}
\newcommand{\cutpoly}{P}
\DeclareMathOperator{\conv}{conv} 
\DeclareMathOperator{\proj}{Proj}
\DeclareMathOperator{\projx}{\proj_{\bx}}

\newcommand{\algReturn}{\textbf{return}}

\newcommand{\algBDDWidthOne}{\textsf{WidthOneBDD}}
\newcommand{\algBDDNodesBottom}{\textsf{UpdateBDDNodesBottom}}
\newcommand{\algBDDNodesTop}{\textsf{UpdateBDDNodesTop}}
\newcommand{\algBDDSplit}{\textsf{SplitBDDNodes}}
\newcommand{\algBDDFilter}{\textsf{FilterBDDOutgoingEdges}}
\newcommand{\algReduceBDD}{\textsf{ReduceBDD}}

\newcommand{\cplex}{$\texttt{CPLEX}$}
\newcommand{\bddFlow}{$\texttt{B-Flow}$}
\newcommand{\bddFlowLift}{$\texttt{B-Flow+L}$}
\newcommand{\bddGeneral}{$\texttt{B-Gen}$}
\newcommand{\bddGeneralLift}{$\texttt{B-Gen+L}$}
\newcommand{\bddGeneralNoFlow}{$\texttt{B-Gen+NoFlow+L}$}
\newcommand{\bddProject}{$\texttt{B-Proj}$}
\newcommand{\bddProjectLift}{$\texttt{B-Proj+L}$}
\newcommand{\bddTarget}{$\texttt{B-Target}$}
\newcommand{\bddTargetFlow}{$\texttt{B-Ta+Fl+L}$}
\newcommand{\bddCover}{$\texttt{Cover+L}$}
\newcommand{\cover}{$\texttt{Cover}$}
\newcommand{\coverLift}{$\texttt{CoverLift}$}
\newcommand{\hybrid}{$\texttt{B-Hybrid}$}

\newcommand{\lift}{$+\texttt{L}$}
\newcommand{\nlift}{$+\texttt{nL}$}

\newcommand{\scatter}[7]{
	\centering
	\begin{tikzpicture}
		\begin{axis}[ 
			xlabel={#6}, ylabel={#7},
			width=0.33\textwidth,
			height=0.33\textwidth,
			ymin=#2, ymax=#3, 
			xmin=#2, xmax=#3,
			max space between ticks=20,
			tick label style={font=\tiny}] 
			\addplot[only marks, mark =o, mark options={scale=0.6}, color= cyan] table[x=#4,y=#5] {#1};
			\addplot[black] coordinates {(#2,#2) (#3,#3)};
		\end{axis}
	\end{tikzpicture}
}

\definecolor{CommentColor}{rgb}{0.90,0.16,0}

\newcommand{\changes}{\color{black}}

\begin{document}

\title{A Combinatorial Cut-and-Lift Procedure with an Application to 0-1 Second-Order Conic Programming
}

\titlerunning{A Combinatorial Cut-and-Lift Procedure}        

\author{Margarita P. Castro \and  Andre A. Cire \and J. Christopher Beck }

\authorrunning{Castro, Cire, and Beck} 

\institute{ Margarita P. Castro 
\at Department of Mechanical and Industrial Engineering, University of Toronto \\ 
\email{mpcastro@mie.utoronto.ca}  
\and
Andre A. Cire
\at Dept. of Management, University of Toronto Scarborough and \\Rotman School of Management \\ 
\email{andre.cire@rotman.utoronto.ca}
\and
J. Christopher Beck
\at Department of Mechanical and Industrial Engineering, University of Toronto,\\ 
\email{jcb@mie.utoronto.ca}
}

\date{Received: date / Accepted: date}

\maketitle

\begin{abstract}

Cut generation and lifting are key components for the performance of state-of-the-art mathematical programming solvers. This work proposes a new general cut-and-lift procedure that exploits the combinatorial structure of 0-1 problems via a binary decision diagram (BDD) encoding of their constraints. We present a general framework that can be applied to a wide range of binary optimization problems and show its applicability for second-order conic inequalities. We identify conditions for which our lifted inequalities are facet-defining and derive a new BDD-based cut generation linear program. Such a model serves as a basis for a max-flow combinatorial algorithm over the BDD that can be applied to derive valid cuts more efficiently.  Our numerical results show encouraging performance when incorporated into a state-of-the-art mathematical programming solver, significantly reducing the root node gap, increasing the number of problems solved, and reducing the run-time by a factor of three on average.

\keywords{Lifting \and Cutting Planes \and Decision Diagrams \and Binary Optimization  \and Second-order Cones}
\end{abstract}

\newpage 

\section{Introduction}

Cutting plane methodologies have played a key role in the theoretical and computational development of mathematical programming 
\cite{bixby2004mixed,Nemhauser1988}. Extensive literature has focused on cuts that exploit special problem substructure, leading to an array of techniques that are now integral into state-of-the-art solvers \cite{lodi2010mixed}. For general problems, cuts are obtained either by leveraging disjunctive reformulations \cite{balas1993lift,balas1996mixed} or by \textit{lifting}, i.e., relaxing an initial inequality so that it is valid for a higher-dimensional polyhedron \cite{Gomory1969,Louveaux2003,Wolsey1976}. 

In this paper, we study both a cut generation procedure and a lifting approach for general binary optimization problems of the form 
%
\begin{align*}
\max_{\bx \in X \subseteq \B^n} \bm{c}^\top \bx,  \tag{BP} \label{eq:binary_model} 
\end{align*}

\noindent where the feasible set $X$ is arbitrary, e.g., possibly represented by a conjunction of linear and/or non-linear constraints. 
Our methodologies consist of exploiting \textit{network structure} via a binary decision diagram (BDD) embedding of $X$. 
A BDD is a graphical model that represents solutions as paths in a directed acyclic graph, which can be viewed as a network-flow reformulation of $X$. Such a model is potentially orders of magnitude smaller than an explicit representation of $X$ as it identifies and merges equivalent partial solutions. 
Several BDD encodings have already been investigated for linear and non-linear problems \cite{behle2007binary,bergman2019binary,lozano2018binary} and are used to exploit submodularity \cite{bergman2018nonlinear} or more general combinatorial structure \cite{bergman2016discrete}.

We propose a sequential lifting procedure that can be applied to any initial inequality (e.g., given by another cutting-plane technique). The lifting algorithm uses 0-1 disjunctions derived from a BDD representation of $X$ to rotate inequalities while maintaining their validity. 
{\changes We show that each step of our sequential lifting can be performed efficiently in the size of the BDD and, when applicable, increases the dimension of the face by at least one. We also establish conditions for which the inequality becomes facet-defining and draw connections between our procedure and existing lifting techniques from disjunctive programming \cite{balas2018disjunctive}, showing that our approach generalizes well-known lifting procedures for 0-1 inequalities \cite{balas1975facets,hammer1975facet,perregaard2001generating}.}
 
{\changes
For our cut generation approach, we propose a reformulation of the BDD polytope based on capacitated flows, which leads to an alternative cut generation linear program (CGLP) for separating infeasible points. We show that the set of cuts derived from this model defines the convex hull of the solutions encoded by the BDD, i.e., $X$, and that this methodology circumvents common issues of existing BDD cut techniques. 
Finally, we build on this model to develop a weaker but computationally faster alternative that solves a combinatorial max-flow/min-cut problem over the BDD to generate valid inequalities.}

For optimization problems where a BDD for $X$ may be exponentially large in $n$, our lifting and cut procedures remain valid when considering instead a limited-size \textit{relaxed} BDD for \ref{eq:binary_model}, i.e., where the BDD encodes a superset of $X$. Several efficient methods exist to build relaxed BDDs, such as only considering a subset of the problem constraints \cite{bergman2016discrete}. This approach is similar in spirit, e.g., to when a linear relaxation is used to lift cover inequalities of a single knapsack constraint \cite{balas1975facets}. {\changes We exploit the discrete relaxation given by the BDD as opposed to a continuous relaxation, which captures some of the combinatorial structure of the problem. Both our lifting procedure and combinatorial cuts are also of low complexity in the size of the BDD and, when a relaxed BDD is employed, its size can be controlled through a parameter limiting its maximum width.

To assess our lifting and cut generation procedure numerically, we apply our methodology to a class of second-order conic programming problems (SOCPs). Second-order conic (SOC) inequalities arise in many applications, including network design \cite{atamturk2018network}, assortment \cite{csen2018conic}, overcommitment \cite{cohen2019overcommitment}, and chance-constrained stochastic problems \cite{van1963minimum,lobo1998applications}.
This paper focuses on SOCPs of the form
\begin{align*}
	\max_{\bx \in \B^n} 	\;\; \left\{\bm{c}^\top \bx : \quad \ba_j^\top\bx + || D_j^\top \bx - \bm{h}_j ||_2 \leq b_j, \quad  \forall j \in \{1,...,m\}  \right\}, \tag{SP} \label{eq:soc_model}
\end{align*}
where $||\cdot||_2$ is the Euclidean norm and, for each $j$, $\ba_j, \bm{h}_j$, and $D_j$ are real vectors and matrices of appropriate dimension. SOCPs are supported by commercial solvers such as \cplex\ \cite{CPLEXManual} and Gurobi \cite{gurobi} which facilitates the evaluation with state-of-the-art techniques. Our methodology, thus, also aims at contributing to the active research in linearization/lifting methods \cite{vielma2008lifted,vielma2017extended} as well as cutting methods \cite{atamturk2009submodular,atamturk2010conic,atamturk2013separation,bhardwaj2015binary,lodi2019disjunctive} in this area. 

We investigate problems with multiple SOC inequalities, each reformulated as an BDD. We experiment on the SOC knapsack benchmark \cite{atamturk2009submodular,joung2017lifting}, and 270 randomly generated instances with more general and challenging SOC inequalities, incorporating our cut-and-lift approach into \cplex. We also compare with existing BDD cutting techniques \cite{tjandraatmadja2019target,davarnia2020outer} and a SOC cut-and-lift method \cite{atamturk2009submodular}, also noting that \cplex{} includes many state-of-the-art SOC techniques \cite{CPLEXManual}. 

Our numerical results indicated that (a) our lifting procedure reduced the average root gap up to 29\%  in our benchmark when applied to cuts generated by all procedures; (b) when cuts are added only at the root node, a hybrid technique combining our method with BDD target cuts \cite{tjandraatmadja2019target} is the most effective for SOC knapsacks, while for general SOC inequalities our cuts perform similarly to target cuts; and (c) when adding cuts during the tree search, the hybrid is also the best performing across all methods, particularly achieving the best final gaps, solution times, and number of solved instances for our general SOC benchmark. Furthermore, the hybrid BDD technique improves upon default \cplex\ and reduced up to 52.2\% of the root node gap, closed at last 17 instances in each benchmark, and decreased solution times by threefold.
}

The paper is structured as follows. \S\ref{sec:background} introduces notation and background material. \S\ref{sec:relatedWork} describes related works in the BDD and lifting literature. \S\ref{sec:lifting} describes our combinatorial lifting procedure while \S\ref{sec:bddcut} details our BDD-based cutting-plane algorithms. \S\ref{sec:soc} introduces the case study problem and describes the BDD encoding for SOC inequalities. Lastly, \S\ref{sec:experiments} and \S\ref{sec:conclusions} present the empirical evaluation and final remarks, respectively.

\section{Background} 
\label{sec:background}

This section introduces the notation used throughout this work and the background material on BDDs. 
For convenience, we assume $n \ge 1$ and let $\varIndex := \{1,...,n\}$ represent the component indices of any $n$-dimensional point $\bx$.

We denote by $\dim(P)$ the dimension of a polytope $P \subseteq [0,1]^n$. An inequality $\bpi^\top \bx \leq \pi_0$ with $\bpi \in \R^n$ and $\pi_0\in \R$ is valid for $P$ if $\bpi^\top \bx \leq \pi_0$ holds for all $\bx \in P$. The inequality defines a face of $P$ if $\face(\bpi) := \{\bx\in P: \bpi^\top \bx = \pi_0 \}$ is not empty, i.e., the inequality \textit{supports} $P$. A face $\face(\bpi)$ is a \textit{facet} if $\dim(\face(\bpi) ) = \dim(P) -1$; in such a case, $\bpi^\top \bx \leq \pi_0$ is \textit{facet-defining}.
Finally, we denote the convex hull of $P$ by $\conv(P)$. 

\medskip
\noindent \textit{Binary Decision Diagrams.} A BDD $\bdd$ is an extended representation of a set $\sol_\bdd \subseteq \{0,1\}^n$ as a network. Specifically, $\bdd= (\nodes, \arcs)$ is a layered directed acyclic graph with node set $\nodes$ and arc set $\arcs$. The node set $\nodes$ is partitioned into $n+1$ \textit{layers} $\nodes = (\nodes_1,..., \nodes_{n+1})$. The first and last layers are the singletons $\nodes_1=\{\rootnode\}$ and $\nodes_{n+1}=\{\terminalnode\}$, respectively, where $\rootnode$ is the root node and $\terminalnode$ is the terminal node. An arc $a = (u,u') \in \arcs$ has a source node $\source(a) = u$ and a target node $\target(a) = u'$ in consecutive layers, i.e., $u' \in \nodes_{i+1}$ whenever $u \in \nodes_i$ for $i \in \varIndex$.

The points of $\sol_\bdd$ are mapped to paths in the network, as follows. With each arc $a \in \arcs$ we associate a value $\arcLabel_a \in \{0,1\}$, where a node $u \in \nodes$ has at most one arc of each value emanating from it. Given an arc-specified \rt\ path $\pathbdd=(a_1,...,a_n)$ with $\source(a_1) = \rootnode$ and $\target(a_n)=\terminalnode$, we let $\pathvar^p := (\arcLabel_{a_1}, \arcLabel_{a_2}, \dots, \arcLabel_{a_n}) \in \{0,1\}^n$ be the $n$-dimensional point encoded by path $\pathbdd$. Thus, if $\paths$ is the set of all \rt\ paths in $\bdd$, the set of points represented by the BDD is
$\sol_\bdd = \bigcup_{\pathbdd \in \paths} \{ \pathvar^p \}.$

A BDD $\bdd$ is \textit{exact} for set $X \subseteq \{0,1\}^n$ when $X = \sol_\bdd$, i.e., there is a one-to-one relationship between the points in $X$ and the \rt\ paths in $\bdd$. Alternatively, $\bdd$ is \textit{relaxed} when $X \subseteq \sol_\bdd$, i.e., every point in $X$ maps to a path in $\bdd$ but the converse is not necessarily true. 

\begin{example}
	\label{exa:running}
	Consider $X=\{ \bx\in \B^4: 7x_1 + 5x_2 + 4x_3 +x_4 \leq 8 \}$. Figure \ref{fig:bdd-example} illustrates two exact BDDs for $X$: $\bdd_1$ on the left-hand side and $\bdd_2$ on the right-hand side. Dashed and solid arcs have a value of 0 and 1, respectively. 	
	Each point $\bx \in X$ is represented by a path in $\bdd_1$ and $\bdd_2$. For example,  $\bx=(1,0,0,1) \in X$ is encoded by the path $((\rootnode,u_2)$, $(u_2,u_4)$, $(u_4,u_5)$, $(u_5,\terminalnode))$ in $\bdd_1$, and by the path $((\rootnode,u_2')$, $(u_2',u_5')$, $(u_5',u_6')$, $(u_6',\terminalnode))$ in $\bdd_2$. \hfill $\square$
\end{example}

\begin{figure}[tb]
	\centering
	\tikzstyle{zero arc} = [draw,dashed, line width=0.5pt,->]
\tikzstyle{one arc} = [draw,line width=0.5pt,->]
\tikzstyle{optimal arc} = [draw,line width=2pt,->]
\tikzstyle{main node} = [circle,fill=gray!50,font=\tiny\bfseries, inner sep=1pt]
\begin{tikzpicture}[->,>=stealth',shorten >=1pt,auto,node distance=1cm,
thick]        
\node[main node] (r) at (0,0) {$\;\rootnode\;$};
\node[main node] (u1)  at (-1,-1)  {$u_1$};
\node[main node] (u2)  at (1,-1)  {$u_2$};
\node[main node] (u3)  at (-1,-2) {$u_3$};
\node[main node] (u4)  at (1,-2) {$u_4$};
\node[main node] (u5)  at (0,-3) {$u_5$};
\node[main node] (t) at (0,-4) {$\;\terminalnode\;$};

\node (l1) at (-2,-0.5) {\scriptsize{$x_1$:}};
\node (l2) at (-2,-1.5) {\scriptsize{$x_2$:}};
\node (l4) at (-2,-2.5) {\scriptsize{$x_3$:}};
\node (l5) at (-2,-3.5) {\scriptsize{$x_4$:}};

\path[every node/.style={font=\sffamily\small}]
(r) 
edge[zero arc] node [left] {} (u1)
edge[one arc] node [left] {} (u2)
(u1) 
edge[zero arc] node [right] {} (u3)
edge[one arc] node [right] {} (u4)
(u2)
edge[zero arc] node [left] {} (u4)
(u3)
edge[zero arc, bend right=20] node [right] {} (u5)
edge[one arc, bend left=20] node [right] {} (u5)
(u4)
edge[zero arc] node [left] {} (u5)
(u5)
edge[zero arc,  bend right=45] node [left] {} (t)
edge[one arc,  bend left=45] node [left] {} (t);

\end{tikzpicture}
	\hspace{4em}
	\tikzstyle{zero arc} = [draw,dashed, line width=0.5pt,->]
\tikzstyle{one arc} = [draw,line width=0.5pt,->]
\tikzstyle{optimal arc} = [draw,line width=2pt,->]
\tikzstyle{main node} = [circle,fill=gray!50,font=\tiny\bfseries, inner sep=1pt]
\begin{tikzpicture}[->,>=stealth',shorten >=1pt,auto,node distance=1cm,
thick]        
\node[main node] (r) at (0,0) {$\;\rootnode\;$};
\node[main node] (u1)  at (-1,-1)  {$u_1'$};
\node[main node] (u2)  at (1,-1)  {$u_2'$};
\node[main node] (u3)  at (-1,-2) {$u_3'$};
\node[main node] (u4)  at (0,-2) {$u_4'$};
\node[main node] (u5)  at (1,-2) {$u_5'$};
\node[main node] (u6)  at (0,-3) {$u_6'$};
\node[main node] (t) at (0,-4) {$\;\terminalnode\;$};

\node (aux1) at (1.3, 0) {};
\node (aux2) at (2.5, 0) {\scriptsize$\arcLabel_a =0$};
\node (aux3) at (1.3, -0.4) {};
\node (aux4) at (2.5, -0.4) {\scriptsize$\arcLabel_a =1$};

\path[every node/.style={font=\sffamily\small}]
(r) 
edge[zero arc] node [left] {} (u1)
edge[one arc] node [left] {} (u2)
(u1) 
edge[zero arc] node [right] {} (u3)
edge[one arc] node [right] {} (u4)
(u2)
edge[zero arc] node [left] {} (u5)
(u3)
edge[zero arc, bend right=20] node [right] {} (u6)
edge[one arc, bend left=20] node [right] {} (u6)
(u4)
edge[zero arc] node [left] {} (u6)
(u5)
edge[zero arc] node [left] {} (u6)
(u6)
edge[zero arc,  bend right=45] node [left] {} (t)
edge[one arc,  bend left=45] node [left] {} (t)

(aux1) edge[zero arc] node [left] {} (aux2)
(aux3) edge[one arc] node [left] {} (aux4);

\end{tikzpicture}
	\caption{Two BDDs $\bdd_1$ (left-hand side) and $\bdd_2$ (right-hand side) with $\sol_{\bdd_1}=\sol_{\bdd_2}=\{ \bx\in \B^4: 7x_1 + 5x_2 + 4x_3 +x_4 \leq 8 \}$. $\bdd_1$ is reduced.}  
	\label{fig:bdd-example} 
\end{figure}
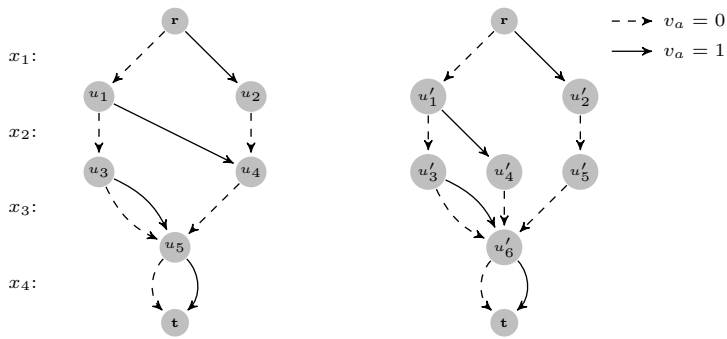

A BDD $\bdd$ is \textit{reduced} if it is the smallest network (with respect to number of nodes) that represents the set $\sol_\bdd$. 
There exists a unique reduced BDD for a given ordering of the indices $\varIndex$. Furthermore, given any $\bdd$ and an ordering, we can obtain its associated reduced BDD in polynomial time in the size of $\bdd$ \cite{bryant1986graph}. For instance, $\bdd_1$ in Figure \ref{fig:bdd-example} is reduced and can be obtained by merging nodes $u_4'$ and $u_5'$ from $\bdd_2$ and adjusting their emanating arcs appropriately.

Several exact and relaxed BDD construction mechanisms are available for general and specialized discrete optimization problems \cite{bergman2016discrete,bergman2018nonlinear,tjandraatmadja2019target}. These techniques either reformulate the problem as a dynamic program, where $\bdd$ represents an underlying state-transition graph, or separate infeasible paths of a relaxed BDD. It is often the case that 
BDDs can be exponentially smaller than enumerating $\sol_\bdd$ explicitly \cite{bergman2016discrete}. 
If exact BDDs are too large, relaxed BDDs can be built for $X$ by either imposing a limit on the number of nodes 
or by considering a subset of the constraints. We discuss the construction and relaxation techniques used for our case study in \S \ref{sec:soc}.

\section{Related Work} 
\label{sec:relatedWork}

Recent research has shown the versatility of BDDs for modeling linear and non-linear inequalities \cite{andersen2007constraint,hoda2010systematic,bergman2018quadratic,bergman2018nonlinear} and there is a growing literature on BDD encodings for vehicle routing \cite{castro2019mdd,raghunathan2018seamless,kinable2017hybrid}, scheduling \cite{cire2013multivalued,van2018multi,van2019lower,hooker2017job}, and other combinatorial optimization problems \cite{bergman2019binary,bergman2012variable,castro2019relaxedbdds,bergman2016discrete}.
Within the context of this work, Becker et al. (2005) \cite{becker2005bdds} presented the first BDD cut generation procedure based on an iterative subgradient algorithm that relies on a longest-path problem over the BDD. Behle (2007) \cite{behle2007binary} formalized this procedure and proposed a branch-and-cut algorithm that employs BDDs to generate exclusion and implication cuts. The author also introduced the network flow model employed by most BDD cutting-plane procedures \cite{davarnia2020outer,lozano2018binary,tjandraatmadja2019target}. 

{\changes
More recently, two BDD-based cutting plane techniques have been proposed with theoretical and computational considerations. Tjandraatmadja and van Hoeve (2019) \cite{tjandraatmadja2019target} generate target cuts from polar sets using relaxed BDDs to derive a more tractable methodology that in Becker et al. (2005) \cite{becker2005bdds}. Davarnia and van Hoeve (2020) \cite{davarnia2020outer} develop an iterative method to generate outer-approximations for non-linear inequalities, relying on a subgradient algorithm to avoid solving linear programs. We compare these two models conceptually to our approach in \S\ref{sec:bddcut-literature} and evaluate them numerically in \S\ref{sec:experiments}. 

Our cutting plane methods are related to the method by Lozano and Smith (2018) \cite{lozano2018binary} for a class of two-stage stochastic programming problems. The authors propose a Benders decomposition approach using BDDs to encode second-stage decisions, where arcs are can be activated or deactivated based on first-stage decisions. Benders cuts are obtained by solving a network-flow model over the BDD where arcs are bounded by the first-stage variables. Similarly, our CGLP model proposed in \S\ref{sec:bddcut} is also based on a capacitated network-flow model. However, our model is more structured in that it incorporates one variable per layer for all arc types, leading to more specialized inequalities and structural results for separation purposes (e.g., Lemma \ref{lem:separation_bdd} and Theorems \ref{theo:bddseparation} and \ref{theo:reduced_bdd}). We also leverage this model to develop our combinatorial max-flow cuts that do not depend on linear programming (LP) solutions.

Our numerical case study is focused on SOCPs. Recent work in the field relies on ideas from split cuts \cite{modaresi2015split}, disjunctions \cite{kilincc2015two,lodi2019disjunctive}, or are more specialized \cite{kilincc2016minimal,santana2017some}. Techniques based on mixed-integer rounding \cite{atamturk2010conic} and lift-and-project \cite{stubbs1999branch} have also shown to be suitable in practice, and are currently implemented in commercial solvers \cite{CPLEXManual}. Several recent works also exploit SOCPs with special structure, such as binary SOC knapsack inequalities \cite{atamturk2009submodular,bhardwaj2015binary,atamturk2013separation,joung2017lifting}. 


While the literature on BDD cutting-plane procedures has grown recently, to the best of our knowledge, this is the first work that leverages BDDs to lift general form linear inequalities. Behle (2007) \cite{behle2007binary} proposes lifting cover inequalities using classic techniques that compute new coefficients one at a time \cite{wolsey1999integer} and where each sub-problem is solved using a BDD.  Becker et al. (2005) \cite{becker2005bdds} also present a mechanism that uses 0-1 disjunctions over a BDD to obtain new inequalities. Their technique differs from ours with respect to both the procedure to obtain the new inequality and its theoretical guarantees. In particular, their lifted inequality might not separate fractional points that the original inequality does nor induce a face with higher dimension.

Our combinatorial lifting relates to sequential lifting algorithms based on 0-1 disjunctions \cite{balas2018disjunctive}, including specialized procedures for the knapsack polytope \cite{balas1975facets,padberg1975note,padberg1973facial} and submodular inequalities \cite{hammer1975facet,atamturk2009submodular}. These procedures successfully address special cases of the general lifting problem we investigate (see \S\ref{sec:lifting}), focusing on given problem structures (e.g., monotone sets). In particular, lifting cover inequalities is a well-studied area \cite{gu1998liftedCom,gu1999liftedCplx}, often using the classical knapsack dynamic program to efficiently lift coefficients \cite{zemel1989easily} when such constraints are present. In contrast, our approach is general in that it can be applied to any type of valid linear inequality (i.e., not restricted to cover inequalities) or feasibility set, including those defined by non-linear constraints. We also note that BDD sizes can be parameterized for large-scale problems.

Lastly, our methodology is closely related to the $n$-step lifting procedure by Perregaard and Balas (2001) \cite{perregaard2001generating}, which generalizes the special cases mentioned above (e.g., lifting cover inequalities). We briefly introduce this procedure below and relate it to our combinatorial lifting algorithm in \S \ref{sec:dp_relationship}.
}

\medskip
\noindent \textit{An Iterative Lifting Procedure based on Disjunctive Programming.} 
Given a mixed-integer linear programming (MILP) problem of the form
$ \max_{\bx} \{ \bm{c}^\top \bx \colon$  $A \bx \leq \bm{b}, x_i \in \mathbb{Z}, \;\; \forall i \in \varIndex' \subseteq \varIndex \}$, the authors \cite{perregaard2001generating} propose the relaxation
\begin{align*}	
\max_{\bx} \bigg\{\bm{c}^\top \bx:\; 
A\bx \leq \bm{b}, \;
\bigvee_{k \in K} D^k\bx \leq \bd^k, x_i \in \mathbb{Z} \;\;\; \forall i \in \varIndex'' \subset \varIndex' \bigg\}, \tag{DP} \label{eq:dp_model}	
\end{align*}
where fewer variables are constrained to be integral. The set $K$ that defines the disjunctive constraints is typically derived by considering the 0-1 integrality constraints of individual variables (e.g., $x_i \le 0 \vee x_i \ge 1$).

Let $P_{DP}$ be the set of solutions of \ref{eq:dp_model}. The $n$-step procedure considers two inputs: (a) an inequality $\bpi^\top\bx \leq \pi_0$ that supports $\conv(P_{DP})$; and (b) an arbitrary target inequality $\widetilde{\bpi}^\top \bx \leq \widetilde{\pi}_0$ that is tight for all integer points in $\face(\bpi)$. The procedure uses a parameter $\gamma$ to rotate the supporting inequality towards the target inequality, generating a new lifted inequality $(\bpi + \gamma\widetilde{\bpi})^\top\bx \leq \pi_0 + \gamma \widetilde{\pi}_0$ that is valid for $\conv(P_{DP})$. In particular, if $\widetilde{\bpi}^\top \bx \leq \widetilde{\pi}_0$ is \textit{not} valid for $P_{DP}$, it can be shown that there is a finite maximal $\gamma$  given by the disjunctive program
\begin{align*}
	\gamma^*= \min_{\bx, x_0} &\bigg\{ \pi_0x_0 - \bpi^\top \bx: \;\;  A\bx - \bm{b}x_0 \leq 0, \;\; \bigvee_{k \in K} D^k\bx -  \bd^k x_0 \leq 0 ,\\
	& \hspace{16ex} \widetilde{\pi}_0 x_0 - \widetilde{\bpi}^\top \bx = -1, \;\;
	x_0 \geq 0, \;\; x_i \in \mathbb{Z} \;\; \forall i \in \varIndex'' \subset \varIndex' \Big\}.
\end{align*}
Under the same assumptions, the lifted inequality becomes a facet of $\conv(P_{DP})$ if the procedure is repeated $n$ times, using the rotated inequality and an appropriate target inequality. 

Similarly, our approach is a sequential procedure that relies on disjunctions. It differs from the above method in that we exploit the combinatorial structure encoded by a BDD as opposed to a disjunctive program relaxation. Such a BDD may encode, e.g., complex non-linear constraints that are not necessarily convex \cite{bergman2018nonlinear}. Furthermore, we also exploit the network to derive a tractable and efficient way to compute several disjunctions simultaneously, while previous algorithms are typically restricted to a small number of disjunctions \cite{perregaard2001generating}.

\section{Combinatorial Lifting} \label{sec:lifting}

We now present our combinatorial lifting procedure and develop its structural properties. 
{\changes 
We begin by introducing our basic methodology in \S\ref{sec:dslacks}, which is defined in general terms and does not depend on an BDD $\bdd$ encoding. Next, in \S\ref{sec:lift-bdd} we present a methodology that exploits network structure to perform the proposed lifting in polynomial time in the size of $\bdd$ (i.e., in the number of nodes and arcs). Next, \S \ref{sec:sequential} incorporates the technique in a sequential procedure and investigate the dimension of the resulting face. Finally, we depict the relationship with previous disjunctive methodologies in \S \ref{sec:dp_relationship}.
}

Throughout this section, we assume that, for a given $X \subseteq \{0,1\}^n$,
\begin{enumerate*}[label=(\alph*)]
\item  inequality $\bpi^\top \bx\leq \pi_0$ is  valid and supports $\conv(X)$; \label{cond:support}
\item $\bdd$ is an exact BDD for $X$, i.e., $X_\bdd = X$; and \label{cond:exact}
\item for any $i \in \varIndex$, there exists $\bx, \bx' \in X$ such that $x_i = 0$ and $x'_i = 1$. \label{cond:dimension}
\end{enumerate*}
Assumption \ref{cond:support} is a common lifting condition that is satisfied by setting $\pi_0 := \max_{\bx \in X} \left\{ \bpi^\top \bx \right\}$. This, in turn, can be enforced in linear time in the size of $\bdd$ (see \S\ref{sec:lift-bdd}). Assumption \ref{cond:exact} is needed for our theoretical results but it can be relaxed in practice (see \S \ref{sec:experiments}). For \ref{cond:dimension}, we can soundly remove any $i$-th component not satisfying the assumption, adjusting $n$ accordingly. 

Our goal is to lift $\bpi^\top \bx\leq \pi_0$ and better represent $\conv(X)$ by exploiting the network structure of $\bdd$. The resulting cuts are valid for any subset $X' \subseteq X$; e.g., when $\bdd$ (and hence $X$) is a relaxation of some feasible set.

\subsection{Disjunctive Slack Lifting}
\label{sec:dslacks}

The core element of our lifting procedure is what we denote by \textit{disjunctive slack vector} (or \textit{d-slack} in short). The $i$-th component of the d-slack indicates the change in the maximum values of the left-hand side of $\bpi^\top \bx\leq \pi_0$ when varying $x_i$. This is formalized in Definition \ref{def:slack}.
\begin{definition} 
	\label{def:slack}
	The disjunctive slack vector $\blambda(\bpi)$ with respect to $\bpi$ is given by
	\[ 
		\ds_i(\bpi) := \dszero_i(\bpi) - \dsone_i(\bpi),  \;\;\;\; \forall i \in \varIndex,
	\]
	with $\dszero_i(\bpi) := \max_{\bx \in X}\{\bpi^\top \bx:x_i =0 \}$ and  $\dsone_i(\bpi) : = \max_{\bx \in X}\{\bpi^\top \bx: x_i =1 \}$.
\end{definition}

For notational convenience, we let $\slacks^{-}(\bpi) := \{i \in \varIndex: \ds_i(\bpi) < 0\}$, $\slacks^{0}(\bpi) := \{i \in \varIndex: \ds_i(\bpi) = 0\}$, and $\slacks^{+}(\bpi) := \{i \in \varIndex: \ds_i(\bpi) > 0\}$ be a partition of $\varIndex$ with respect to negative, zero, and positive d-slacks, respectively. Lemma \ref{lem:properties_ds} presents key properties of d-slacks used for our main results. 
\begin{lemma} \label{lem:properties_ds}
	For any $\blambda(\bpi)$ and index $i \in \varIndex$,
	\begin{enumerate}[label=(\arabic*)] 
	\item $i \in \slacks^{-}(\bpi)$ if and only if $x_i = 1$ for all $\bx \in \face(\bpi)$.
	\item $i \in \slacks^{+}(\bpi)$ if and only if $x_i = 0$ for all $\bx \in \face(\bpi)$.
	\item $i \in \slacks^{0}(\bpi)$ if and only if there exists $ \bx, \bx' \in \face(\bpi)$ with $x_i=0$ and $x'_i=1$.
	\end{enumerate}
\end{lemma}
\begin{proof}
	For the necessary conditions, consider first $x_i = 1$ for all $\bx \in \face(\bpi)$. Since the solutions when optimizing over $\bpi^\top \bx$ must belong to the face $\face(\bpi)$, we must necessarily have $\dsone_i(\bpi) > \dszero_i(\bpi)$ and so the d-slack $\ds_i(\bpi)$ is negative. An analogous reasoning holds for the other two cases.	
	
	For the sufficient conditions, consider first $i \in \slacks^{-}(\bpi)$. Then $\dsone_i(\bpi) = \pi_0$ and $\dszero_i(\bpi) < \pi_0$, i.e., all $\bx\in \face(\bpi)$ are such that $x_i=1$. The same argument can be applied to the case $i \in \slacks^{+}(\bpi)$. Lastly, if $i \in \slacks^{0}(\bpi)$, $\dszero_i(\bpi) = \dsone_i(\bpi)=\pi_0$. Thus, there exists $\bx\in \face(\bpi)$ that maximizes $\dsone_i(\bpi)$ (i.e., $x_i = 1$) and $\bx'\in \face(\bpi)$ that maximizes $\dszero_i(\bpi)$ (i.e., $x'_i = 0$). \hfill $\blacksquare$
\end{proof}

We now show in Theorem \ref{theo:LiftingDim} how to apply the d-slacks to lift $\bpi^\top \bx\leq \pi_0$. In particular, the resulting inequality is valid for $X$ (and thereby $\conv(X)$), the dimension of the face necessarily increases, and points separated by the original inequality are still separated after lifting. This last characteristic is important, e.g., if the input inequality 
$\bpi^\top \bx \leq \pi_0$ was derived to separate a fractional point. Note that we require a d-slack with a non-zero component to rotate the inequality, as we later illustrate in Example \ref{exa:Facet}. 

\begin{theorem} 
	\label{theo:LiftingDim}
	
	Suppose $\lambda_i(\bpi) \neq 0$ for some $i \in \varIndex$. Let $\langle\bpi', \pi'_0\rangle$ be such that 	
	\begin{align*}
		&\pi'_j 
		:=  
		\begin{cases}
			\pi_j 						& \textnormal{if $j \neq i$}, \\
			\pi_j + \lambda_j(\bpi) 	& \textnormal{otherwise,}
		\end{cases}
		&\textnormal{$\forall j \in \varIndex$}, 
		\;\quad
		&\pi_0' 
		:=  
		\begin{cases}
			\pi_0  				& \textnormal{if} \; i \in \slacks^{+}(\bpi), \\
			\pi_0 + \ds_i(\bpi) & \textnormal{otherwise.}
		\end{cases}	
	\end{align*}
	The following properties hold:
	\begin{enumerate}[label=(\arabic*)] 
		\item $\bpi'^\top\bx  \leq  \pi'_0$ is valid for $X$. \label{cond:validity}
		\item $\face(\bpi) \subset \face(\bpi')$ and $\dim(\face(\bpi') )\geq \dim(\face(\bpi)) +1$. \label{cond:lifting}
		\item For any $\bar{\bx}\in [0,1]^n$ with $\bpi^\top\bar{\bx} > \pi_0$, we have  that $\bpi'^\top \bar{\bx} > \pi'_0$. \label{cond:fractional}
	\end{enumerate}
\end{theorem}

\begin{proof}
	Let $\bx \in X$. We begin by showing \ref{cond:validity} and \ref{cond:lifting}. Assume first that $i \in S^{+}(\bpi)$. By construction,
	$
		\bpi'^\top\bx  \le \pi'_0 \Longleftrightarrow \bpi^\top \bx + \ds_i(\bpi)x_i \le \pi_0.
	$
	
	If $x_i = 0$, the lifted inequality is equivalent to the original and therefore valid. Otherwise, if $x_i = 1$, $i \in \slacks^{+}(\bpi)$ implies that $\ds_i(\bpi) = \pi_0 - \dsone_i(\bpi)$. Thus,
	\begin{align*}
		\bpi'^\top\bx  \le \pi'_0 \Longleftrightarrow \bpi^\top \bx + \pi_0 - \dsone_i(\bpi) \le \pi_0 \Longleftrightarrow \bpi^\top \bx \le \dsone_i(\bpi).
	\end{align*}
	The last inequality above holds because we are restricting to the case $x_i = 1$ and, by definition,  $\dsone_i(\bpi) = \max_{\bx' \in X}\{\bpi^\top \bx': x'_i = 1 \}$. Since $x'_i = 0$ for all $\bx' \in \face(\bpi)$ (Lemma \ref{lem:properties_ds}), the lifted inequality is tight for all $\bx' \in \face(\bpi)$, i.e.,  $\face(\bpi) \subset \face(\bpi')$. Notice also that this inequality is tight for $\bx^* = \argmax_{\bx' \in X}\{\bpi^\top \bx': x'_i = 1 \}$, i.e.,
	$\bx^* \in \face(\bpi')$. Then, $\bx^*$ is affinely independent to all points of $\face(\bpi)$ and therefore  $\dim(\face(\bpi') )\geq \dim(\face(\bpi)) + 1$.

	Assume now that $i \in S^{-}(\bpi)$. Once again by construction,
	$\bpi'^\top\bx  \le \pi'_0$ if and only if $\bpi^\top \bx + \ds_i(\bpi)x_i \le \pi_0 + \lambda_i(\bpi)$. If $x_i = 1$, the lifted inequality is equivalent to the original and therefore valid. Otherwise, if $x_i = 0$, $i \in \slacks^{-}(\bpi)$ implies that $\ds_i(\bpi) = \dszero_i(\bpi) - \pi_0$. Thus,
	\begin{align*}
		\bpi'^\top\bx  \le \pi'_0 \Longleftrightarrow \bpi^\top \bx \le \pi_0 + \dszero_i(\bpi) - \pi_0 \Longleftrightarrow \bpi^\top \bx \le \dszero_i(\bpi).
	\end{align*}
	The last inequality above holds because $x_i = 0$ and, by definition,  $\dszero_i(\bpi) = \max_{\bx' \in X} \{\bpi^\top \bx': x'_i = 0 \}$. As before, notice that this inequality is tight for $\bx^* = \argmax_{\bx' \in X}\{\bpi^\top \bx': x'_i = 0 \}$, i.e., $\bx^* \in \face(\bpi')$. Since $x'_i = 1$ for all $\bx' \in \face(\bpi)$ (Lemma \ref{lem:properties_ds}), the inequality is tight for all $\bx'\in \face(\bpi)$,
	$\bx^*$ is affinely independent to all points of $\face(\bpi)$, and therefore $\dim(\face(\bpi') )\geq \dim(\face(\bpi)) + 1$.
	
	Lastly, we demonstrate \ref{cond:fractional}. We restrict to the case $i \in \slacks^{+}(\bpi)$; the other case is analogous. Given a fractional point $\bar{\bx}\in [0,1]^n$ as defined above, we have $\bpi'^\top \bar{\bx} = \bpi^\top \bar{\bx} + \ds_i(\bpi)\bar{x}_i > \pi_0 +  \ds_i(\bpi)\bar{x}_i \geq  \pi_0=\pi_0'$.  \hfill $\blacksquare$
	
\end{proof}

\begin{example}\label{exa:Facet}
	Let $X=\{ \bx\in \B^4: 7x_1 + 5x_2 + 4x_3 +x_4 \leq 8 \}$ and consider an inequality $x_1 + x_2  \leq 1$ supporting $\conv(X)$. The d-slack is $\blambda(\bpi) = (0,0,1, 0)^\top$ and the lifted inequality with respect to $\ds_3(\bpi) = 1$ is $\bpi'^\top\bx =x_1 + x_2 + x_3 \leq 1$. Note that $\bpi'^\top\bx\leq 1$ is facet-defining for $\conv(X)$ and $\blambda(\bpi')= \bm{0}$. \hfill $\square$
\end{example}

\subsection{Extracting Disjunctive Slacks from a BDD} 
\label{sec:lift-bdd}

Identifying d-slacks $\blambda(\bpi)$ is a non-trivial task since we are required to solve $2n$ binary optimization problems, i.e., one for each component $i \in \varIndex$ and values 0 and 1. 
{\changes In this section, we leverage the network representation of a BDD $\bdd=(\nodes, \arcs)$ for $X$ to generate all d-slacks simultaneously, which is key to the computational complexity of the approach. We also show that the procedure complexity is linear in the number of arcs $|\arcs|$ of $\bdd$.}

We associate a \textit{length} of $\pi_i \cdot \arcLabel_a$ to each arc $a \in \arcs$ with value $\arcLabel_a \in \{0,1\}$ and source $\source(a) \in \nodes_i$ for some $i \in \varIndex$. 
The longest \rt\ path of $\bdd$ with respect to such lengths maximizes $\bpi^\top \bx$ over $X$.
Given the \rt\ paths $\paths$ of $\bdd$, let
\[
	\arclength_a := \max \left\{ \sum_{k=1}^n \pi_k \cdot \arcLabel_{a_k} \colon \textnormal{$\pathbdd = (a_1,\dots,a_n) \in \paths,  a_i=a$} \right\} 
\]
be the longest-path value conditioned on all paths that include arc $a$. Because each variable is uniquely associated with a layer, it follows that
\begin{align*}
	\ds^j_i(\bpi) = \max_{a \in \arcs} \left\{ \arclength_a \colon \source(a) \in \nodes_i, \; \arcLabel_a = j  \right\}, \quad \textnormal{$\forall i \in \varIndex$, $\forall j \in \B$},
\end{align*}
and the final d-slacks are obtained by the differences $\dszero_i(\bpi) - \dsone_i(\bpi)$ for all $i$. 


The lengths $\arclength_a$ are derived by performing two longest-path computations over $\bdd$. Specifically, let 
$\incoming(u)$ and $\outgoing(u)$ be the set of incoming and outgoing arcs of a node $u \in \nodes$, respectively. The solution of the recursion $\costTop(\bpi, \rootnode) = 0$,
\begin{align*}
	\costTop(\bpi, u) 				&= \max_{a \in \incoming(u)} \{ \costTop(\bpi,\source(a) ) + \pi_{i-1} \cdot \arcLabel_a \},  &\textnormal{$\forall u \in \nodes_i, \forall i \in \{2,\dots,n+1\}$}
\end{align*}
provides the longest-path value $\costTop(\bpi, u)$ from $\rootnode$ to $u$, while  $\costBottom(\bpi, \terminalnode) = 0$,
\begin{align*}
	\costBottom(\bpi, u) 				&= \max_{a \in \outgoing(u)} \{ \costBottom(\bpi,\target(a) ) + \pi_{i} \cdot \arcLabel_a \},  &\textnormal{$\forall u \in \nodes_i, \; \forall i \in \{1,\dots,n\}$,}
\end{align*}
provides the longest-path value $\costBottom(\bpi, u)$ from $u$ to $\terminalnode$. The values $\costTop(\bpi, u)$ can be calculated via a top-down pass on $\bdd$, i.e., starting from $\rootnode$ and considering one layer $\nodes_2, \dots, \nodes_{n+1}$ at a time. Analogously, the values 
$\costBottom(\bpi, u)$ are obtained via a bottom-up pass on $\bdd$, i.e., starting from $\terminalnode$ and considering one layer $\nodes_{n}, \nodes_{n-1}, \dots, \nodes_1$ at a time. For any arc $a=(\source(a),\target(a))$ such that $\source(a) \in \nodes_i$, its length is given by
$ 
	\arclength_a = \costTop(\bpi, \source(a)) + \costBottom(\bpi, \target(a)) + \pi_{i}\cdot\arcLabel_a.
$ 
Since each arc is traversed twice via the top-down and bottom-up passes, the complexity of the procedure is $\mathcal{O}(|A|)$.

{\changes
We remark that the algorithm above is similar in spirit to the lifting procedures for cover inequalities based on dynamic programming \cite{zemel1989easily,gu1998liftedCom,wolsey1999integer}, in particular also solving a recursive model to lift coefficients. The existing techniques, however, are applicable only for the knapsack polytope, solve a new dynamic program for each inequality to be lifted, and specialize on cover inequalities. In contrast, our approach is suitable for any set $X$ and valid linear inequality, and utilizes the same BDD to lift any given inequality (i.e., the BDD needs to be constructed only once).
}

\subsection{Sequential Lifting and Dimension Implications}
\label{sec:sequential}

The lifting procedure detailed in Theorem \ref{theo:LiftingDim} can be applied sequentially to strengthen an inequality. 
Specifically, we start with $\langle \bpi, \pi_0 \rangle$ satisfying our main assumptions \ref{cond:support}-\ref{cond:dimension}. Next, we calculate the 
d-slacks, choose $i \in \varIndex$ such that $\ds_i(\bpi) \neq 0$, and apply Theorem \ref{theo:LiftingDim} to obtain the tuple $\langle \bpi', \pi'_0 \rangle$ defining the lifted inequality. We re-apply this operation with the new $\langle \bpi', \pi'_0 \rangle$, and repeat until all d-slacks are equal to zero. The procedure stops in a finite number of iterations since the face dimension increases after each rotation; see property \ref{cond:lifting} of Theorem \ref{theo:LiftingDim}. We summarize the procedure in Algorithm \ref{alg:Lifting}.

\begin{algorithm}[tb] 
	\caption{Sequential Combinatorial Lifting Procedure} \label{alg:Lifting}
	\begin{algorithmic}[1]
		\Procedure{CombinatorialLifting}{$\langle \bpi, \pi_0\rangle $, $\bdd$}
		\State Calculate the disjunctive slacks $\blambda(\bpi)$ using $\bdd$ as explained in \S \ref{sec:lift-bdd}
		\While{$\blambda(\bpi)\neq \bm{0}$ }
		\State Choose $i \in \varIndex$ such that $\ds_i(\bpi) \neq 0$ \label{alg:varChoice}
		\State Apply Theorem \ref{theo:LiftingDim} to calculate $\langle \bpi', \pi_0' \rangle$
		\State Set $\langle \bpi, \pi_0\rangle$ = $\langle \bpi', \pi_0'\rangle$
		\State Recalculate $\blambda(\bpi)$
		\EndWhile
		\State \textbf{return} $\langle \bpi, \pi_0\rangle $ 
		\EndProcedure
	\end{algorithmic}
\end{algorithm}

The choice of $i$ in step \ref{alg:varChoice} of Algorithm \ref{alg:Lifting} is critical to the dimension of the resulting face, 
as illustrated in Example \ref{exa:noFacet}. 

\begin{example}\label{exa:noFacet}
	Consider the set $X = \{x\in \B^3: 5x_1+2x_2+3x_3 \leq 6 \}$ and inequality $\bpi^\top\bx = x_1 + x_2 + x_3 \leq 2$ that supports $\conv(X)$. We have $\blambda(\bpi) = (1,-1,-1)^\top$ and the lifted inequality with respect to $\ds_1(\bpi) = 1$ is $\bpi'^\top\bx =2x_1 + x_2 + x_3 \leq 2$ and has $\blambda(\bpi') = \bm{0}$. The lifted inequality is not facet-defining since $\dim(\conv(X)) = 3$ and $\dim(\face(\bpi')) = 1$.
	
	If we instead lift $x_1 + x_2 + x_3 \leq 2$ with respect to $\ds_2(\bpi) =-1$ the lifted inequality is $\bpi'^\top\bx =x_1 + x_3 \leq 1$ and $\blambda(\bpi') = \bm{0}$. In this case, the lifted inequality is facet-defining since $\dim(\face(\bpi')) = 2$.
	 \hfill $\square$
\end{example}

In order to understand the impact of the index choice, we first show in Lemma \ref{lem:dimFace} a relationship between d-slacks and the dimension of the face. Specifically, the cardinality of $\slacks^0(\bpi)$ bounds $\dim(\face(\bpi))$. We later use this result to gauge when the sequential procedure leads to a facet-defining inequality.

\begin{lemma} \label{lem:dimFace}
	The dimension of a face $\face(\bpi)$ satisfies $\dim(\face(\bpi) ) \leq |\slacks^{0}(\bpi)|$.  Moreover, $|\slacks^{0}(\bpi)| = 0$ if $\dim(\face(\bpi) ) = 0$.
\end{lemma}

\begin{proof}
	For any $i \in \slacks^{-}(\bpi) \cup \slacks^{+}(\bpi)$, the value of $x_i$ is fixed at either 0 or 1 for all $\bx \in \face(\bpi)$ according to Lemma \ref{lem:properties_ds}. Thus, the dimension of $\dim(\face(\bpi))$ is bounded by $|\slacks^{0}(\bpi)|$, since at most $|\slacks^{0}(\bpi)| + 1$ affinely independent points can be obtained from $\face(\bpi)$. 
	Now, assume $\slacks^{0}(\bpi) \neq \emptyset$ and $\dim(\face(\bpi) ) \geq 0$. There exist $\bx, \bx' \in \face(\bpi)$ such that $x_i \neq x'_i$ for 
	$i \in \slacks^{0}(\bpi)$. These two points are affinely independent and therefore $\dim(\face(\bpi)) \ge 1$. Thus, 
	$\slacks^{0}(\bpi) = \emptyset$ if $\dim(\face(\bpi)) = 0$. \hfill $\blacksquare$
\end{proof}

Example \ref{exa:noFacet} depicts a case where $|\slacks^{0}(\bpi)|$ increases faster than the number of affinely independent points in 
$\face(\bpi)$. In view of Lemma \ref{lem:dimFace}, we would like to choose $i$ so that $|\slacks^{0}(\bpi)|$ increases at a slower rate, since each lifting operation increases $\dim(\face(\bpi))$ by at least one according to Theorem \ref{theo:LiftingDim}-\ref{cond:lifting}. We show in Theorem \ref{theo:LiftingOrder} that the slow increase of  $|\slacks^{0}(\bpi)|$ occurs when there exists a unique slack with minimum non-zero absolute value.

\begin{theorem}\label{theo:LiftingOrder}
	Suppose there exists $i \not \in \slacks^{0}(\bpi)$ such that $|\ds_i(\bpi)| < |\ds_{i'}(\bpi)|$ for all 
	$i' \not \in \slacks^{0}(\bpi)$ ($i'\neq  i$). Then, for $\langle \bpi', \pi_0' \rangle$ obtained when lifting $\langle \bpi, \pi_0 \rangle$ with respect to $\ds_i(\bpi)$, $\dim(\face(\bpi')) = \dim(\face(\bpi)) + 1$ and $|\slacks^0(\bpi')| = |\slacks^0(\bpi)| + 1$.
\end{theorem}

\begin{proof}
	From Lemma \ref{lem:properties_ds}, it suffices to show that, for any $\bx \in \face(\bpi')$ and $i' \not\in \slacks^{0}(\bpi)$ such that $i' \neq i$, we have:
	\begin{enumerate*}[label=\arabic*)]
	\item $i' \in \slacks^{+}(\bpi)$ implies that $x_{i'}=0$; and 
	\item $i' \in\slacks^{-}(\bpi)$ implies that $x_{i'}=1$.
	\end{enumerate*}
	In such cases, an index $i'$ that was originally in $\slacks^{-}(\bpi)$ or $\slacks^{+}(\bpi)$ will remain in its original partition $\slacks^{-}(\bpi')$ or $\slacks^{+}(\bpi')$ for the lifted $\bpi'$. The statement then follows due to Theorem \ref{theo:LiftingDim}-\ref{cond:lifting} and Lemma \ref{lem:dimFace}. 
	
	We will focus our attention to the case $\ds_i(\bpi) > 0$ (the others are analogous). For any $\bx \in \face(\bpi')$, we have by construction that 
	$\bpi^\top\bx= \pi_0' - \ds_i(\bpi)x_i \geq \pi_0 - \ds_i(\bpi).$
	Assume, for the purpose of a contradiction, that $x_{i'} =1$ and that $\ds_{i'}(\bpi) > 0$. Thus, $\dsone_{i'}(\bpi)\geq \bpi^\top\bx\geq  \pi_0 -\ds_i(\bpi)$. Moreover, $\dszero_{i'}(\bpi)= \pi_0$. This implies that 
	$\ds_{i'}(\bpi) = \dszero_{i'}(\bpi) - \dsone_{i'}(\bpi) \leq  \pi_0 -\pi_0 +\ds_i(\bpi) \leq \ds_i(\bpi)$ and hence $0 < \ds_{i'}(\bpi) \leq \ds_i(\bpi)$. This cannot hold since $|\ds_i(\bpi)| < |\ds_{i'}(\bpi)|$.
	
	Similarly, assume that  $\ds_{i'}(\bpi)<0$ and $x_{i'}=0$. Then, $\dszero_{i'}(\bpi)\geq \pi_0 -\ds_i(\bpi)$ and $\dsone_{i'}(\bpi)=  \pi_0$. 
	This implies that $\ds_{i'}(\bpi) = \dszero_{i'}(\bpi) - \dsone_{i'}(\bpi) \geq  \pi_0 - \ds_i(\bpi) - \pi_0 = - \ds_i(\bpi)$. Thus, $0 > \ds_{i'}(\bpi) \geq -\ds_i(\bpi)$. This contradicts $|\ds_i(\bpi)| < |\ds_{i'}(\bpi)|$.	
	\hfill $\blacksquare$
\end{proof}

Theorem \ref{theo:LiftingOrder} provides a simple choice rule based on picking $i$ with the minimum absolute d-slack. It also indicates when this rule will converge to a facet-defining inequality. We formalize it in Corollary \ref{cor:facetDefining} below, which can be derived as a direct consequence of Theorem \ref{theo:LiftingOrder}.
\begin{corollary}
	\label{cor:facetDefining}
	If $\dim(\face(\bpi)) = |\slacks^0(\bpi)|$, the sequential lifting procedure (Algorithm \ref{alg:Lifting}) with the minimum slack absolute rule produces a facet-defining inequality if, at each lifting iteration except the last, the chosen $i \in \varIndex$ is such that $|\ds_i(\bpi)| < |\ds_{i'}(\bpi)|$ for all 
	$i' \not \in \slacks^{0}(\bpi)$ ($i'\neq  i$).
\end{corollary}

Finally, we note that, in general, it may not be possible to achieve a facet-defining inequality. For example, all non-zero d-slacks can have the same absolute value and the cardinality of $|\slacks^0(\bpi)|$ might increase by more than one while the dimension of $\face(\bpi)$ does not  (see Example \ref{exa:noFacet}).

\subsection{Relationship with Lifting based on Disjunctive Programming}\label{sec:dp_relationship}

We now formalize the connection between our lifting methodology 
and the $n$-step lifting procedure by Perregaard and Balas \cite{perregaard2001generating} mentioned in \S\ref{sec:relatedWork}.
Assume that $X = \{A\bx \leq \bm{b}, \; \bx \in \B^n \}$ for a matrix $A$ and vector $\bm{b}$ of appropriate dimensions. We consider a relaxation of the form \ref{model:dpl} that includes one disjunctive term for each index $i \in \varIndex$ and removes all integrality constraints:
\begin{align*}	
\max_{\bx} \left\{\bm{c}^\top \bx: \; A\bx \leq \bm{b}, \; \bigvee_{i' \in \varIndex} (x_{i'} \leq 0) \vee (x_{i'} \geq 1), \; \bx \in [0,1]^n \right\}. \tag{DP-L} \label{model:dpl}
\end{align*}

Proposition \ref{prop:disjuctiveSlack} below shows that, for \ref{model:dpl}, the optimal rotation parameter in the $n$-step lifting is such that $\gamma^*= |\ds_i(\bpi)|$ when using the individual binary disjunctions as target inequalities. 

\begin{proposition} \label{prop:disjuctiveSlack}
	Suppose $\ds_i(\bpi) \neq 0$ for some $i \in \varIndex$. Then, $\gamma^* = |\ds_i(\bpi)|$  if we employ either $x_i \leq 0 $ or $x_i \geq 1$ as a target inequality in the $n$-step procedure. 
\end{proposition}

\begin{proof}
	Consider the case when $i \in \slacks^{+}(\bpi)$ and let $\bm{e}_i$ be the $i$-th column of an $n \times n$ identity matrix. Since all $\bx\in \face(\bpi)$ have $x_i=0$ ,  $\widetilde{\bpi}^\top\bx = \bm{e}_i^\top\bx =  x_i \leq 0$ is an invalid target inequality for $\conv(X)$ and satisfies $\widetilde{\bpi}^\top\bx= x_i = 0$ for all $\bx\in \face(\bpi)$. The system that defines $\gamma^*$ is therefore
	\begin{align}
		\gamma^*= \min_{\bx \in [0,1]^n, x_0 \ge 0}\big\{ \; & \pi_0x_0 - \bpi^\top\bx: \; A\bx \leq \bm{b}, \; - x_i =-1, \nonumber  \\
		& \bigvee_{i' \in \varIndex} ( x_{i'} \leq 0 )\vee (-x_{i'}+ x_0\leq 0) \label{disj:dis} \bigg\}.		
	\end{align}
	
	It follows from \eqref{disj:dis} and $x_i =1$ that $x_0\leq 1$. Without loss of generality, we consider $x_0 = 1$. The system reduces to:
	\begin{align*}
	\gamma^* & = \min_{\bx,x_0}\left\{\pi_0x_0 - \bpi^\top\bx : A\bx \leq \bm{b}, \; x_i = 1,\; \bx \in \B^n, \; x_0 = 1 \right\} \\
	& = \pi_0 - \max_{\bx}\{ \bpi^\top\bx : \bx\in X, \; x_i = 1 \} 
	 = \dszero_i(\bpi)  - \dsone_i(\bpi)  = \ds_i(\bpi).
	\end{align*}
	\noindent The second to last equality comes from $i \in \slacks^{+}(\bpi)$ and  $\dszero_i(\bpi) = \pi_0$. The proof for $i \in \slacks^{-}(\bpi)$ and target inequality $\widetilde{\bpi}^\top\bx = x_i \geq 1$ is analogous. \hfill $\blacksquare$
\end{proof}

Proposition \ref{prop:disjuctiveSlack} indicates when these techniques are equivalent. By taking $\widetilde{\bpi}^\top\bx = x_i \leq 0 $ as the target inequality, we obtain  $\gamma^*= \ds_i(\bpi) > 0$. The rotated inequality $(\bpi +\gamma^*\widetilde{\bpi})^\top\bx = (\bpi + \ds_i(\bpi)\bm{e}_i)^\top\bx \leq \pi_0$ is equivalent to the lifted inequality in Theorem \ref{theo:LiftingDim}. Similarly, using target inequality $\widetilde{\bpi}^\top\bx = - x_i \leq -1 $ would result in  $\gamma^*= -\ds_i(\bpi) > 0$. Then, the rotated and lifted inequalities are equivalent, i.e.,  $\bpi +\gamma^*\widetilde{\bpi} = \bpi + \ds_i(\bpi)\bm{e}_i $ and $\pi_0 + \gamma^*\widetilde{\pi}_0=  \pi_0 + \ds_i(\bpi)$.

While the techniques are equivalent in this restricted case, our approach is valid for any binary set $X$ and, thus, can handle models where a BDD (or BDD relaxation) is a more advantageous representation in comparison to a linear description of $X$ \cite{bergman2018nonlinear}. We also note that the BDD network structure allows us to efficiently compute the disjunctive terms in a combinatorial fashion.

\section{Combinatorial Cutting-Plane Algorithm}
\label{sec:bddcut}

While the BDD-based lifting procedure developed in \S\ref{sec:lifting} can enhance inequalities from any cutting-plane methodology, 
we now exploit similar concepts to derive new valid inequalities for $X$ based on the network structure of $\bdd$. In particular, we design inequalities that separate points from $X$ by only relying on the combinatorial structure encoded by $\bdd$. Thus, no other specific structure  (e.g., linearity, submodularity, or gradient information) is required.

We assume, as before, that we are given an exact BDD $\bdd$ for $X$. Our cutting-plane method is based on an alternative linear description 
of $\bdd$ as an extended capacitated flow problem. We present this formulation in \S\ref{sec:bddPolytope} and our BDD-based cut generation linear program in \S\ref{sec:bddcut-general}. For cases where solving such model is not computationally practical, in \S \ref{sec:bddcut-simple} we develop a weaker but more efficient combinatorial cutting-plane method based on a max-flow/min-cut problem over $\bdd$. Finally, we show in \S\ref{sec:bddcut-literature} the relationship between our approach and existing BDD cutting-plane techniques \cite{davarnia2020outer,tjandraatmadja2019target}.

\subsection{BDD Polytope}
\label{sec:bddPolytope}

Existing BDD-based cut generation procedures \cite{davarnia2020outer,lozano2018binary,tjandraatmadja2019target} rely on the network-flow formulation $\flow(\bdd)$ introduced by Behle \cite{behle2007binary}, described as follows:
\begin{subequations}
\begin{align}
\flow(\bdd):= \big\{  (\bx; \by) &\in [0,1]^n\times \R_+^{|\arcs|} :  \nonumber \\
& \sum_{a\in \outgoing(u) } y_a - \sum_{a\in \incoming(u)} y_a = 0, 		&& \; \forall u \in \nodes\setminus\{\rootnode, \terminalnode\}, \label{eq:flow_all} \\ 
& \sum_{a\in \outgoing(\rootnode) } y_a = \sum_{a\in \incoming(\terminalnode)} y_a = 1, \label{eq:flow_rt}\\
& \sum_{a \in \arcs: \source(a) \in \nodes_i, \arcLabel_a = 1} y_a = x_i,  	&& \; \forall i \in \varIndex \label{eq:flow_conv} \Big \}. 
\end{align}
\end{subequations}
	
Equalities \eqref{eq:flow_all} and \eqref{eq:flow_rt} are balance-of-flow constraints over $\bdd$. Constraint \eqref{eq:flow_conv} links the arcs of $\bdd$ with solutions $\bx$. 
In particular, the polytope $\flow(\bdd)$ projected over the $\bx$ variables is equivalent to the convex hull of all solutions represented by $\bdd$, i.e., $\projx(\flow(\bdd)) = \conv(\sol)$.

{\changes
One drawback of $\flow(\bdd)$ is that constraints \eqref{eq:flow_conv} only consider flow variables associated with arc labels equal to one (i.e., $\arcLabel_a = 1$). Thus, there is no constraint explicitly limiting the flow passing through zero-value arcs. CGLPs  based on $\flow(\bdd)$ are potentially unbounded, which has been a fundamental challenge in existing works \cite{tjandraatmadja2019target,davarnia2020outer}.}

We propose an alternative formulation of $\flow(\bdd)$ that addresses its main limitations and use the reformulation to define our cutting-plane algorithms. The new formulation, here denoted by $\flowjoint(\bdd)$, corresponds to a joint capacitated network-flow polytope. The new model maintains the flow conservation constraints, \eqref{eq:flow_all} and \eqref{eq:flow_rt}, and replaces \eqref{eq:flow_conv} with \eqref{eq:flowcap_1} and \eqref{eq:flowcap_0} below. Both inequalities enforce a common capacity for arcs in a layer with the same value. Proposition \ref{prop:flow_equal} shows that the two formulations are equivalent and, thus, $\projx(\flowjoint(\bdd)) = \conv(\sol)$.
\begin{subequations}
	\begin{align}
	\flowjoint(\bdd):= \big\{  (\bx; \by) &\in [0,1]^n\times \R_+^{|\arcs|} :  \eqref{eq:flow_all}-\eqref{eq:flow_rt},  \nonumber \\ 
	& \sum_{a \in \arcs: \source(a) \in \nodes_i, \arcLabel_a =1 } y_a \leq  x_i,  		&&\forall i \in \varIndex, \label{eq:flowcap_1} \\
	& \sum_{a \in \arcs: \source(a) \in \nodes_i, \arcLabel_a =0 } y_a \leq  1 - x_i,  	&&\forall i \in \varIndex \label{eq:flowcap_0} \; \Big\}. 
	\end{align}
\end{subequations}

{\changes
Cut generation methods based on capacitated network flows over BDDs have been previously studied in the context of two-stage stochastic programs \cite{lozano2018binary}. In our model, the proposed polytope $\flowjoint(\bdd)$ is specially structured given, e.g., the use a single variable $x_i$ per layer $i$ to limit the capacity of the zero and one-value arcs. This structural property give us desired properties for separation in \S\ref{sec:bddcut-general}, and is key to when developing combinatorial cuts that do not depend on linear programs in \S\ref{sec:bddcut-simple}.
}

\begin{proposition}\label{prop:flow_equal}
	$\flowjoint(\bdd) = \flow(\bdd)$.
\end{proposition}

\begin{proof}
	Consider $(\bx';\by')\in \flow(\bdd)$, so $(\bx';\by')$ satisfies \eqref{eq:flow_all} and \eqref{eq:flow_rt}. Since \eqref{eq:flow_conv} holds, $(\bx';\by')$ also satisfies \eqref{eq:flowcap_1}. From the flow conservation constraints, \eqref{eq:flow_all} and \eqref{eq:flow_rt}, the flow traversing each layer $i \in \varIndex$ in $\bdd$ is exactly one, i.e., 
	\begin{align*}
	\sum_{a \in \arcs \colon \source(a) \in \nodes_i} y'_a = 1 \quad 
	&\Rightarrow \quad \sum_{a \in \arcs: \source(a) \in \nodes_i \colon \arcLabel_a = 1} y'_a + \sum_{a \in \arcs: \source(a) \in \nodes_i \colon \arcLabel_a = 0} y'_a  = 1 \\
	&\Rightarrow \quad \sum_{a \in \arcs: \source(a) \in \nodes_i \colon \arcLabel_a = 0} y'_a  = 1 - x_i'.
	\end{align*}
	\noindent Then, $ (\bx';\by')$ satisfies \eqref{eq:flowcap_0} and $(\bx';\by')\in \flowjoint(\bdd)$.
	Consider now $(\bx';\by') \in \flowjoint(\bdd)$. Since flows traversing a layer sum to one, constraints \eqref{eq:flowcap_1} and \eqref{eq:flowcap_0} are satisfied as equalities and therefore \eqref{eq:flow_conv} holds for $(\bx';\by')$.  
	\hfill $\blacksquare$
\end{proof}

\subsection{General BDD Flow Cuts}\label{sec:bddcut-general}

Our cutting-plane procedure formulates a max-flow optimization problem over $\flowjoint(\bdd)$ to identify and separate points $\bx' \not \in \conv(\sol)$, given by \eqref{eq:maxflow_general} below:
\begin{equation}	
z(\bdd; \bx'):= \max_{\by \in  \R_+^{|\arcs|} }\left\{ \sum_{a\in \outgoing(\rootnode) } y_a : \; \eqref{eq:flow_all},\eqref{eq:flowcap_1}-\eqref{eq:flowcap_0}, \bx = \bx' \right\}. \label{eq:maxflow_general}
\end{equation}
\noindent This model omits constraint \eqref{eq:flow_rt} which enforces the flow to be equal to one. We argue in Lemma \ref{lem:separation_bdd} that $z(\bdd; \bx')=1$ is a necessary and sufficient condition to check if $\bx'$ belongs to $\conv(\sol)$.

\begin{lemma}\label{lem:separation_bdd}
	$\bx'\in \conv(\sol)$ if and only if $z(\bdd; \bx')=1$.
\end{lemma}
\begin{proof}
	Constraints \eqref{eq:flowcap_1} and \eqref{eq:flowcap_0} enforce that the flow in each layer $i$ is at most $x'_i + 1-x'_i = 1$. Thus, 
	$z(\bdd; \bx') \leq 1$. Consider $\bx' \in \conv(\sol)$. From Proposition \ref{prop:flow_equal}, there exists $\by' \in \R_+^{|\arcs|}$ such that 
	$\sum_{ a \in \outgoing(\rootnode) } y'_a = 1$ and therefore $z(\bdd; \bx')=1$. For the converse, suppose $z(\bdd; \bx')=1$. It follows that there exists  $\by' \in \R_+^{|\arcs|}$  such that $(\bx';\by') \in  \flowjoint(\bdd) = \flow(\bdd)$, so $\bx'\in \conv(\sol)$.  \hfill $\blacksquare$
\end{proof}

Our BDD-based CGLP uses the dual of \eqref{eq:maxflow_general} to separate $\bx'\not\in \conv(\sol)$. Consider $\bomega\in \R^{|\nodes|}$ and $ \bnu,\bbeta\in \R_+^n$ as the dual variables associated with constraints \eqref{eq:flow_all}, \eqref{eq:flowcap_1}, and \eqref{eq:flowcap_0}, respectively. The resulting model is
\begin{subequations}
\begin{align}
\min_{\bomega,\bnu,\bbeta} \quad& 
\sum_{i \in \varIndex} x'_i\nu_i + \sum_{i \in \varIndex}(1 - x'_i)\eta_i \nonumber \tag{BDD-CGLP} \label{model:bddcglp} \\
\textnormal{s.t.} 							 \quad& \omega_{\target(a)} - \omega_{\source(a)} + \arcLabel_a \nu_i +(1-\arcLabel_a)\eta_i \geq 0, && \forall i \in \varIndex, a \in \arcs, \source(a) \in \nodes_i, \label{eq:flow_dual1} \\
& \omega_{\target(a)} +\arcLabel_a \nu_1 +(1-\arcLabel_a)\eta_1 \geq 1, &&\forall a \in \outgoing(\rootnode), \label{eq:flow_dual2}\\
& -\omega_{\source(a)} +\arcLabel_a\nu_n +(1-\arcLabel_a)\eta_n \geq 0, &&\forall a \in \incoming(\terminalnode), \label{eq:flow_dual3}\\
& \bomega\in \R^{|\nodes|}, \; \bnu,\bbeta \in \R_+^n.\label{eq:flow_dual4}
\end{align}
\end{subequations}

Let $w(\bdd; \bx')$ be the optimal solution value of \ref{model:bddcglp}. Strong duality and Lemma \ref{lem:separation_bdd} imply that we can identify if a point $\bx'$ belongs to $\conv(\sol)$ if $w(\bdd;\bx') =1$. Furthermore, we can use the optimal solution $(\bnu^*; \bbeta^*)$ to create a valid cut when $w(\bdd;\bx') < 1$. Specifically, the cut is given by 
\begin{align}
\label{eq:sep_general}
\sum_{i \in \varIndex} x_i \nu^*_i + \sum_{i \in \varIndex}(1 - x_i) \eta^*_i \geq 1.
\end{align}

Theorem \ref{theo:bddseparation} shows that the set of all cuts of the form \eqref{eq:sep_general} describes $\conv(\sol)$.

\begin{theorem}\label{theo:bddseparation}
	Let $\Lambda(\bdd)$ be the set of extreme points of the \ref{model:bddcglp} polyhedron defined by \eqref{eq:flow_dual1}-\eqref{eq:flow_dual4}. Furthermore, let $\cutpoly_\bdd$ be the set of points $\bx \in [0,1]^n$ that satisfy \eqref{eq:sep_general} for all $(\bnu; \bbeta) \in \proj_{\bnu, \bbeta}(\Lambda(\bdd))$. Then, $\conv(\sol) = \cutpoly_\bdd$.
\end{theorem}

\begin{proof}
	Consider a point $\bx' \in \conv(\sol)$. Lemma \ref{lem:separation_bdd} guarantees that  $z(\bdd; \bx') = w(\bdd; \bx') = 1$, so constraint \eqref{eq:sep_general} holds for any extreme point of \ref{model:bddcglp}. Now consider a point $\bx' \in \cutpoly_\bdd$. Since $\bx'$ satisfies \eqref{eq:sep_general} for all extreme points in $\Lambda(\bdd)$, we have that $w(\bdd; \bx') \geq 1$ and, thus, $w(\bdd; \bx') = z(\bdd; \bx') =1$. Finally, using Lemma \ref{lem:separation_bdd}, we have that $\bx' \in \conv(\sol)$. \hfill $\blacksquare$
\end{proof}

Thus, we solve \ref{model:bddcglp} to separate points $\bx'\notin \conv(\sol)$. The procedure returns a cut  \eqref{eq:sep_general} where $(\bnu^*; \bbeta^*)$ is the optimal solution of \ref{model:bddcglp}. 

\subsection{Combinatorial BDD Flow Cuts} 
\label{sec:bddcut-simple}

The above cutting-plane procedure requires solving a linear program with $|\arcs|$ constraints and $|\nodes| + 2n$ variables. Obtaining  $w(\bdd;\bx')$, thus, could be computationally expensive for instances where $\bdd$ is large (see \S\ref{sec:experiments}).
We propose an alternative cut-generation procedure based on \ref{model:bddcglp} that involves a combinatorial and more efficient max-flow solution over $\bdd$.

First, we consider a reformulation of $\flowjoint(\bdd)$ where the joint capacity constraints are replaced by individual constraints for each arc, i.e., a standard capacitated network flow polytope over $\bdd$:
\begin{subequations}
\begin{align}
\flowcap(\bdd):= \{ &(\bx; \by) \in [0,1]^n\times \R_+^{|\arcs|} : \eqref{eq:flow_all}-\eqref{eq:flow_rt},  \nonumber\\
& y_a \leq x_i,  \qquad\qquad\quad\;\;\; \forall a \in \arcs, \source(a) \in \nodes_i, \arcLabel_a=1, i \in \varIndex, \label{eq:rflow_cap1} \\
& y_a \leq 1- x_i,  \qquad\qquad\; \forall a\in \arcs, \source(a) \in \nodes_i, \arcLabel_a=0,  i \in \varIndex \label{eq:rflow_cap0} \}. 
\end{align}
\end{subequations}

\begin{proposition}  
	\label{prop:relaxedflow} 
	$\flowjoint(\bdd) \subseteq \flowcap(\bdd)$. Moreover, for any  integer $\bx \notin \conv(\sol)$, we have that $\bx\notin \projx(\flowcap(\bdd))$.
\end{proposition}

\begin{proof}
	Consider $\bx' \in \flowjoint(\bdd)$. By construction, $\bx'$ satisfies \eqref{eq:flow_all}-\eqref{eq:flow_rt}. Since $\bx'$ satisfies \eqref{eq:flowcap_1} and \eqref{eq:flowcap_0}, it follows that $\bx'$ holds for \eqref{eq:rflow_cap1} and \eqref{eq:rflow_cap0}. 
	
	Now take an integer point $\bx'\notin \conv(\sol)$ and  $\by'\in \R_+^{|\arcs|}$  such that  $(\bx'; \by')$ satisfies \eqref{eq:flow_all}, \eqref{eq:rflow_cap1}-\eqref{eq:rflow_cap0}. Notice that such a $\by'$ exists (e.g., $\by'=\bm{0}$). By construction, there is no path $\pathbdd \in \paths$ associated with $\bx'$. From constraints \eqref{eq:rflow_cap1}-\eqref{eq:rflow_cap0}, in any path $\pathbdd \in \paths$ there exists an arc $a\in \pathbdd$ with capacity zero (i.e., $y_a \leq 0$). We can then deduce that $\by'=\bm{0}$, therefore $(\bx';\by')$ violates \eqref{eq:flow_rt}. Finally, for any $\bx' \in\B^n \setminus\conv(\sol)$ there is no $\by' \in \R_+^{|\arcs|}$ such that $(\bx';\by')\in \flowcap(\bdd)$. \hfill $\blacksquare$
\end{proof}

Proposition \ref{prop:relaxedflow} shows that for any integer point $\bx'$,  $\bx'\not \in \sol$ implies $\bx' \not \in \projx(\flowcap(\bdd))$.
Example \ref{exa:mf_wrong} illustrates that, conversely, there might exist fractional points $\bx'\notin \conv(X)$ such that $\bx' \in \projx(\flowcap(\bdd))$,
and hence $\flowcap(\bdd)$ is a weaker representation.

\begin{example}\label{exa:mf_wrong}
	Consider our example $X=\{ \bx\in \B^4: 7x_1 + 5x_2 + 4x_3 +x_4 \leq 8 \}$, a fractional point $\bx'=(0.4,0.6,0.4,1)$, and the exact BDD $\bdd_1$ in Figure \ref{fig:bdd-example}. It is easy to see that $\bx' \notin \conv(X)$ since $7x'_1 + 5x'_2 + 4x'_3 +x'_4 = 8.4 \geq 8$. However, there exists a $\by' \in  \R_+^{|\arcs|}$ such that $(\bx',\by') \in \flowcap(\bdd)$ with value $y_{(\rootnode,u_1)} = 0.6$, $y_{(\rootnode,u_2)} = 0.4$, $y_{(u_1,u_4)} = 0.2$, $y_{(u_1,u_3)} = 0.4$,  $y_{(u_2,u_4)} = 0.4$, $y_{(u_3,u_4)} = 0.4$, $y_{(u_4,u_5)} = 0.6$, $y_{(u_5,\terminalnode)} = 1$, and all other arcs with flow equal to zero. \hfill $\square$
\end{example}


Similar to the general BDD flow cuts, we use the dual of the max-flow version of $\flowcap(\bdd)$ to identify points that do not belong to $\conv(\sol)$. Consider $\bomega\in \R^{|\nodes|}$ as the dual variables associated with constraints \eqref{eq:flow_all} and $\balpha$ the dual variables associated with constraints \eqref{eq:rflow_cap1}-\eqref{eq:rflow_cap0}. Then, the separation problem for the alternative BDD cuts is as follow:
\begin{subequations}
\begin{align*}
\min_{\bomega, \balpha} \quad& \sum_{i \in \varIndex} \left(  \sum_{\substack{a \in \arcs: \source(a) \in \nodes_i, \arcLabel_a=1}}  x'_i \alpha_a \;\; + \;\; \!\!\!\!\!\!\sum_{\substack{a \in \arcs: \source(a) \in \nodes_i, \arcLabel_a=0}}\!\!\!\!\!\! (1 - x'_i) \alpha_a \right) \nonumber \tag{CN-CGLP} \label{model:wbddcglp} \\
\textnormal{s.t.}		\quad& \omega_{\target(a)} - \omega_{\source(a)} + \alpha_a \geq 0, 
\qquad\qquad\forall i \in \varIndex, a \in \arcs, \source(a) \in \nodes_i, 
\\
\quad& \omega_{\target(a)} +\alpha_a \geq 1, 
\;\qquad\qquad\quad\quad\quad \forall a \in \outgoing(\rootnode), 
\\
\quad& -\omega_{\source(a)} +\alpha_a \geq 0, 
\qquad\qquad\quad\quad \forall a \in \incoming(\terminalnode), 
\\
\quad& \bm{\omega}\in \R^{|\nodes|}, \; \balpha \in \R_+^{|\arcs|}.
\end{align*}
\end{subequations}

Let $w^\mathsf{r}(\bdd;\bx')$ be the optimal solution value of \ref{model:wbddcglp}. Proposition \ref{prop:relaxedflow} implies that for any $\bx' \in \conv(\sol_\bdd)$, $w^\mathsf{r}(\bdd;\bx') = 1$. It follows that inequality \eqref{eq:sep_simple} holds for any $\bx \in \conv(\sol_\bdd)$, where $\balpha^*$ is optimal to \ref{model:wbddcglp}:
\begin{align}\label{eq:sep_simple}
\sum_{i \in \varIndex}\left(\sum_{\substack{a\in \arcs: \source(a) \in \nodes_i, \arcLabel_a=1}} x_i \alpha^*_a + \sum_{\substack{a\in \arcs: \source(a) \in \nodes_i, \arcLabel_a=0}} (1 - x_i)\alpha^*_a \right) \geq 1.
\end{align}

Of important note is that \ref{model:wbddcglp} is a classical min-cut problem, i.e., we are searching for a maximum-capacity arc cut in the network that certifies that a point does not belong to the convex hull of $\sol$. While the resulting inequalities are not as strong as the general BDD cuts from \ref{model:bddcglp}, we can leverage max-flow/min-cut combinatorial algorithms to solve it more efficiently in the size of the BDD. Several algorithms are readily available to that end \cite{Ahuja1993} and provide both primal and dual solutions to \ref{model:wbddcglp}.




Furthermore, another consequence of the design of such cuts is that their strength depends on the BDD size. That is, two BDDs $\bdd$ and $\bdd'$ encoding the same set might generate different combinatorial flow cuts because of distinct min-cut solutions. We show in Theorem \ref{theo:reduced_bdd} that the reduced BDD, which is unique, generates the tightest $\flowcap(\bdd)$ formulation and is hence critical in such a formulation. We note that a reduced BDD can be generated in polynomial time in $\bdd'$ for any $\bdd'$ representing the desired solution set \cite{bryant1986graph}.

\begin{theorem}
	\label{theo:reduced_bdd}
	Let $\bddr=(\nodesr, \arcsr)$ be the reduced version of $\bdd$, i.e., $\sol_{\bddr} = \sol_\bdd$, and for each layer $i \in \varIndex$, 
	$|\nodesr_i| \le |\nodes_i|$. Then, $\flowcap(\bddr) \subseteq \flowcap(\bdd)$.
\end{theorem}

\begin{proof}
	Consider $\pathsr$ to be the set of \rt\ paths in $\bddr$.
	First, $\sol_{\bddr} = \sol_\bdd$ implies that, for any \rt\ path $p \in \paths$, there exists a unique \rt\ path $p'\in \pathsr$ such that $\pathvar^p = \pathvar^{p'}$. Thus, 
	we will consider that the set of paths in both BDDs are equivalent, i.e., $\pathsr=\paths$.
	
	Let $\arcs_i = \{ a \in \arcs \colon \source(a) \in \nodes_i \}$ and $\arcsr_i = \{ a \in \arcsr \colon \source(a) \in \nodesr_i \}$.
	Since $\bddr$ is unique, there exists a unique surjective function $f_i: \arcs_i \rightarrow \arcsr_i$ that maps arcs from $\bdd$ to $\bddr$ for each layer $i \in \varIndex$.  Thus, for every arc $a\in \arcsr_i$,  let us define the pre-image of $f_i$ as  $f^{-1}_i(a):=\{a' \in \arcs_i:\; f(a') = a \}$, i.e., the subset of arcs in $\arcs_i$ that map to arc $a\in \arcsr_i$. 
	Next, denote by $\Gamma(\bdd; a):=\{ \pathbdd \in \bdd: a\in \pathbdd \}$ the set of paths in a BDD that traverse an arc $a \in \arcs$. From the construction procedure of $\bddr$ given $\bdd$ \cite{bryant1986graph}, $\Gamma(\bddr; a) =\bigcup_{a' \in f^{-1}_i(a)} \Gamma(\bdd; a')$ for all $a\in \arcsr_i$, i.e., the set of \rt\ paths passing through $a$ is equivalent to the set of \rt\ paths passing through all the arcs in $f^{-1}_i(a)$.
	
	Now consider the path formulation of $\flowcap(\bdd)$, $\flowcap^P(\bdd)$. It suffices to show that $\flowcap^P(\bddr)\subseteq\flowcap^P(\bdd)$. Since the paths in $\bdd$ and $\bddr$ are equivalent, we will consider equivalent variables $\bw$ for $\flowcap^P(\bddr)$ and $\flowcap^P(\bdd)$.
	\begin{subequations}
		\begin{align}
		\flowcap^P(\bdd):= \{ &(\bx; \bw)  \in [0,1]^n\times \R_+^{|\sol(\bdd) |} : &  \nonumber\\
		& \sum_{\pathbdd \in \paths: a\in \pathbdd}w_\pathbdd \leq x_i,  & \forall a\in \arcs_i, \arcLabel_a=1, i \in \varIndex, \;\;\label{eq:path_cap1} \\
		& \sum_{\pathbdd \in \paths: a\in \pathbdd}w_\pathbdd  \leq 1- x_i,  & \forall a\in \arcs_i, \arcLabel_a=0,  i \in \varIndex \label{eq:path_cap0} \}. 
		\end{align}
	\end{subequations}
	
	Using the path equivalence  $\Gamma(\bddr; a) =\bigcup_{a' \in f^{-1}_i(a)} \Gamma(\bdd; a')$ for any $a\in \arcsr_i$ and $i \in \varIndex$, we have that
	\[  \sum_{\pathbdd \in \pathsr: a\in \pathbdd}w_\pathbdd  = \sum_{a'\in f^{-1}_i(a)\; } \sum_{\pathbdd \in \paths: a'\in \pathbdd}w_\pathbdd, \qquad \forall a \in \arcsr_i,\; i \in \varIndex. \]
	
	\noindent Thus, constraints \eqref{eq:path_cap1} and \eqref{eq:path_cap0} of $\flowcap^P(\bddr)$ are tighter since they restrict more paths than for the case of $\flowcap^P(\bdd)$. This implies that, for any $(\bx';\bw') \in \flowcap^P(\bddr)$, $(\bx';\bw') \in \flowcap^P(\bdd)$. 
	\hfill $\blacksquare$
\end{proof}

Theorem \ref{theo:reduced_bdd} indicates that the reduced BDD can separate more points than any other BDD representing the same solution set. We also note in passing that the variable ordering plays a role on the size of the BDD and, hence, on the effectiveness of the combinatorial BDD flow cuts. Investigating variable orderings for specific problem classes and how they impact the cuts (theoretically and computationally) may lead to new research avenues.

\subsection{Relationship with Existing BDD Cut Generation Procedures} 
\label{sec:bddcut-literature}


The two existing BDD-based CGLPs rely on dual reformulations of $\flow(\bdd)$, and, thus, also describe $\conv(\sol)$  \cite{tjandraatmadja2019target,davarnia2020outer}. These techniques rely on additional information: Tjandraatmadja et al. (2019) \cite{tjandraatmadja2019target} CGLP requires an interior point of $\sol$ and Davarnia et al. (2020) \cite{davarnia2020outer} must incorporate possibly non-linear normalization constraints. In contrast, \ref{model:bddcglp} exploits the structure of $\bdd$ directly to describe $\conv(\sol)$. We now detail these two BDD-based CGLPs and highlight the main theoretical differences to \ref{model:bddcglp}. 

Consider $\bomega\in \R^{|\nodes|}$ as the dual variables associated with constraints \eqref{eq:flow_all} and \eqref{eq:flow_rt}, and $\btheta \in \R^{n}$ as the dual variables associated with \eqref{eq:flow_conv}. The two BDD-based CGLP models employ flow inequalities of the form
\begin{equation}\label{eq:flow_dual}
\omega_{\target(a)} - \omega_{\source(a)} + \theta_i \arcLabel_a \geq 0,  \qquad \forall i \in \varIndex, a \in \arcs, \source(a) \in \nodes_i.
\end{equation}

\noindent Notice that \eqref{eq:flow_dual} resembles the flow inequalities \eqref{eq:flow_dual1}-\eqref{eq:flow_dual3} of \ref{model:bddcglp}. However, our flow constraints use two sets of positive dual variables for each BDD layer (i.e., $\bnu,\bbeta\in \R_+^n$) instead of the single unbounded set of variables $\btheta \in \R^{n}$. This difference emerges because \eqref{eq:flow_conv} only bounds the arc flow variables $\by\in \R_+^{|\arcs|}$ with value $v_a =1$, while our joint-capacity constraints \eqref{eq:flowcap_1}-\eqref{eq:flowcap_0} bound all variables $\by$. This is one reason, e.g., why \ref{model:bddcglp} does not require any normalization as in previous techniques.

Tjandraatmadja et al. (2019) \cite{tjandraatmadja2019target} propose a BDD-based CGLP \eqref{model:cglp_target} to generate target cuts that are facet-defining. Their CGLP yields a valid inequality that intersects the ray passing through an interior point $\bm{u} \in \conv(\sol)$ and the fractional point $\bx'\in [0,1]^n$ to be cut-off. The procedure returns a cut $\btheta^{*\top} \bx' \leq  1+ \btheta^{*\top} \bm{u}$ whenever the optimal value of \eqref{model:cglp_target} is greater than one. 
\begin{equation}\label{model:cglp_target}
\max_{\bomega, \btheta}\left\{ \btheta^\top( \bx' - \bm{u} ) : \; \eqref{eq:flow_dual}, \; \omega_\terminalnode = 0, \; \omega_\rootnode = 1+ \btheta^\top\bm{u}\right\}.
\end{equation}

Davarnia et al. (2020) \cite{davarnia2020outer} circumvent the need of an interior point by proposing a simpler but possibly non-linear BDD-based CGLP presented in \eqref{model:cglp_gradient}. The model checks if $\bx'$ can be represented as a linear combination of points in $\sol$,  i.e., whether there exists $\btheta,\bomega$ such that $\btheta^\top \bx' = \omega_\terminalnode$. Otherwise, their procedure returns a valid inequality $\btheta^{*\top} \bx \leq \omega_\terminalnode^*$,
which is not necessarily facet-defining. Since the model may be unbounded, the optimization problem \eqref{model:cglp_gradient} includes normalization constraints $\mathcal{C}(\bomega,\btheta)\leq 0$ which are potentially non-linear. The CGLP \eqref{model:cglp_gradient} is addressed by an iterative subgradient algorithm.
\begin{equation}\label{model:cglp_gradient}
\max_{\bomega, \btheta}\left\{ \btheta^\top \bx' - \omega_\terminalnode : \; \eqref{eq:flow_dual}, \; \omega_\rootnode = 0, \; \mathcal{C}(\bomega,\btheta)\leq 0 \right\}.
\end{equation}

Note that, in our approach, we either solve \ref{model:bddcglp} (a linear program) or a single max-flow/min-cut problem, both relying only on $\bdd$. We compare our cutting-plane approaches with these procedures in \S \ref{sec:experiments}.


\section{\changes Case Study: Second-order Cone Programming}
\label{sec:soc}

For our numerical evaluation, we apply our combinatorial cut-and-lift procedure to binary problems with SOC inequalities. Recall that problem \ref{eq:soc_model} considers a linear objective and $m$ constraints of the form
\begin{align}\label{eq:soc}
\ba^\top\bx + || D^\top \bx - \bm{h} ||_2 \leq b \quad \Leftrightarrow \quad \ba^\top\bx + \sqrt{ \sum_{k \in \{1,...,l\}}(\bd_k^\top\bx - h_k)^2  } \leq b,
\end{align}
\noindent where $\ba, \bm{h} \in \R^n$ are real vectors, $D \in \R^{l \times n}$ is a matrix with $l$ rows of dimension $n$, and $b \in \R$ is the right-hand-side constant.

We propose a novel BDD encoding for the general SOC inequalities \eqref{eq:soc} using a  
recursive reformulation \cite{bergman2018nonlinear}. Our formulation considers $l+1$ sets of state variables, $\bQ_0, \bQ_1,..., \bQ_l$, where each set of variables has $n+1$ stages, i.e.,  $\bQ_k \in \R^{n+1}$ for each $k \in \{0,1,...,l\}$. State variables $\bQ_0$ represent the value of the linear term (i.e., $\ba^\top\bx$), while $\bQ_k$ encodes the $k$-th linear expression in the quadratic term (i.e., $\bd_k^\top\bx - h_k$). The recursive model for \eqref{eq:soc} is given by
\begin{subequations}
	\begin{align}
		\mbox{RSOC}:=\big\{&(\bx;\bm{Q})\in \B^n\times \R^{(l+1)\times(n+1)}: \nonumber  \\
		&Q_{0,0} = 0, \; Q_{k,0} = h_k, & \forall k \in \{1,...,l\}, \label{eq:rec_ini} \\
		&Q_{0,i} = Q_{0,i-1} + a_{i} x_i,  & \forall i \in \varIndex,  \label{eq:rec_mu}\\
		&Q_{k,i} = Q_{k,i-1} + d_{ki} x_i,  & \forall  i \in \varIndex,\; k \in \{1,...,l\}, \label{eq:rec_sigma} \\
		&Q_{0,n} + \sqrt{\sum_{k\in \{1,...,l\}} \left(Q_{k,n}\right)^2 } \leq b  &\bigg\}.	\label{eq:rec_b}
	\end{align}
\end{subequations}

\noindent The first set of equalities \eqref{eq:rec_ini} initialize  the state variables at stage 0. Equalities \eqref{eq:rec_mu} and \eqref{eq:rec_sigma} correspond to the recursive formulas for each linear expression, and constraint \eqref{eq:rec_b} enforces the SOC inequality. Notice that $\projx(RSOC)$ is equivalent to the feasible set of the SOC inequality \eqref{eq:soc}. 

An exact BDD $\bdd=(\nodes,\arcs)$ is the reduced state-transition graph of the dynamic programs above, where each BDD node maps to a state variable and each arc represents a  state transition. If we consider, for example, the RSOC model, the root node $\rootnode$ stores the stage-0 values, and each node $u$ in  layer $i\in I$ corresponds to a reachable state from the $(i-1)$-th stage. The recursions \eqref{eq:rec_mu} and \eqref{eq:rec_sigma} are used to compute the transitions between nodes. 

The proposed recursive model (and thereby the BDD) can be used for any type of SOC inequality. For our numerical evaluation, we consider two classes of SOC inequalities commonly found in the literature. The first is defined by SOC knapsack inequalities \cite{atamturk2009submodular,joung2017lifting}, i.e., where $D\in \R^{n\times n}$ is a diagonal matrix and the SOC constraint is given by
\begin{equation}\label{eq:soc_k}
	\ba^\top \bx + \Omega\sqrt{ \sum_{i\in \varIndex}d_{ii}^2 x_i  } \leq b.
\end{equation}

We develop a simpler recursive model for \eqref{eq:soc_k} with only two sets of state variables, $\bQ_0$ and $\bQ_1$. As before, $\bQ_0$ represents the linear term and $\bQ_1$ encodes the linear term inside the square root. Thus, the recursive model is given by
\begin{align*}
\mbox{RSOC-K}:=\big\{&(\bx;\bm{Q})\in \B^n\times \R^{2\times(n+1)}: \; \eqref{eq:rec_ini}, \eqref{eq:rec_mu},\\
	&Q_{1,i} = Q_{1,i-1} +d_{ii}^2 x_i,  \;\; \forall  i \in \varIndex, \quad 
	Q_{0,n} + \Omega\sqrt{\sum_{i \in \varIndex} Q_{1,n} } \leq b  &\bigg\}.
\end{align*}

For our second class, we consider SOC inequalities derived from chance constraints of the form $\prob(\bm{\xi}^\top \bx \leq b) \geq \epsilon$ where $\bm{\xi}$ is a random variable with normal distribution $N(\ba, D)$ and $\epsilon \in [0.5 ,1]$ \cite{van1963minimum,lobo1998applications}. Specifically, the constraint can be reformulated as
 \begin{equation} \label{eq:cc_general}
 	\ba^\top \bx + \mathrm{\Phi}^{-1}(\epsilon) || D\bx ||_2 \leq b \quad \Leftrightarrow \quad \ba^\top \bx + \Omega\sqrt{ \sum_{k \in \{1,...,n\}}(\bd_k^\top\bx)^2  } \leq b,
 \end{equation} 
where we use $\Omega =\mathrm{\Phi}^{-1}(\epsilon)$ for simplicity. Notice that \eqref{eq:cc_general} is a special case of $\eqref{eq:soc}$ where $D$ is square matrix and $\bm{h} = \bm{0}$. 



An exact BDD for both problem classes can grow exponentially large. Therefore, we construct relaxed BDDs using a standard incremental refinement procedure \cite{bergman2016discrete}. For completeness, we provide a detailed explanation of our BDD construction procedure in Appendix \ref{appendix:soc-bdd}.

\section{\changes Empirical Evaluation and Discussion} \label{sec:experiments}

This section presents an empirical evaluation of our combinatorial cut-and-lift procedure for \ref{eq:soc_model} (see \S \ref{sec:soc}). We create a BDD for each of the $m$ SOC inequalities and apply our procedure for each such constraint at the root node of the branch-and-bound tree. For any fractional point $\bx\in [0,1]^n$,  we iterate over each BDD until one of them generates a cut, as we describe in detail below. We then lift the inequality using Algorithm \ref{alg:Lifting}. The procedure ends when $\bx$ cannot be cut-off by any BDD. 

\medskip
\noindent \textit{Datasets.} We test our approach on the SOC knapsack (SOC-K) benchmark \cite{atamturk2009submodular,joung2017lifting} composed of 90 instances with 
$n \in \{100,125,150\}$ variables and $m \in \{10, 20\}$ constraints. We also generate a random set of instances (SOC-CC) for the more general SOC inequalities derived from chance constraints (i.e., inequality \eqref{eq:cc_general}) following a similar procedure to the one used for SOC-K. We consider $n \in \{75,100,125\}$, $m \in \{10, 20\}$, $\Omega \in \{1,3,5\}$, and a density of $2/\sqrt{n}$ over all the constraints. Parameters $\ba_j$, $D_j$, and $\bm{c}$ are sampled from a discrete uniform distribution with $\ba_j\in [-50,50]^n$, $D_j \in [-20, 20]^{n\times n}$, and $\bm{c}\in [0,100]^n$. Parameters $b_j$  are given by
\[ b_j = t\cdot \left( \sum_{i \in \varIndex} a_{ji}^+  + \Omega\sqrt{ \sum_{i \in \varIndex} \max\left\{ \sum_{k \in \varIndex} d_{jik}^+, \sum_{k \in \varIndex} d_{jik}^- \right\}^2 } \right), \qquad \forall j \in \{1,...,m\}, \]

\noindent where $t \in \{0.1, 0.2, 0.3\}$ is the constraint tightness, $f^+:=\max\{0, f\}$, and $f^-:=\max\{0, -f\}$ for any $f\in \R$. Note that $b_j$ with $t=0.3$ will remove approximately 50\% of the possible assignments for $\bx \in \B^n$. Then, we generate 5 random instances for each parameter combination, 
i.e., 270 instances.

\medskip
\noindent \textit{Benchmarks.} We implement four basic variants of our approach to assess the BDD cuts and the lifting procedure. \bddFlow\ computes the low-complexity combinatorial flow cuts in \S\ref{sec:bddcut-simple}, while \bddGeneral\ also compute these cuts and, if the approach fails to produce any cut, it applies BDD flow cuts derived from our proposed CGPL in \S \ref{sec:bddcut-general}.
The two other variants, \bddGeneralLift\ and \bddFlowLift\, are the respective versions of these cutting algorithms augmented with the proposed BDD lifting from \S\ref{sec:lifting} for every new constraint generated.


Moreover, we implement the two most recent BDD-based cuts, namely target cuts (\bddTarget) \cite{tjandraatmadja2019target} and projected cuts (\bddProject)  \cite{davarnia2020outer}. 
As before, we use the suffix \lift to denote if we apply lifting. Lastly, we implement the cover cuts and lifting procedure by Atamt{\"u}rk and Narayanan (2009) \cite{atamturk2009submodular} for the SOC-K dataset, here denoted by \cover\ and \coverLift\ for the version without lifting and with their lifting. We also test their cover cuts in conjunction with our BDD lifting, \bddCover. Finally, we evaluate combinations of different cut classes, which we will define in each relevant section.

The procedures are implemented in \texttt{C++} in the IBM ILOG CPLEX 12.9 solver \cite{CPLEXManual} using the \texttt{UserCuts} callback at the root node of the search.\footnote{We will make the code available.}  All experiments consider a single thread, a one-hour time limit, and the linearization strategy (i.e., \texttt{MIQCPStrat = 2}) to solve the SOC problems.\footnote{Numerical testing suggested that this was the best strategy across all techniques.} We deactivate all solver cuts when running the BDD techniques (i.e., \bddFlow, \bddGeneral, \bddTarget, \bddProject) and the cover-cut variants (i.e., \cover, \coverLift, and \bddCover) to evaluate their effectiveness on their own. Notice that, given the \texttt{UserCuts} callback, our techniques omit the \texttt{Presolve} option and use \texttt{TraditionalSearch}. For a fair comparison, when running unaugmented \cplex, we use the same configuration except with the solver cuts activated .\footnote{Numerical testing suggested that \cplex\ performs better on the problem sets when \texttt{Presolve} is \emph{deactivated}.}

We experimented with three BDD widths $\maxwidth \in \{2000,3000,4000\}$, i.e., the maximum number of nodes per layer allowed in a relaxed BDD. We present results for the width with the best overall performance, $\maxwidth=4000$ (we include aggregated results for other widths in Appendix \ref{appendix:width}). Notice that most of the created BDDs are relaxed due to width limit, especially when  $n\geq 100$. 

\subsection{Effectiveness of BDD-based Lifting Procedure}

We first evaluate our lifting algorithm over the BDD-based cuts and other cutting-plane procedures. Table \ref{tab:root_info} shows the average root node information for our two datasets, i.e., optimality gap, time, number of cuts, percentage of inequalities lifted at least once, and time to solve the separation problem.  The first three columns are divided into two sub-columns, the first showing results without lifting (\nlift) and the second with lifting (\lift). Column \lift\ refers to our BDD-based lifting, except in row  \cover\ that shows the continuous lifting \cite{atamturk2009submodular}. Appendices \ref{appendix:allresutls_knapsack} and \ref{appendix:allresutls_general} present detailed results for each dataset.

\begin{table}[tbp]
	\setlength{\tabcolsep}{4pt}
	\footnotesize
	\centering
	\caption{{\changes Aggregated root information for the SOC-K and SOC-CC datasets.} }
	\begin{tabular}{ll|rr|rr|rr|r|r}
		\toprule
		& & \multicolumn{2}{c|}{Root Gap} & \multicolumn{2}{c|}{Root Time (s)} & \multicolumn{2}{c|}{\# Cuts} & \% Lift & Sep.(s) \\
		\midrule
		& & \nlift    & \lift     & \nlift    & \lift     & \nlift    & \lift     &       &  \\
		\midrule
		\multirow{7}[8]{*}{ \scriptsize{\textbf{SOC-K}} } &\cplex & 2.8\% & -     & 2.4   & -     & 123.5 & -     & -     & - \\
		&\cover     & 3.3\% & 2.8\% & \textbf{1.3} & 1.8   & 96.4  & 95.7  & 98.6\% & \textbf{0.001} \\
		&\bddCover   & -     & 2.7\% & -     & 11.6  & -     & 81.0  & 70.0\% & \textbf{0.001} \\
		\cmidrule{2-10}
		&\bddFlow    & 3.6\% & 2.8\% & 20.0  & 12.8  & 240.7 & 54.9  & 99.2\% & \textbf{0.001} \\
		&\bddGeneral    & \textbf{1.3\%} & \textbf{1.3\%} & 400.4 & 109.4 & 1087.2 & 191.1 & 80.7\% & 0.442 \\
		\cmidrule{2-10}
		& \bddProject    & 3.9\% & 3.76\% & 16.6  & 16.5  & 8.0   & 6.5   & 72.2\% & 0.020 \\
		&\bddTarget    & \textbf{1.3\%} & -     & 118.0 & -     & 108.1 & -     & -     & 0.329 \\
		\midrule
		\midrule
		\multirow{5}[6]{*}{\scriptsize{\textbf{SOC-CC}}} & \cplex & 20.4\% & -     & \textbf{9.1} & -     & 127.6 & -     & -     & - \\
		\cmidrule{2-10}
		&\bddFlow    & 19.5\% & 15.4\% & 17.1  & 16.1  & 31.1  & 23.3  & 90.7\% & \textbf{0.001} \\
		&\bddGeneral    & \textbf{13.0\%} & \textbf{13.0\%} & 144.7 & 85.0  & 305.7 & 124.5 & 72.6\% & 0.096 \\
		\cmidrule{2-10}
		& \bddProject    & 17.6\% & 16.3\% & 25.3  & 18.0  & 61.3  & 20.5  & 71.4\% & 0.017 \\
		&\bddTarget    & 13.3\% & -     & 71.4  & -     & 98.1  & -     & -     & 0.102 \\
		\bottomrule
	\end{tabular}%
	\label{tab:root_info}%
\end{table}%

Our lifting procedure significantly reduces the root gap and time for \bddFlow. We observe a similar behavior for \bddProject\ but with a smaller impact due to fewer added cuts. \bddGeneral\ also benefits from our lifting, reducing the number of cuts added and, as a consequence, the root node time. However, there is no root gap improvement for \bddGeneralLift\ since the general BDD cuts separate all infeasible fractional points from each BDD.  Lastly, our lifting has no impact on \bddTarget\ since these cuts are facet-defining and, thus, we omit this variant.

Our BDD lifting also has a positive impact when applied to cuts obtained from other techniques. For the cover cut \cover, our lifting achieves a smaller root gap than the continuous lifting (i.e., 2.67\% v.s. 2.81\%) and adds fewer valid inequalities. Lastly, our BDD lifting is faster than the continuous lifting (0.001 seconds vs. 0.006 seconds per individual cut). This is as expected given that the continuous cuts require solving a linear program.

\subsection{BDD-based Cutting Planes Comparison}

Table \ref{tab:root_info} also highlights the root performance differences between the BDD cuts. As expected, the complete methods based on a CGLP (i.e., \bddGeneral, \bddGeneralLift, and \bddTarget) achieve the lowest root gap. However, these techniques are computationally expensive since they solve LPs with as many variables as arcs in the BDD. In total time, \bddTarget\ is more efficient than \bddGeneral\ because it adds fewer cuts, possibly due to the fact that its inequalities are facet-defining. This difference is partially mitigated by our BDD lifting since \bddGeneralLift\ has similar average performance to \bddTarget.

The low-complexity BDD cuts (i.e., \bddFlow\ and \bddProject) are orders of magnitude faster but have a larger root gap than the CGLP-based alternatives. As discussed in \S\ref{sec:bddcut-simple}, \bddFlow \ and \bddFlowLift\ solve a min-cut problem over the BDD, which explains the fast separation time, but might not remove all infeasible fractional points. \bddProject\ is a complete algorithm but its subgradient routine struggles to separate points close to the convex hull \cite{davarnia2020outer}, which explains its poor root gaps. Also, \bddProject\ is 15 to 20 times slower than \bddFlow\ on average because it requires several subgradient iterations to derive an inequality. 

\begin{figure}[tb]
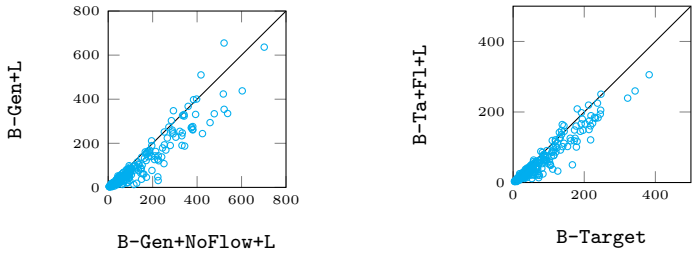

	\centering
	\scatter{Figures/plots/root_time_general.txt}{0}{800}{BGL}{BGWL}{\bddGeneralNoFlow}{\bddGeneralLift}
	\hspace{3em}
	\scatter{Figures/plots/root_time_general.txt}{0}{500}{BT}{BT-F}{\bddTarget}{\bddTargetFlow}
	\caption{\changes Root time influence (in seconds) of combinatorial BDD flow cuts for SOC-CC instances.  }\label{fig:time_root_cglp}
\end{figure}

Lastly, we evaluate the impact of combining the low-complexity BDD cuts with the CGLP variants; i.e., the CGLP is invoked after no more fast cuts can be added. Figure \ref{fig:time_root_cglp} depicts two plots where each $(x,y)$ point represents the root time for an instance given by the $x$-axis and the $y$-axis techniques. \bddTargetFlow\ employs target and combinatorial flow cuts together, while \bddGeneralNoFlow\ represents a pure general BDD cut approach. In both cases, all methods have approximately the same root gap. We observe that adding the low-complexity cuts improve the root time performance of the methods, as they decrease the number of calls to the corresponding CGLP.

\subsection{Solution Performance - Cuts at the Root Node}

We now present in Table \ref{tab:overall} the average solution performance of the tested techniques when cuts are added only at the root node. The first two columns present the number of instances solved to optimality and the average final gap. The next columns show the average solution time and nodes explored for only the instances that all techniques solved. 

We observe that \bddGeneralLift, \bddTarget, and \bddTargetFlow\ are the best performing techniques, the latter two solving all instances in SOC-K. These algorithms have small performance differences that are explained by the root information in Table \ref{tab:root_info}, where \bddTarget\ is slightly more  efficient than \bddGeneralLift. In terms of time, \bddTargetFlow\ is the fastest approach and explores the fewest number of nodes. We note, however, that \bddTarget\ is effective for SOC-CC and solved more instances in that benchmark, albeit more slowly than \bddTargetFlow.

\begin{table}[tbp]
	\footnotesize
	\setlength{\tabcolsep}{4pt}
	\centering
	\caption{ {\changes Aggregated results showing the overall performance of each technique.}}
	\begin{tabular}{l|rrrr||rrrr}
		\toprule
		& \multicolumn{4}{c||}{\textbf{SOC-K  Dataset}} & \multicolumn{4}{c}{\textbf{SOC-CC  Dataset}} \\
		\midrule
		& \# Solv & Gap & Time & \# Nodes & \# Solv & Gap & Time & \# Nodes \\
		\midrule
		\cplex & 70    & 0.39\% & 167.8 &      155,856.4  & 137   & 7.46\% & 530.2 &    169,802.5  \\
		\cover     & 71    & 0.41\% & 209.9 &      446,171.5  & -     & -     & -     & - \\
		\coverLift    & 70    & 0.34\% & 149.3 &      303,020.2  & -     & -     & -     & - \\
		\bddCover   & 70    & 0.27\% & 109.6 &      191,505.4  & -     & -     & -     & - \\
		\midrule
		\bddFlow    & 64    & 0.52\% & 458.2 &      867,798.4  & 135 &	7.01\% &	367.4 &	 214,629.4  \\
		\bddFlowLift   & 73    & 0.27\% & 139.6 &      270,937.6  & 149 &	5.55\%	& 240.3	& 121,800.3  \\
		\bddGeneral   & 76    & 0.16\% & 441.1 &        38,879.7  & 162   & 4.31\% & 172.4 &      56,490.3  \\
		\bddGeneralLift  & 88    & \textbf{0.01\%} & 126.9 &        32,289.6  & 167   & 4.18\% & 142.0 &      54,563.6  \\
		\midrule
		\bddProject    & 68    & 0.47\% & 422.7 &   1,019,999.4  & 136   & 6.63\% & 380.0 &    208,567.2  \\
		\bddProjectLift   & 66    & 0.50\% & 335.1 &      846,437.0  & 139   & 6.34\% & 366.5 &    202,853.6  \\
		\bddTarget    & \textbf{90} & \textbf{0.01\%} & 131.2 &        28,031.5  & \textbf{172} & \textbf{4.04\%} & 132.9 &      51,319.9  \\
		\bddTargetFlow  & \textbf{90} & \textbf{0.01\%} & \textbf{83.5} & \textbf{       24,685.9} & 168   & 4.19\% & \textbf{110.0} & \textbf{     45,979.1} \\
		\bottomrule
	\end{tabular}%
	\label{tab:overall}%
\end{table}%

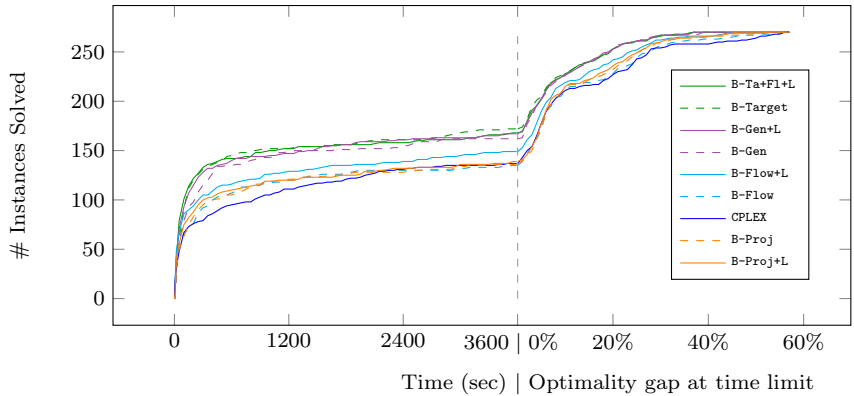
\begin{figure}[b]
	\centering
	\begin{tikzpicture}
		\begin{axis}[ legend style={at={(0.85,0.8)},anchor=north, font=\tiny, legend cell align=left},
			width=0.95\textwidth,
			height=0.49\textwidth,
			xlabel={\hspace{26.5ex} Time (sec) $|$ Optimality gap at time limit }, 
			ylabel={ \# Instances Solved},
			ytick={0,  50, 100, 150, 200, 250},
			yticklabels={0,  50, 100, 150, 200, 250},
			xtick={0,  1200, 2400, 3600, 4600, 5600, 6600},
			xticklabels={0,  1200, 2400,  $3600\; |\; 0\%\;\;$, 20\%, 40\%, 60\%},
			]
			
			\addplot[color=green!60!black] table[x=coord, y=BT-F] {Figures/plots/profile_time_soc.txt};
			\addlegendentry{\bddTargetFlow}
			\addplot[color=green!60!black, dashed] table[x=coord, y=BT] {Figures/plots/profile_time_soc.txt};
			\addlegendentry{\bddTarget}
			\addplot[color=violet!70!white] table[x=coord, y=BGWL] {Figures/plots/profile_time_soc.txt};
			\addlegendentry{\bddGeneralLift}
			\addplot[color=violet!70!white, dashed] table[x=coord, y=BGW] {Figures/plots/profile_time_soc.txt};
			\addlegendentry{\bddGeneral}
			\addplot[color= cyan] table[x=coord, y=BFL] {Figures/plots/profile_time_soc.txt};
			\addlegendentry{\bddFlowLift}
			\addplot[color= cyan, dashed] table[x=coord, y=BF] {Figures/plots/profile_time_soc.txt};
			\addlegendentry{\bddFlow}
			\addplot[color= blue] table[x=coord, y=CPLEX] {Figures/plots/profile_time_soc.txt};
			\addlegendentry{\cplex}
			\addplot[color= orange, dashed] table[x=coord, y=BP] {Figures/plots/profile_time_soc.txt};
			\addlegendentry{\bddProject}
			\addplot[color= orange] table[x=coord, y=BPL] {Figures/plots/profile_time_soc.txt};
			\addlegendentry{\bddProjectLift}
			
			\addplot[gray!80!white, dashed] coordinates {(3600,0) (3600,270)};
			
			\addplot[color=green!60!black, dashed] table[x=coord, y=BT] {Figures/plots/profile_gap_soc.txt};
			\addplot[color=green!60!black] table[x=coord, y=BT-F] {Figures/plots/profile_gap_soc.txt};
			\addplot[color=violet!70!white] table[x=coord, y=BGWL] {Figures/plots/profile_gap_soc.txt};
			\addplot[color=violet!70!white, dashed] table[x=coord, y=BGW] {Figures/plots/profile_gap_soc.txt};
			\addplot[color= cyan] table[x=coord, y=BFL] {Figures/plots/profile_gap_soc.txt};
			\addplot[color= cyan, dashed] table[x=coord, y=BF] {Figures/plots/profile_gap_soc.txt};
			\addplot[color= blue] table[x=coord, y=CPLEX] {Figures/plots/profile_gap_soc.txt};
			\addplot[color= orange, dashed] table[x=coord, y=BP] {Figures/plots/profile_gap_soc.txt};
			\addplot[color= orange] table[x=coord, y=BPL] {Figures/plots/profile_gap_soc.txt};
		\end{axis}
	\end{tikzpicture}
	\caption{Profile plot comparing the cumulative number of instances solved over time (left), and the cumulative number of instances over a final gap range (right) for SOC-CC dataset.  }\label{fig:profile_general}
\end{figure}

Figure \ref{fig:profile_general} depicts the performance of each algorithm for the SOC-CC instances (similar results can be found for SOC-K in Appendix \ref{appendix:allresutls_knapsack}). The graph illustrates the number of instances solved over time (left-hand side) and the accumulated number of instances over a final gap range (right-hand side). Once again, we see the positive impact of our lifting procedure and the dominance of \bddGeneralLift, \bddTarget, and \bddTargetFlow. 

\subsection{Solution Performance - Branch-and-Cut}

Lastly, we evaluate the effectiveness of the BDD cuts when added in each node of the branch-and-bound tree search. We compare the best CGLP-based approach (i.e., \bddTarget) with our best low-budget alternative (i.e., \bddFlowLift). We also propose a hybrid \hybrid\ that adds BDD target cuts at the root and combinatorial BDD flow cuts during search.  Table  \ref{tab:tree-gen} presents the average results for our two datasets. The table shows the number of instances solved, final gap, average solution time, and nodes explored for instances solved by all techniques, in addition to the total number of cuts added. For ease of comparison, we include again the results when adding cuts only at the root node in row \texttt{Root}, while the new version is included in row \texttt{Search}. 

\begin{table}[tb]
	\footnotesize
  \centering
  \caption{{\changes Aggregated results when adding cuts during tree search.}}
    \begin{tabular}{lll|rrrrr}
    \toprule
          &       &       & \# Solve & Gap   & Time (s)  & \# Nodes & \# Cuts \\
          \midrule
          \multirow{5}[6]{*}{\scriptsize{\textbf{SOC-K}}} & \multirow{2}[2]{*}{\bddFlowLift} & \texttt{Root} & 73    & 0.3\% & 25.8  &    56,376.5  &            54.9  \\
          &       & \texttt{Search} & 78    & 0.1\% & 51.5  &       8,355.5  &       2,918.4  \\
          \cmidrule{2-8}          
          & \multirow{2}[2]{*}{\bddTarget} & \texttt{Root} & \textbf{90} & \textbf{0.0\%} & \textbf{38.6} &    10,949.5  &          108.1  \\
          &       & \texttt{Search} & 47    & 0.7\% & 3516.0 & \textbf{     3,553.0} &          700.5  \\
          \cmidrule{2-8}          
          & \multicolumn{2}{l|}{\hybrid} & 85    & \textbf{0.0\%} & 57.0  &       3,900.0  &       1,866.7  \\
    \midrule
    \midrule
    \multirow{5}[6]{*}{\scriptsize{\textbf{SOC-CC}}} & \multirow{2}[2]{*}{\bddFlowLift} & \texttt{Root} & 149   & 5.6\% & 176.1 &    54,328.5  & 23.3 \\
          &       & \texttt{Search} & 177   & 3.6\% & \textbf{21.4} &       2,046.8  & 1,317.2 \\
\cmidrule{2-8}          & \multirow{2}[2]{*}{\bddTarget} & \texttt{Root} & 172   & 3.9\% & 54.2  &       4,118.2  & 98.1 \\
          &       & \texttt{Search} & 102   & 6.7\% & 704.9 &      \textbf{ 1,524.7}  & 1,249.7 \\
\cmidrule{2-8}          & \multicolumn{2}{l|}{\hybrid} & \textbf{179} & \textbf{3.5\%} & 52.2  & 1,751.6 & 1,326.9 \\
    \bottomrule
    \end{tabular}%
  \label{tab:tree-gen}%
\end{table}%

\bddFlowLift\ has a stronger performance when we add these cuts during search than just at the root node. In fact, \bddFlowLift\  outperform \bddTarget\ over the SOC-CC dataset when we add the latter cuts during search (i.e., 172 vs 177 instances solved). In contrast, \bddTarget\ performs best when the cuts are only added at the root node due to their long separation time (see Table \ref{tab:root_info}). Lastly, Table  \ref{tab:tree-gen} shows that \hybrid\ performs the best on the SOC-CC dataset, solving 179 instances in total. Nonetheless, this hybrid technique solves fewer instances than \bddTarget\ in the SOC-K dataset due to overhead of adding combinatorial cuts during search.

\section{Conclusions} \label{sec:conclusions}

We introduce a novel lifting and cutting-plane procedure for binary programs that leverage their combinatorial structure via a binary decision diagram (BDD) encoding of their constraints. Our lifting procedure relies on 0-1 disjunctions to rotate valid inequalities and uses a BDD to efficiently compute the disjunctive sub-problems. While our combinatorial lifting can enhance any cutting-plane approach, we also propose two novel BDD-based cut generation algorithms based on an alternative network-flow representation of the BDD. 

{\changes
BDDs give us the flexibility to apply our cut-and-lift approach to a wide range of non-linear problems. As a case study, we tested our procedure over second-order conic inequalities and compare its performance against a state-of-the-art solver (\cplex) and existing BDD cuts in the literature. Overall, our lifting procedure reduced the average root gap up to 29\% in our benchmark when applied to cuts generated by all tested methods. Also, a hybrid technique combining our approach with existing BDD cuts proved to be the most efficient over the tested benchmarks.
}

\bibliographystyle{spmpsci}      
\bibliography{bib_bdd_lifting}   

\begin{thebibliography}{10}
\providecommand{\url}[1]{{#1}}
\providecommand{\urlprefix}{URL }
\expandafter\ifx\csname urlstyle\endcsname\relax
  \providecommand{\doi}[1]{DOI~\discretionary{}{}{}#1}\else
  \providecommand{\doi}{DOI~\discretionary{}{}{}\begingroup
  \urlstyle{rm}\Url}\fi

\bibitem{Ahuja1993}
Ahuja, R.K., Magnanti, T.L., Orlin, J.B.: Network Flows: Theory, Algorithms,
  and Applications.
\newblock Prentice-Hall, Inc., USA (1993)

\bibitem{andersen2007constraint}
Andersen, H.R., Hadzic, T., Hooker, J.N., Tiedemann, P.: A constraint store
  based on multivalued decision diagrams.
\newblock In: International Conference on Principles and Practice of Constraint
  Programming--CP 2007, pp. 118--132. Springer (2007)

\bibitem{atamturk2018network}
Atamt{\"u}rk, A., Bhardwaj, A.: Network design with probabilistic capacities.
\newblock Networks \textbf{71}(1), 16--30 (2018)

\bibitem{atamturk2013separation}
Atamt{\"u}rk, A., Muller, L.F., Pisinger, D.: Separation and extension of cover
  inequalities for conic quadratic knapsack constraints with generalized upper
  bounds.
\newblock INFORMS Journal on Computing \textbf{25}(3), 420--431 (2013)

\bibitem{atamturk2009submodular}
Atamt{\"u}rk, A., Narayanan, V.: The submodular knapsack polytope.
\newblock Discrete Optimization \textbf{6}(4), 333--344 (2009)

\bibitem{atamturk2010conic}
Atamt{\"u}rk, A., Narayanan, V.: Conic mixed-integer rounding cuts.
\newblock Mathematical programming \textbf{122}(1), 1--20 (2010)

\bibitem{balas1975facets}
Balas, E.: Facets of the knapsack polytope.
\newblock Mathematical programming \textbf{8}(1), 146--164 (1975)

\bibitem{balas2018disjunctive}
Balas, E.: Disjunctive Programming.
\newblock Springer (2018)

\bibitem{balas1993lift}
Balas, E., Ceria, S., Cornu{\'e}jols, G.: A lift-and-project cutting plane
  algorithm for mixed 0--1 programs.
\newblock Mathematical programming \textbf{58}(1-3), 295--324 (1993)

\bibitem{balas1996mixed}
Balas, E., Ceria, S., Cornu{\'e}jols, G.: Mixed 0-1 programming by
  lift-and-project in a branch-and-cut framework.
\newblock Management Science \textbf{42}(9), 1229--1246 (1996)

\bibitem{becker2005bdds}
Becker, B., Behle, M., Eisenbrand, F., Wimmer, R.: {BDDs} in a branch and cut
  framework.
\newblock In: International Workshop on Experimental and Efficient Algorithms,
  pp. 452--463. Springer (2005)

\bibitem{behle2007binary}
Behle, M.: Binary decision diagrams and integer programming.
\newblock Ph.D. Thesis  (2007)

\bibitem{bergman2019binary}
Bergman, D., Cardonha, C., Mehrani, S.: Binary decision diagrams for bin
  packing with minimum color fragmentation.
\newblock In: International Conference on Integration of Constraint
  Programming, Artificial Intelligence, and Operations Research--CPAIOR 2019,
  pp. 57--66. Springer (2019)

\bibitem{bergman2018nonlinear}
Bergman, D., Cire, A.A.: Discrete nonlinear optimization by state-space
  decompositions.
\newblock Management Science \textbf{64}(10), 4700--4720 (2018)

\bibitem{bergman2012variable}
Bergman, D., Cire, A.A., van Hoeve, W.J., Hooker, J.N.: Variable ordering for
  the application of {BDDs} to the maximum independent set problem.
\newblock In: International Conference on Integration of Constraint
  Programming, Artificial Intelligence, and Operations Research--CPAIOR 2012,
  pp. 34--49. Springer (2012)

\bibitem{bergman2016discrete}
Bergman, D., Cire, A.A., van Hoeve, W.J., Hooker, J.N.: Discrete optimization
  with decision diagrams.
\newblock INFORMS Journal on Computing \textbf{28}(1), 47--66 (2016)

\bibitem{bergman2018quadratic}
Bergman, D., Lozano, L.: Decision diagram decomposition for quadratically
  constrained binary optimization.
\newblock Optimization Online e-prints  (2018)

\bibitem{bhardwaj2015binary}
Bhardwaj, A.: Binary conic quadratic knapsacks.
\newblock Ph.D. thesis, UC Berkeley (2015)

\bibitem{bixby2004mixed}
Bixby, R.E., Fenelon, M., Gu, Z., Rothberg, E., Wunderling, R.: Mixed-integer
  programming: A progress report.
\newblock In: The sharpest cut: the impact of Manfred Padberg and his work, pp.
  309--325. SIAM (2004)

\bibitem{van2018multi}
van~den Bogaerdt, P., de~Weerdt, M.: Multi-machine scheduling lower bounds
  using decision diagrams.
\newblock Operations Research Letters \textbf{46}(6), 616--621 (2018)

\bibitem{van2019lower}
van~den Bogaerdt, P., de~Weerdt, M.: Lower bounds for uniform machine
  scheduling using decision diagrams.
\newblock In: International Conference on Integration of Constraint
  Programming, Artificial Intelligence, and Operations Research--CPAIOR 2019,
  pp. 565--580. Springer (2019)

\bibitem{bryant1986graph}
Bryant, R.E.: Graph-based algorithms for boolean function manipulation.
\newblock Computers, IEEE Transactions on \textbf{100}(8), 677--691 (1986)

\bibitem{castro2019mdd}
Castro, M.P., Cire, A.A., Beck, J.C.: An {MDD}-based lagrangian approach to the
  multicommodity pickup-and-delivery tsp.
\newblock INFORMS Journal on Computing  (2019)

\bibitem{castro2019relaxedbdds}
Castro, M.P., Piacentini, C., Cire, A.A., Beck, J.C.: Relaxed {BDD}s: An
  admissible heuristic for delete-free planning based on a discrete relaxation.
\newblock In: Proceedings of the International Conference on Automated Planning
  and Scheduling, pp. 77--85 (2019)

\bibitem{cire2013multivalued}
Cire, A.A., van Hoeve, W.J.: Multivalued decision diagrams for sequencing
  problems.
\newblock Operations Research \textbf{61}(6), 1411--1428 (2013)

\bibitem{cohen2019overcommitment}
Cohen, M.C., Keller, P.W., Mirrokni, V., Zadimoghaddam, M.: Overcommitment in
  cloud services: Bin packing with chance constraints.
\newblock Management Science \textbf{65}(7), 3255--3271 (2019)

\bibitem{davarnia2020outer}
Davarnia, D., van Hoeve, W.J.: Outer approximation for integer nonlinear
  programs via decision diagrams.
\newblock Mathematical Programming  (2020)

\bibitem{Gomory1969}
Gomory, R.E.: Some polyhedra related to combinatorial problems.
\newblock Linear Algebra and its Applications \textbf{2}(4), 451 -- 558 (1969)

\bibitem{gu1998liftedCom}
Gu, Z., Nemhauser, G.L., Savelsbergh, M.W.: Lifted cover inequalities for 0-1
  integer programs: Computation.
\newblock INFORMS Journal on Computing \textbf{10}(4), 427--437 (1998)

\bibitem{gu1999liftedCplx}
Gu, Z., Nemhauser, G.L., Savelsbergh, M.W.: Lifted cover inequalities for 0-1
  integer programs: Complexity.
\newblock INFORMS Journal on Computing \textbf{11}(1), 117--123 (1999)

\bibitem{gurobi}
Gurobi~Optimization, L.: Gurobi optimizer reference manual (2020)

\bibitem{hammer1975facet}
Hammer, P.L., Johnson, E.L., Peled, U.N.: Facet of regular 0--1 polytopes.
\newblock Mathematical Programming \textbf{8}(1), 179--206 (1975)

\bibitem{hoda2010systematic}
Hoda, S., Van~Hoeve, W.J., Hooker, J.N.: A systematic approach to {MDD}-based
  constraint programming.
\newblock In: International Conference on Principles and Practice of Constraint
  Programming--CP 2010, pp. 266--280. Springer (2010)

\bibitem{hooker2017job}
Hooker, J.N.: Job sequencing bounds from decision diagrams.
\newblock In: International Conference on Principles and Practice of Constraint
  Programming--CP 2017, pp. 565--578. Springer (2017)

\bibitem{CPLEXManual}
IBM: ILOG CPLEX Studio 12.9 Manual (2019)

\bibitem{joung2017lifting}
Joung, S., Park, S.: Lifting of probabilistic cover inequalities.
\newblock Operations Research Letters \textbf{45}(5), 513--518 (2017)

\bibitem{kilincc2016minimal}
K{\i}l{\i}n{\c{c}}-Karzan, F.: On minimal valid inequalities for mixed integer
  conic programs.
\newblock Mathematics of Operations Research \textbf{41}(2), 477--510 (2016)

\bibitem{kilincc2015two}
K{\i}l{\i}n{\c{c}}-Karzan, F., Y{\i}ld{\i}z, S.: Two-term disjunctions on the
  second-order cone.
\newblock Mathematical Programming \textbf{154}(1-2), 463--491 (2015)

\bibitem{kinable2017hybrid}
Kinable, J., Cire, A.A., van Hoeve, W.J.: Hybrid optimization methods for
  time-dependent sequencing problems.
\newblock European Journal of Operational Research \textbf{259}(3), 887--897
  (2017)

\bibitem{lobo1998applications}
Lobo, M.S., Vandenberghe, L., Boyd, S., Lebret, H.: Applications of
  second-order cone programming.
\newblock Linear algebra and its applications \textbf{284}(1-3), 193--228
  (1998)

\bibitem{lodi2010mixed}
Lodi, A.: Mixed integer programming computation.
\newblock In: 50 Years of Integer Programming 1958-2008, pp. 619--645. Springer
  (2010)

\bibitem{lodi2019disjunctive}
Lodi, A., Tanneau, M., Vielma, J.P.: Disjunctive cuts for mixed-integer conic
  optimization.
\newblock arXiv preprint arXiv:1912.03166  (2019)

\bibitem{Louveaux2003}
Louveaux, Q., Wolsey, L.A.: Lifting, superadditivity, mixed integer rounding
  and single node flow sets revisited.
\newblock Quarterly Journal of the Belgian, French and Italian Operations
  Research Societies \textbf{1}(3), 173--207 (2003)

\bibitem{lozano2018binary}
Lozano, L., Smith, J.C.: A binary decision diagram based algorithm for solving
  a class of binary two-stage stochastic programs.
\newblock Mathematical Programming pp. 1--24 (2018)

\bibitem{modaresi2015split}
Modaresi, S., K{\i}l{\i}n{\c{c}}, M.R., Vielma, J.P.: Split cuts and extended
  formulations for mixed integer conic quadratic programming.
\newblock Operations Research Letters \textbf{43}(1), 10--15 (2015)

\bibitem{Nemhauser1988}
Nemhauser, G.L., Wolsey, L.A.: Integer and Combinatorial Optimization.
\newblock Wiley-Interscience, USA (1988)

\bibitem{padberg1973facial}
Padberg, M.W.: On the facial structure of set packing polyhedra.
\newblock Mathematical programming \textbf{5}(1), 199--215 (1973)

\bibitem{padberg1975note}
Padberg, M.W.: A note on zero-one programming.
\newblock Operations Research \textbf{23}(4), 833--837 (1975)

\bibitem{van1963minimum}
Van~de Panne, C., Popp, W.: Minimum-cost cattle feed under probabilistic
  protein constraints.
\newblock Management Science \textbf{9}(3), 405--430 (1963)

\bibitem{perregaard2001generating}
Perregaard, M., Balas, E.: Generating cuts from multiple-term disjunctions.
\newblock In: International Conference on Integer Programming and Combinatorial
  Optimization, pp. 348--360. Springer (2001)

\bibitem{raghunathan2018seamless}
Raghunathan, A.U., Bergman, D., Hooker, J.N., Serra, T., Kobori, S.: Seamless
  multimodal transportation scheduling.
\newblock arXiv preprint arXiv:1807.09676  (2018)

\bibitem{santana2017some}
Santana, A., Dey, S.S.: Some cut-generating functions for second-order conic
  sets.
\newblock Discrete Optimization \textbf{24}, 51--65 (2017)

\bibitem{csen2018conic}
{\c{S}}en, A., Atamt{\"u}rk, A., Kaminsky, P.: A conic integer optimization
  approach to the constrained assortment problem under the mixed multinomial
  logit model.
\newblock Operations Research \textbf{66}(4), 994--1003 (2018)

\bibitem{stubbs1999branch}
Stubbs, R.A., Mehrotra, S.: A branch-and-cut method for 0-1 mixed convex
  programming.
\newblock Mathematical programming \textbf{86}(3), 515--532 (1999)

\bibitem{tjandraatmadja2019target}
Tjandraatmadja, C., van Hoeve, W.J.: Target cuts from relaxed decision
  diagrams.
\newblock INFORMS Journal on Computing \textbf{31}(2), 285--301 (2019)

\bibitem{vielma2008lifted}
Vielma, J.P., Ahmed, S., Nemhauser, G.L.: A lifted linear programming
  branch-and-bound algorithm for mixed-integer conic quadratic programs.
\newblock INFORMS Journal on Computing \textbf{20}(3), 438--450 (2008)

\bibitem{vielma2017extended}
Vielma, J.P., Dunning, I., Huchette, J., Lubin, M.: Extended formulations in
  mixed integer conic quadratic programming.
\newblock Mathematical Programming Computation \textbf{9}(3), 369--418 (2017)

\bibitem{Wolsey1976}
Wolsey, L.A.: Technical note—facets and strong valid inequalities for integer
  programs.
\newblock Operations Research \textbf{24}(2), 367--372 (1976).
\newblock \doi{10.1287/opre.24.2.367}

\bibitem{wolsey1999integer}
Wolsey, L.A., Nemhauser, G.L.: Integer and combinatorial optimization, vol.~55.
\newblock John Wiley \& Sons (1999)

\bibitem{zemel1989easily}
Zemel, E.: Easily computable facets of the knapsack polytope.
\newblock Mathematics of Operations Research \textbf{14}(4), 760--764 (1989)

\end{thebibliography}


\newpage

\appendix

\section{Relaxed BDD Construction Procedure for Second-Order Cones}\label{appendix:soc-bdd}

We now present the relaxed BDD construction procedure for SOC inequalities based on the RSOC model. The construction algorithm is analogous for the case of SOC knapsack constraints and the RSOC-K model.

The state information at the core of our BDD construction algorithm is based on recursive model RSOC. This information is stored in each node of the BDD and  used to identify infeasible assignments and decide how to widen the BDD (i.e., split nodes). Our state information keeps track of each of the $l+1$ linear components in inequality \eqref{eq:soc}, i.e., $\ba^\top \bx$ and $\bd_k^\top \bx - h_k$ for each $k \in \{1,...,l\}$. Thus, each node $u \in \nodes_i$ has two $l+1$ dimensional vectors for top-down information,  $\infoMinDown(u)$ and  $\infoMaxDown(u)$, that approximate the linear components considering the partial assignments from $\rootnode$ to $u$. 
We set the top-down state information at the root node as $\infoMinDown_0(\rootnode):= \infoMaxDown_0(\rootnode):= 0$ and $\infoMinDown_k(\rootnode):= \infoMaxDown_k(\rootnode):= -h_k$ for all $k \in  \{1,...,l\}$. Then, for any $u \in \nodes_i$, $i \in \{2,...,n\}$, and $k \in \{1,..., l\}$ we update the states as:
\begin{align*}
\infoMinDown_0(u):= & \min_{a\in \incoming(u)}\{ \infoMinDown_0(\source(a)) + a_{i-1}\cdot\arcLabel_a \} , \\
\infoMaxDown_0(u):= & \max_{a\in \incoming(u)}\{ \infoMaxDown_0(\source(a)) + a_{i-1}\cdot\arcLabel_a \}, \\
\infoMinDown_k(u):= & \min_{a\in \incoming(u)}\{ \infoMinDown_k(\source(a)) + d_{ki-1}\cdot\arcLabel_a \}, \\
\infoMaxDown_k(u):= & \max_{a\in \incoming(u)}\{ \infoMaxDown_k(\source(a)) + d_{ki-1}\cdot\arcLabel_a \},
\end{align*}

Similarly, we use two $l+1$ dimensional vectors,  $\infoMinUp(u)$ and  $\infoMaxUp(u)$, for our bottom-up state information for each node $u \in \nodes$. The information is initialized at the terminal node as $\infoMinUp_0(\terminalnode):= \infoMaxUp_0(\terminalnode):= 0$ and $\infoMinUp_k(\terminalnode):= \infoMaxUp_k(\terminalnode):= 0$ for all $k \in  \{1,...,l\}$. Then, for any $u \in \nodes_i$, $i \in \{1,...,n-1\}$, and  $k \in \{1,..., l\}$, we have:
\begin{align*}
\infoMinUp_0(u):= & \min_{a\in \outgoing(u)}\{ \infoMinUp_0(\target(a)) + a_k\cdot\arcLabel_a \} , \\
\infoMaxUp_0(u):= & \max_{a\in \outgoing(u)}\{ \infoMaxUp_0(\target(a)) + a_k\cdot\arcLabel_a \}, \\
\infoMinUp_k(u):= & \min_{a\in \outgoing(u)}\{ \infoMinUp_k(\target(a)) + d_{ki}\cdot\arcLabel_a \}, \\
\infoMaxUp_k(u):= & \max_{a\in \outgoing(u)}\{ \infoMaxUp_k(\target(a)) + d_{ki}\cdot\arcLabel_a \}.
\end{align*}

For each node $u \in \nodes$, the state information under and over approximates the value of the linear components of \eqref{eq:soc} for all $\rootnode-u$ paths (i.e., top-down information) and for all $u-\terminalnode$ paths (i.e., bottom-up information).
We use the state information to identify if an arc corresponds to an infeasible assignment, i.e., all paths traversing it correspond to infeasible solutions of \eqref{eq:soc}. In particular, we can remove an arc $a=(u,u')\in \arcs_i$ if the following condition holds:
\begin{equation}\label{eq:filtering}
\infoMinDown_0(u) + a_i\cdot \arcLabel_a  + \infoMinUp_0(u') + \Omega\sqrt{ \sum_{k=1}^l g_k(a) } > b,
\end{equation}
\noindent where $g_k(a)$ for $k \in \{1,..., l\}$ is given by:
\[
g_k(a):= \begin{cases}
(\infoMinDown_k(u) + d_{ki}\cdot\arcLabel_a + \infoMinUp_k(u')) ^2,  & \!\!\mbox{if } \infoMinDown_k(u) + d_{ki}\cdot\arcLabel_a + \infoMinUp_k(u') >0, \\
(\infoMaxDown_k(u) + d_{ki}\cdot\arcLabel_a + \infoMaxUp_k(u'))^2,  & \!\!\mbox{if } \infoMaxDown_k(u) + d_{ki}\cdot\arcLabel_a + \infoMaxUp_k(u') <0, \\
0, & \mbox{otherwise}.
\end{cases}
\]

Notice that $g_k(a)$ under approximates $(\bd_k^\top \bx -d_k)^2$ for all paths traversing arc $a\in \arcs$, and so the left-hand side (LHS) of \eqref{eq:filtering} under approximates the LHS of \eqref{eq:soc} for all paths traversing arc $a\in \arcs_i$. Then,
all paths traversing an arc $a$ that satisfy \eqref{eq:filtering} correspond to invalid assignments for \eqref{eq:soc}.

If $\infoMaxDown_k(u) = \infoMinDown_k(u)$ for all nodes $u \in \nodes$, we recover the exact BDD based on the recursive model RSOC and condition \eqref{eq:filtering} is equivalent to \eqref{eq:rec_b}. Thus, our splitting procedure tries to achieve this property by  selecting nodes $u$ with $\infoMaxDown_k(u) - \infoMinDown_k(u) \geq \delta$ ($\delta > 0$) for some $k \in \{0,...,l\}$ and then split it into two new nodes, $u'$ and $u''$, so  $\infoMaxDown_k(u') - \infoMinDown_k(u')<\delta$ and $\infoMaxDown_k(u'') - \infoMinDown_k(u'')<\delta$. The splitting procedure then duplicates the outgoing arcs of $u$ and assigns them to both $u'$ and $u''$ to keep the same set of paths in $\bdd$.

\begin{algorithm}[htb]
	\caption{Relaxed (Exact) BDD Construction Procedure} \label{alg:bdd_iterative}
	\begin{algorithmic}[1]
		\Procedure{ConstructBDD}{SOC Constraint $\langle\ba,D, \bm{h},\Omega, b, n\rangle$, $\maxwidth$}
		\State $\bdd$ := \algBDDWidthOne($n$)
		\While{$\bdd$ has been modified}
		\State \algBDDNodesBottom($\nodes$)
		\For{$i \in \{1,...,n\}$}
		\State \algBDDNodesTop($\nodes_i$)
		\State \algBDDSplit($\nodes_i$, $\maxwidth$)
		\State \algBDDFilter($\nodes_i$)
		\EndFor
		\State  \algBDDNodesTop($\nodes_{n+1}$)
		\EndWhile
		\State \algReduceBDD($\bdd$)
		\State \algReturn\ $\bdd$ 
		\EndProcedure
	\end{algorithmic} 
\end{algorithm}

Our construction procedure creates a relaxed BDD  $\bdd=(\nodes,\arcs)$  by limiting its width $\width(\bdd)$ by a positive value $\maxwidth$, where $\width(\bdd):=\max_{i\in \varIndex}\{|\nodes_i|\}$ represents the maximum number of nodes in each layer. 
The complete BDD construction procedure is shown in Algorithm \ref{alg:bdd_iterative}. The algorithm starts creating a width-one BDD for the SOC constraint (line 2) and then updates the bottom-up information for all the nodes (line 4). During the top-down pass through the BDD (lines 5-9), the procedure updates the top-down information of layer $\nodes_i$, splits the nodes until we reach the width limit $\maxwidth$, and filters the emanating arcs of $\nodes_i$. The algorithm then checks if the BDD has been updated (line 3) and repeats the bottom-up and top-down iterations until the BDD cannot be updated any more. Lastly, we reduce the BDD (line 10) following the standard procedure in the literature \cite{bryant1986graph}.

The resulting BDD starts with all possible variable assignments (i.e., a width-one BDD) and removes arcs using condition \eqref{eq:filtering}. Thus, the procedure is guaranteed to construct a relaxed BDD for a SOC constraint. Notice that for a big enough $\maxwidth$, the procedure will return an exact BDD.

\begin{example}\label{exa:bdd_constrcution}
	Consider the following binary set defined by an SOC inequality $X= \{\bx \in \{0,1\}^3: 3x_1+x_2+x_3 + \sqrt{ (x_1+x_2+2x_3)^2 + (x_1+3x_2-x_3+3)^2  } \leq 8 \}$. 
	Figure \ref{fig:bdd-soc-example} depicts some of the steps to construct an exact BDD for $X$. The left most diagram corresponds to a width-one BDD for this problem. The top-down state information in  the root node is $( (\infoMinDown_0(\rootnode), \infoMaxDown_0(\rootnode)),\;(\infoMinDown_1(\rootnode), \infoMaxDown_1(\rootnode)),$ $ (\infoMinDown_2(\rootnode), \infoMaxDown_2(\rootnode)) )= ( (0,0),\; (0,0),\; (3,3) )$, while for node $u_1$ is $( (0,3), \; (0,1),\; (3,4)) $.
	
	The middle BDD illustrates the resulting BDD after splitting node $u_1$. The resulting nodes, $u_1'$ and $u_1''$, have top-down state information $ ( (0,0),\; (0,0),\; (3,3) )$ and $ ( (3,3),\; (1,1),$ $ (4,4) )$, respectively. In addition, the gray arc from $u_1''$ to $u_2$ corresponds to an invalid assignment: the bottom-up information of $u_2$ is $ ( (0,1),\; (0,2),\; (-1,0) )$, thus,  \eqref{eq:filtering} evaluates to $10.3 > 8$. 
	
\end{example}

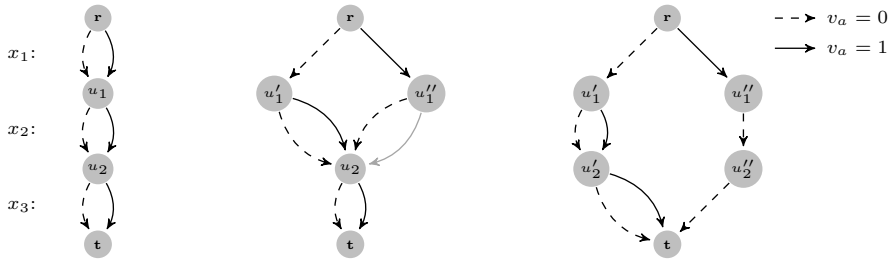
\begin{figure}[htb]
	\centering
	\tikzstyle{zero arc} = [draw,dashed, line width=0.5pt,->]
\tikzstyle{one arc} = [draw,line width=0.5pt,->]
\tikzstyle{optimal arc} = [draw,line width=2pt,->]
\tikzstyle{main node} = [circle,fill=gray!50,font=\tiny\bfseries, inner sep=1pt]
\begin{tikzpicture}[->,>=stealth',shorten >=1pt,auto,node distance=1cm,
thick]        
\node[main node] (r) at (0,0) {$\;\rootnode\;$};
\node[main node] (u1)  at (0,-1)  {$u_1$};
\node[main node] (u2)  at (0,-2) {$u_2$};
\node[main node] (t) at (0,-3) {$\;\terminalnode\;$};

\node (l1) at (-1,-0.5) {\scriptsize{$x_1$:}};
\node (l2) at (-1,-1.5) {\scriptsize{$x_2$:}};
\node (l3) at (-1,-2.5) {\scriptsize{$x_3$:}};

\path[every node/.style={font=\sffamily\small}]
(r) 
edge[zero arc,bend right=30] node [right] {} (u1)
edge[one arc, bend left=30] node [right] {} (u1)
(u1) 
edge[zero arc,bend right=30] node [right] {} (u2)
edge[one arc, bend left=30] node [right] {} (u2)
(u2)
edge[zero arc, bend right=30] node [right] {} (t)
edge[one arc, bend left=30] node [right] {} (t);

\end{tikzpicture}
	\hfill
	\tikzstyle{zero arc} = [draw,dashed, line width=0.5pt,->]
\tikzstyle{one arc} = [draw,line width=0.5pt,->]
\tikzstyle{optimal arc} = [draw,line width=2pt,->]
\tikzstyle{main node} = [circle,fill=gray!50,font=\tiny\bfseries, inner sep=1pt]
\begin{tikzpicture}[->,>=stealth',shorten >=1pt,auto,node distance=1cm,
thick]        
\node[main node] (r) at (0,0) {$\;\rootnode\;$};
\node[main node] (u1)  at (-1,-1)  {$u_1'$};
\node[main node] (u2)  at (1,-1)  {$u_1''$};
\node[main node] (u3)  at (0,-2) {$u_2$};
\node[main node] (t) at (0,-3) {$\;\terminalnode\;$};

\path[every node/.style={font=\sffamily\small}]
(r) 
edge[zero arc] node [left] {} (u1)
edge[one arc] node [left] {} (u2)
(u1) 
edge[zero arc,bend right=30] node [right] {} (u3)
edge[one arc, bend left=30] node [right] {} (u3)
(u2)
edge[zero arc,bend right=30] node [right] {} (u3)
edge[one arc, gray!70!white, bend left=30] node [right] {} (u3)
(u3)
edge[zero arc, bend right=30] node [right] {} (t)
edge[one arc, bend left=30] node [right] {} (t);

\end{tikzpicture}
	\hfill
	\tikzstyle{zero arc} = [draw,dashed, line width=0.5pt,->]
\tikzstyle{one arc} = [draw,line width=0.5pt,->]
\tikzstyle{optimal arc} = [draw,line width=2pt,->]
\tikzstyle{main node} = [circle,fill=gray!50,font=\tiny\bfseries, inner sep=1pt]
\begin{tikzpicture}[->,>=stealth',shorten >=1pt,auto,node distance=1cm,
thick]        
\node[main node] (r) at (0,0) {$\;\rootnode\;$};
\node[main node] (u1)  at (-1,-1)  {$u_1'$};
\node[main node] (u2)  at (1,-1)  {$u_1''$};
\node[main node] (u3)  at (-1,-2) {$u_2'$};
\node[main node] (u4)  at (1,-2) {$u_2''$};
\node[main node] (t) at (0,-3) {$\;\terminalnode\;$};

\node (aux1) at (1.3, 0) {};
\node (aux2) at (2.5, 0) {\scriptsize$\arcLabel_a =0$};
\node (aux3) at (1.3, -0.4) {};
\node (aux4) at (2.5, -0.4) {\scriptsize$\arcLabel_a =1$};

\path[every node/.style={font=\sffamily\small}]
(r) 
edge[zero arc] node [left] {} (u1)
edge[one arc] node [left] {} (u2)
(u1) 
edge[zero arc,bend right=30] node [right] {} (u3)
edge[one arc, bend left=30] node [right] {} (u3)
(u2)
edge[zero arc] node [left] {} (u4)
(u3)
edge[zero arc, bend right=30] node [right] {} (t)
edge[one arc, bend left=30] node [right] {} (t)
(u4)
edge[zero arc] node [left] {} (t)

(aux1) edge[zero arc] node [left] {} (aux2)
(aux3) edge[one arc] node [left] {} (aux4);

\end{tikzpicture}
	\caption{BDD construction procedure for set $X$ defined in Example \ref{exa:bdd_constrcution}. The figure depicts a width-one BDD (left), a BDD after the splitting and filtering procedure over $\nodes_2$ (middle), and the resulting exact reduced BDD (right).}  \label{fig:bdd-soc-example} 
\end{figure}

\section{Experiments Comparing Different BDD Widths} \label{appendix:width}

Table \ref{tab:width_all_general} presents the average performance for \cplex\ and  our four alternatives (i.e., \bddFlow, \bddFlowLift, \bddGeneral, and \bddGeneralLift) with three different maximum width values, $\maxwidth\in \{2000,3000,4000\}$, over the SOC-CC instances. The table shows the number of instances solved, average root gap, and average final gap for all techniques. 
Our four alternatives with $\maxwidth\in \{2000,3000,4000\}$ each outperform \cplex. $\maxwidth=4000$ achieves the best overall performance across the four combinatorial cut-and-lift alternatives. Similarly, $\maxwidth=4000$ achieves the best or comparable performance across the five BDD approaches over the SOC-K instances (right).

\begin{table}[htbp]
  \centering
  \caption{Average performance of all techniques for different BDD widths for SOC-CC.}
    \begin{tabular}{l|r|rrr}
    \toprule
          & \multicolumn{1}{l|}{Width} & \multicolumn{1}{l}{\# Solve} & \multicolumn{1}{l}{Root Gap} & \multicolumn{1}{l}{Final Gap} \\
    \midrule
    \cplex &       & 137   & 20.82\% & 7.75\% \\
    \midrule
    \multirow{3}[2]{*}{\bddFlow} & 2000  & 137   & 20.31\% & 7.35\% \\
          & 3000  & \textbf{140} & 20.11\% & \textbf{7.25\%} \\
          & 4000  & 139   & \textbf{20.01\%} & \textbf{7.25\%} \\
    \midrule
    \multirow{3}[2]{*}{\bddFlowLift} & 2000  & 149   & 16.61\% & 6.09\% \\
          & 3000  & \textbf{150} & 16.03\% & 6.04\% \\
          & 4000  & \textbf{150} & \textbf{15.68\%} & \textbf{5.89\%} \\
    \midrule
    \multirow{3}[2]{*}{\bddGeneral} & 2000  & 160   & 14.93\% & 5.21\% \\
          & 3000  & 159   & 14.05\% & 4.87\% \\
          & 4000  & \textbf{166} & \textbf{13.47\%} & \textbf{4.59\%} \\
    \midrule
    \multirow{3}[2]{*}{\bddGeneralLift} & 2000  & 160   & 14.91\% & 5.00\% \\
          & 3000  & \textbf{168} & 14.03\% & 4.66\% \\
          & 4000  & \textbf{168} & \textbf{13.44\%} & \textbf{4.43\%} \\
    \bottomrule
    \end{tabular}%
  \label{tab:width_all_general}%
\end{table}%

Similarly, Table \ref{tab:width_all_knapsack} presents the average performance over the SOC-K instances for \cplex, our  four alternatives (i.e., \bddFlow, \bddFlowLift, \bddGeneral, and \bddGeneralLift) 
with three different maximum width values, $\maxwidth\in \{2000,3000,4000\}$. Overall, $\maxwidth=4000$ achieves the best or comparable performance across the five BDD approaches.

\begin{table}[htbp]
\small
  \centering
  \caption{Average performance of all techniques for different BDD widths for SOC-K.}
    \begin{tabular}{l|r|rrr}
    \toprule
          & Width & \# Solve & Root Gap & Final Gap \\
    \midrule
    \cplex &       & 70    & 2.84\% & 0.39\% \\
    \midrule
    \multirow{3}[2]{*}{\bddFlow} & 2000  & 66    & 3.64\% & \textbf{0.52\%} \\
          & 3000  & \textbf{67} & 3.63\% & \textbf{0.52\%} \\
          & 4000  & 64    & \textbf{3.63\%} & 0.53\% \\
    \midrule
    \multirow{3}[2]{*}{\bddFlowLift} & 2000  & 74    & 2.84\% & 0.27\% \\
          & 3000  & 72    & \textbf{2.80\%} & 0.28\% \\
          & 4000  & \textbf{74} & \textbf{2.80\%} & \textbf{0.26\%} \\
    \midrule
    \multirow{3}[2]{*}{\bddGeneral} & 2000  & 75    & 1.62\% & 0.16\% \\
          & 3000  & \textbf{78} & 1.45\% & \textbf{0.15\%} \\
          & 4000  & 76    & \textbf{1.34\%} & 0.16\% \\
    \midrule
    \multirow{3}[2]{*}{\bddGeneralLift} & 2000  & 83    & 1.61\% & 0.04\% \\
          & 3000  & 87    & 1.44\% & 0.02\% \\
          & 4000  & \textbf{88} & \textbf{1.33\%} & \textbf{0.01\%} \\
    \bottomrule
    \end{tabular}%
  \label{tab:width_all_knapsack}%
\end{table}%

\section{Average Performance Comparison for Knapsack Chance Constraints }\label{appendix:allresutls_knapsack}

We now present additional results for the SOC-K dataset. As in Figure \ref{fig:profile_general}, Figure \ref{fig:profilh_kcc} illustrate the performance of each algorithm for the SOC-K dataset.  We see a clear dominance of \bddGeneralLift, \bddTarget, and \bddTargetFlow\ and also the positive impact of our combinatorial lifting in instances solved and gap reduction.

\begin{figure}[htb]
	\centering
	\begin{tikzpicture}
		\begin{axis}
			[ legend style={at={(0.85,0.8)},anchor=north, font=\tiny, legend cell align=left},
				width=0.95\textwidth,
				height=0.55\textwidth,
				xlabel={\hspace{25ex} Time (sec) $|$ Optimality gap at time limit }, 
				ylabel={ \# Instances Solved},
				ytick={0,  30, 60, 90},
				yticklabels={0, 30, 60, 90},
				xtick={0,  1200, 2400, 3600, 4400, 5200, 6000, 6800},
				xticklabels={0,  1200, 2400,  $3600\; |\; 0\%\;\;$, 1\%, 2\%, 3\%, 4\%},
				]
			
			\addplot[color=green!60!black] table[x=coord, y=BT-F] {Figures/plots/profile_time_soc_k.txt};
			\addlegendentry{\bddTargetFlow}
			\addplot[color=green!60!black, dashed] table[x=coord, y=BT] {Figures/plots/profile_time_soc_k.txt};
			\addlegendentry{\bddTarget}
			\addplot[color=violet!70!white] table[x=coord, y=BGWL] {Figures/plots/profile_time_soc_k.txt};
			\addlegendentry{\bddGeneralLift}
			\addplot[color=violet!70!white, dashed] table[x=coord, y=BGW] {Figures/plots/profile_time_soc_k.txt};
			\addlegendentry{\bddGeneral}
			\addplot[color= cyan] table[x=coord, y=BFL] {Figures/plots/profile_time_soc_k.txt};
			\addlegendentry{\bddFlowLift}
			\addplot[color= cyan, dashed] table[x=coord, y=BF] {Figures/plots/profile_time_soc_k.txt};
			\addlegendentry{\bddFlow}
			\addplot[color= blue] table[x=coord, y=CPLEX] {Figures/plots/profile_time_soc_k.txt};
			\addlegendentry{\cplex}
			\addplot[color=magenta, dashed] table[x=coord, y=C] {Figures/plots/profile_time_soc_k.txt};
			\addlegendentry{\cover}
			\addplot[color= magenta] table[x=coord, y=CL] {Figures/plots/profile_time_soc_k.txt};
			\addlegendentry{\coverLift}
			\addplot[color= lime!80!black] table[x=coord, y=BCL] {Figures/plots/profile_time_soc_k.txt};
			\addlegendentry{\bddCover}
			\addplot[color= orange, dashed] table[x=coord, y=BP] {Figures/plots/profile_time_soc_k.txt};
			\addlegendentry{\bddProject}
			\addplot[color= orange] table[x=coord, y=BPL] {Figures/plots/profile_time_soc_k.txt};
			\addlegendentry{\bddProjectLift}

			\addplot[gray!80!white, dashed] coordinates {(3600,0) (3600,90)};
			
			\addplot[color=green!60!black, dashed] table[x=coord, y=BT] {Figures/plots/profile_gap_soc_k.txt};
			\addplot[color=green!60!black] table[x=coord, y=BT-F] {Figures/plots/profile_gap_soc_k.txt};
			\addplot[color=violet!70!white] table[x=coord, y=BGWL] {Figures/plots/profile_gap_soc_k.txt};
			\addplot[color=violet!70!white, dashed] table[x=coord, y=BGW] {Figures/plots/profile_gap_soc_k.txt};
			\addplot[color= cyan] table[x=coord, y=BFL] {Figures/plots/profile_gap_soc_k.txt};
			\addplot[color= cyan, dashed] table[x=coord, y=BF] {Figures/plots/profile_gap_soc_k.txt};
			\addplot[color= blue] table[x=coord, y=CPLEX] {Figures/plots/profile_gap_soc_k.txt};
			\addplot[color=magenta, dashed] table[x=coord, y=C] {Figures/plots/profile_gap_soc_k.txt};
			\addplot[color=magenta] table[x=coord, y=CL] {Figures/plots/profile_gap_soc_k.txt};
			\addplot[color= lime!80!black] table[x=coord, y=BCL] {Figures/plots/profile_gap_soc_k.txt};
			\addplot[color= orange, dashed] table[x=coord, y=BP] {Figures/plots/profile_gap_soc_k.txt};
			\addplot[color= orange] table[x=coord, y=BPL] {Figures/plots/profile_gap_soc_k.txt};
		\end{axis}
	\end{tikzpicture}
	\caption{Profile plot comparing the accumulated number of instances solved over time (left), and the accumulated number of instances over a final gap range (right) for SOC-K dataset.  }\label{fig:profilh_kcc}
\end{figure}
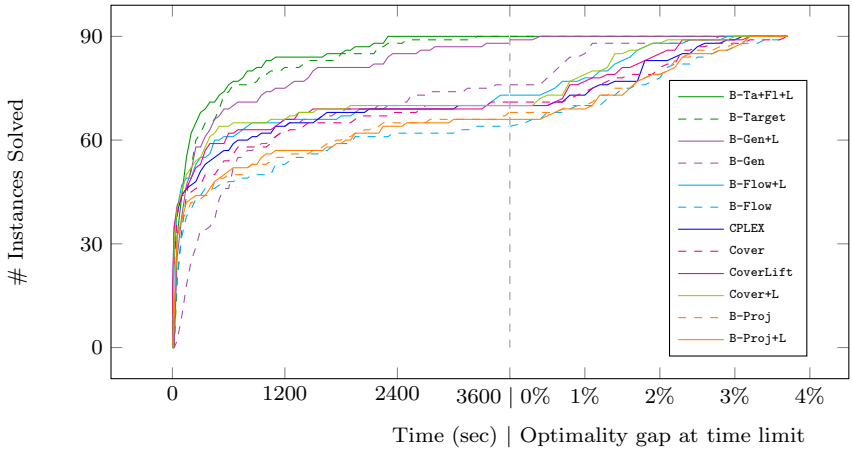

%
%

The following tables show average results for each parameter configuration over the SOC-CC benchmark. We present results for our four variants (i.e., \bddFlow, \bddFlowLift, \bddGeneral, and \bddGeneralLift), cover cut variants (i.e., \cover, \coverLift, and \bddCover), \cplex, and best performing BDD-based cuts (i.e., \bddTarget\ and \bddTargetFlow). All techniques add cuts only at the root node of the tree search. Tables \ref{tab:sock_solve}, \ref{tab:sock_root_gap}, and \ref{tab:sock_gap} show the number of instances solved, average root gap, and average final gap for each $n$, $m$, and $\Omega$ combination with $\maxwidth=4000$, respectively. Similarly,  Tables \ref{tab:sock_nodes} and \ref{tab:sock_time} show the average number of nodes in the branch-and-bound search and the average run time for the instances that all techniques solved to optimality. 

\begin{table}
	\centering
	\scriptsize
	\setlength{\tabcolsep}{2pt}
	\caption{Instances solved to optimality comparison across all techniques for SOC-K benchmark.}
	\begin{tabular}{r|r|r|rrrrrrrrrr}
    \toprule
    \multicolumn{3}{c}{}  & \multicolumn{10}{c}{\textbf{\# Instances Solved }} \\
    \midrule
    $n$     & $m$ & $\Omega$     & \cplex & \cover     & \coverLift    & \bddCover   & \bddFlow    & \bddFlowLift   & \bddGeneral    & \bddGeneralLift   & \bddTarget    & \bddTargetFlow \\
    \midrule
    \multirow{6}[4]{*}{100} & \multirow{3}[2]{*}{10} & 1     & \textbf{5} & \textbf{5} & \textbf{5} & \textbf{5} & \textbf{5} & \textbf{5} & \textbf{5} & \textbf{5} & \textbf{5} & \textbf{5} \\
          &       & 3     & \textbf{5} & \textbf{5} & \textbf{5} & \textbf{5} & \textbf{5} & \textbf{5} & \textbf{5} & \textbf{5} & \textbf{5} & \textbf{5} \\
          &       & 5     & \textbf{5} & \textbf{5} & \textbf{5} & \textbf{5} & \textbf{5} & \textbf{5} & \textbf{5} & \textbf{5} & \textbf{5} & \textbf{5} \\
\cmidrule{2-13}          & \multirow{3}[2]{*}{20} & 1     & \textbf{5} & \textbf{5} & \textbf{5} & \textbf{5} & \textbf{5} & \textbf{5} & \textbf{5} & \textbf{5} & \textbf{5} & \textbf{5} \\
          &       & 3     & \textbf{5} & \textbf{5} & \textbf{5} & \textbf{5} & \textbf{5} & \textbf{5} & \textbf{5} & \textbf{5} & \textbf{5} & \textbf{5} \\
          &       & 5     & \textbf{5} & \textbf{5} & \textbf{5} & \textbf{5} & 3     & \textbf{5} & \textbf{5} & \textbf{5} & \textbf{5} & \textbf{5} \\
    \midrule
    \multirow{6}[4]{*}{125} & \multirow{3}[2]{*}{10} & 1     & \textbf{5} & \textbf{5} & \textbf{5} & \textbf{5} & \textbf{5} & \textbf{5} & \textbf{5} & \textbf{5} & \textbf{5} & \textbf{5} \\
          &       & 3     & \textbf{5} & \textbf{5} & \textbf{5} & \textbf{5} & \textbf{5} & \textbf{5} & \textbf{5} & \textbf{5} & \textbf{5} & \textbf{5} \\
          &       & 5     & \textbf{5} & \textbf{5} & \textbf{5} & \textbf{5} & \textbf{5} & \textbf{5} & \textbf{5} & \textbf{5} & \textbf{5} & \textbf{5} \\
\cmidrule{2-13}          & \multirow{3}[2]{*}{20} & 1     & \textbf{5} & \textbf{5} & \textbf{5} & \textbf{5} & 4     & \textbf{5} & \textbf{5} & \textbf{5} & \textbf{5} & \textbf{5} \\
          &       & 3     & 1     & 2     & 1     & 1     & 1     & 3     & 4     & \textbf{5} & \textbf{5} & \textbf{5} \\
          &       & 5     & 0     & 0     & 0     & 0     & 0     & 1     & 2     & \textbf{5} & \textbf{5} & \textbf{5} \\
    \midrule
    \multirow{6}[4]{*}{150} & \multirow{3}[2]{*}{10} & 1     & \textbf{5} & \textbf{5} & \textbf{5} & \textbf{5} & \textbf{5} & \textbf{5} & \textbf{5} & \textbf{5} & \textbf{5} & \textbf{5} \\
          &       & 3     & \textbf{5} & \textbf{5} & \textbf{5} & \textbf{5} & \textbf{5} & \textbf{5} & \textbf{5} & \textbf{5} & \textbf{5} & \textbf{5} \\
          &       & 5     & \textbf{5} & \textbf{5} & \textbf{5} & \textbf{5} & 3     & \textbf{5} & \textbf{5} & \textbf{5} & \textbf{5} & \textbf{5} \\
\cmidrule{2-13}          & \multirow{3}[2]{*}{20} & 1     & 4     & 4     & 4     & 4     & 3     & 4     & 4     & \textbf{5} & \textbf{5} & \textbf{5} \\
          &       & 3     & 0     & 0     & 0     & 0     & 0     & 0     & 1     & 3     & \textbf{5} & \textbf{5} \\
          &       & 5     & 0     & 0     & 0     & 0     & 0     & 0     & 0     & \textbf{5} & \textbf{5} & \textbf{5} \\
    \midrule
    \multicolumn{3}{r|}{Total} & 70    & 71    & 70    & 70    & 64    & 73    & 76    & 88    & \textbf{90} & \textbf{90} \\
    \bottomrule
    \end{tabular}%
 \label{tab:sock_solve}
\end{table}

\begin{table}
	\centering
	\scriptsize
	\setlength{\tabcolsep}{2pt}
	\caption{Root Gap comparison across all techniques for SOC-K benchmark.}
	\begin{tabular}{r|r|r|rrrrrrrrrr}
    \toprule
    \multicolumn{3}{c}{}      & \multicolumn{10}{c}{\textbf{Root Gap (\%)}} \\
    \midrule
    $n$     & $m$ & $\Omega$     & \cplex & \cover     & \coverLift    & \bddCover   & \bddFlow    & \bddFlowLift   & \bddGeneral    & \bddGeneralLift   & \bddTarget    & \bddTargetFlow \\
    \midrule
    \multicolumn{1}{r|}{\multirow{6}[4]{*}{100}} & \multicolumn{1}{r|}{\multirow{3}[2]{*}{10}} & 1     & 0.9\% & 1.8\% & 1.4\% & 1.4\% & 2.0\% & 1.4\% & \textbf{0.5\%} & \textbf{0.5\%} & \textbf{0.5\%} & \textbf{0.5\%} \\
    \multicolumn{1}{r|}{} & \multicolumn{1}{r|}{} & 3     & 2.4\% & 2.9\% & 2.2\% & 1.9\% & 3.2\% & 2.0\% & 0.5\% & \textbf{0.4\%} & 0.5\% & \textbf{0.4\%} \\
    \multicolumn{1}{r|}{} & \multicolumn{1}{r|}{} & 5     & 3.7\% & 3.9\% & 2.8\% & 2.2\% & 4.4\% & 2.6\% & \textbf{0.3\%} & \textbf{0.3\%} & \textbf{0.3\%} & \textbf{0.3\%} \\
\cmidrule{2-13}    \multicolumn{1}{r|}{} & \multicolumn{1}{r|}{\multirow{3}[2]{*}{20}} & 1     & 1.9\% & 2.8\% & 2.4\% & 2.3\% & 3.2\% & 2.5\% & \textbf{1.1\%} & \textbf{1.1\%} & \textbf{1.1\%} & \textbf{1.1\%} \\
    \multicolumn{1}{r|}{} & \multicolumn{1}{r|}{} & 3     & 4.5\% & 5.0\% & 4.5\% & 4.4\% & 5.4\% & 4.3\% & \textbf{1.5\%} & \textbf{1.5\%} & \textbf{1.5\%} & \textbf{1.5\%} \\
    \multicolumn{1}{r|}{} & \multicolumn{1}{r|}{} & 5     & 6.1\% & 6.3\% & 5.9\% & 5.5\% & 7.1\% & 5.4\% & \textbf{1.5\%} & \textbf{1.5\%} & \textbf{1.5\%} & \textbf{1.5\%} \\
    \midrule
    \multicolumn{1}{r|}{\multirow{6}[4]{*}{125}} & \multicolumn{1}{r|}{\multirow{3}[2]{*}{10}} & 1     & 0.7\% & 1.4\% & 0.9\% & 1.0\% & 1.8\% & 1.2\% & \textbf{0.7\%} & \textbf{0.7\%} & \textbf{0.7\%} & \textbf{0.7\%} \\
    \multicolumn{1}{r|}{} & \multicolumn{1}{r|}{} & 3     & 1.8\% & 2.2\% & 1.6\% & 1.5\% & 2.5\% & 1.6\% & \textbf{0.8\%} & \textbf{0.8\%} & \textbf{0.8\%} & \textbf{0.8\%} \\
    \multicolumn{1}{r|}{} & \multicolumn{1}{r|}{} & 5     & 2.7\% & 2.8\% & 2.0\% & 1.7\% & 3.3\% & 1.8\% & \textbf{0.9\%} & \textbf{0.9\%} & \textbf{0.9\%} & \textbf{0.9\%} \\
\cmidrule{2-13}    \multicolumn{1}{r|}{} & \multicolumn{1}{r|}{\multirow{3}[2]{*}{20}} & 1     & 2.1\% & 2.9\% & 2.6\% & 2.5\% & 3.1\% & 2.7\% & \textbf{1.9\%} & \textbf{1.9\%} & \textbf{1.9\%} & \textbf{1.9\%} \\
    \multicolumn{1}{r|}{} & \multicolumn{1}{r|}{} & 3     & 4.0\% & 4.3\% & 4.1\% & 4.1\% & 4.7\% & 4.1\% & \textbf{2.4\%} & \textbf{2.4\%} & \textbf{2.4\%} & \textbf{2.4\%} \\
    \multicolumn{1}{r|}{} & \multicolumn{1}{r|}{} & 5     & 5.8\% & 6.0\% & 5.6\% & 5.5\% & 6.5\% & 5.6\% & \textbf{2.7\%} & \textbf{2.7\%} & \textbf{2.7\%} & \textbf{2.7\%} \\
    \midrule
    \multicolumn{1}{r|}{\multirow{6}[4]{*}{150}} & \multicolumn{1}{r|}{\multirow{3}[2]{*}{10}} & 1     & \textbf{0.5\%} & 1.1\% & 0.9\% & 0.9\% & 1.3\% & 1.1\% & 0.7\% & 0.7\% & 0.7\% & 0.7\% \\
    \multicolumn{1}{r|}{} & \multicolumn{1}{r|}{} & 3     & 1.6\% & 1.9\% & 1.5\% & 1.4\% & 2.1\% & 1.5\% & \textbf{0.9\%} & \textbf{0.9\%} & \textbf{0.9\%} & \textbf{0.9\%} \\
    \multicolumn{1}{r|}{} & \multicolumn{1}{r|}{} & 5     & 2.7\% & 2.8\% & 2.4\% & 2.3\% & 3.2\% & 2.6\% & \textbf{1.3\%} & \textbf{1.3\%} & \textbf{1.3\%} & \textbf{1.3\%} \\
\cmidrule{2-13}    \multicolumn{1}{r|}{} & \multicolumn{1}{r|}{\multirow{3}[2]{*}{20}} & 1     & 1.6\% & 2.3\% & 1.9\% & 1.9\% & 2.5\% & 2.2\% & \textbf{1.6\%} & \textbf{1.6\%} & \textbf{1.6\%} & \textbf{1.6\%} \\
    \multicolumn{1}{r|}{} & \multicolumn{1}{r|}{} & 3     & 3.3\% & 3.7\% & 3.3\% & 3.1\% & 3.9\% & 3.3\% & \textbf{2.1\%} & \textbf{2.1\%} & \textbf{2.1\%} & \textbf{2.1\%} \\
    \multicolumn{1}{r|}{} & \multicolumn{1}{r|}{} & 5     & 4.7\% & 5.1\% & 4.5\% & 4.3\% & 5.3\% & 4.4\% & 2.6\% & \textbf{2.5\%} & \textbf{2.5\%} & \textbf{2.5\%} \\
    \midrule
    \multicolumn{3}{r|}{Average} & 2.8\% & 3.3\% & 2.8\% & 2.7\% & 3.6\% & 2.8\% & \textbf{1.3\%} & \textbf{1.3\%} & \textbf{1.3\%} & \textbf{1.3\%} \\
    \bottomrule
    \end{tabular}%

 \label{tab:sock_root_gap}
\end{table}

\begin{table}
	\centering
	\scriptsize
	\setlength{\tabcolsep}{2pt}
	\caption{Final Gap comparison across all techniques for SOC-K benchmark.}
	 \begin{tabular}{r|r|r|rrrrrrrrrr}
    \toprule
     \multicolumn{3}{c}{} & \multicolumn{10}{c}{\textbf{Final Gap (\%)}} \\
    \midrule
    $n$     & $m$ & $\Omega$     & \cplex & \cover     & \coverLift    & \bddCover   & \bddFlow    & \bddFlowLift   & \bddGeneral    & \bddGeneralLift   & \bddTarget    & \bddTargetFlow \\
    \midrule
    \multicolumn{1}{r|}{\multirow{6}[4]{*}{100}} & \multicolumn{1}{r|}{\multirow{3}[2]{*}{10}} & 1     & \textbf{0.0\%} & \textbf{0.0\%} & \textbf{0.0\%} & \textbf{0.0\%} & \textbf{0.0\%} & \textbf{0.0\%} & \textbf{0.0\%} & \textbf{0.0\%} & \textbf{0.0\%} & \textbf{0.0\%} \\
    \multicolumn{1}{r|}{} & \multicolumn{1}{r|}{} & 3     & \textbf{0.0\%} & \textbf{0.0\%} & \textbf{0.0\%} & \textbf{0.0\%} & \textbf{0.0\%} & \textbf{0.0\%} & \textbf{0.0\%} & \textbf{0.0\%} & \textbf{0.0\%} & \textbf{0.0\%} \\
    \multicolumn{1}{r|}{} & \multicolumn{1}{r|}{} & 5     & \textbf{0.0\%} & \textbf{0.0\%} & \textbf{0.0\%} & \textbf{0.0\%} & \textbf{0.0\%} & \textbf{0.0\%} & \textbf{0.0\%} & \textbf{0.0\%} & \textbf{0.0\%} & \textbf{0.0\%} \\
\cmidrule{2-13}    \multicolumn{1}{r|}{} & \multicolumn{1}{r|}{\multirow{3}[2]{*}{20}} & 1     & \textbf{0.0\%} & \textbf{0.0\%} & \textbf{0.0\%} & \textbf{0.0\%} & \textbf{0.0\%} & \textbf{0.0\%} & \textbf{0.0\%} & \textbf{0.0\%} & \textbf{0.0\%} & \textbf{0.0\%} \\
    \multicolumn{1}{r|}{} & \multicolumn{1}{r|}{} & 3     & \textbf{0.0\%} & \textbf{0.0\%} & \textbf{0.0\%} & \textbf{0.0\%} & \textbf{0.0\%} & \textbf{0.0\%} & \textbf{0.0\%} & \textbf{0.0\%} & \textbf{0.0\%} & \textbf{0.0\%} \\
    \multicolumn{1}{r|}{} & \multicolumn{1}{r|}{} & 5     & \textbf{0.0\%} & \textbf{0.0\%} & \textbf{0.0\%} & \textbf{0.0\%} & 0.4\% & \textbf{0.0\%} & \textbf{0.0\%} & \textbf{0.0\%} & \textbf{0.0\%} & \textbf{0.0\%} \\
    \midrule
    \multicolumn{1}{r|}{\multirow{6}[4]{*}{125}} & \multicolumn{1}{r|}{\multirow{3}[2]{*}{10}} & 1     & \textbf{0.0\%} & \textbf{0.0\%} & \textbf{0.0\%} & \textbf{0.0\%} & \textbf{0.0\%} & \textbf{0.0\%} & \textbf{0.0\%} & \textbf{0.0\%} & \textbf{0.0\%} & \textbf{0.0\%} \\
    \multicolumn{1}{r|}{} & \multicolumn{1}{r|}{} & 3     & \textbf{0.0\%} & \textbf{0.0\%} & \textbf{0.0\%} & \textbf{0.0\%} & \textbf{0.0\%} & \textbf{0.0\%} & \textbf{0.0\%} & \textbf{0.0\%} & \textbf{0.0\%} & \textbf{0.0\%} \\
    \multicolumn{1}{r|}{} & \multicolumn{1}{r|}{} & 5     & \textbf{0.0\%} & \textbf{0.0\%} & \textbf{0.0\%} & \textbf{0.0\%} & \textbf{0.0\%} & \textbf{0.0\%} & \textbf{0.0\%} & \textbf{0.0\%} & \textbf{0.0\%} & \textbf{0.0\%} \\
\cmidrule{2-13}    \multicolumn{1}{r|}{} & \multicolumn{1}{r|}{\multirow{3}[2]{*}{20}} & 1     & \textbf{0.0\%} & \textbf{0.0\%} & \textbf{0.0\%} & \textbf{0.0\%} & 0.1\% & \textbf{0.0\%} & \textbf{0.0\%} & \textbf{0.0\%} & \textbf{0.0\%} & \textbf{0.0\%} \\
    \multicolumn{1}{r|}{} & \multicolumn{1}{r|}{} & 3     & 1.0\% & 0.9\% & 0.8\% & 0.8\% & 1.2\% & 0.4\% & 0.1\% & \textbf{0.0\%} & \textbf{0.0\%} & \textbf{0.0\%} \\
    \multicolumn{1}{r|}{} & \multicolumn{1}{r|}{} & 5     & 2.5\% & 2.5\% & 2.0\% & 1.6\% & 2.8\% & 1.6\% & 0.5\% & \textbf{0.0\%} & \textbf{0.0\%} & \textbf{0.0\%} \\
    \midrule
    \multicolumn{1}{r|}{\multirow{6}[4]{*}{150}} & \multicolumn{1}{r|}{\multirow{3}[2]{*}{10}} & 1     & \textbf{0.0\%} & \textbf{0.0\%} & \textbf{0.0\%} & \textbf{0.0\%} & \textbf{0.0\%} & \textbf{0.0\%} & \textbf{0.0\%} & \textbf{0.0\%} & \textbf{0.0\%} & \textbf{0.0\%} \\
    \multicolumn{1}{r|}{} & \multicolumn{1}{r|}{} & 3     & \textbf{0.0\%} & \textbf{0.0\%} & \textbf{0.0\%} & \textbf{0.0\%} & \textbf{0.0\%} & \textbf{0.0\%} & \textbf{0.0\%} & \textbf{0.0\%} & \textbf{0.0\%} & \textbf{0.0\%} \\
    \multicolumn{1}{r|}{} & \multicolumn{1}{r|}{} & 5     & \textbf{0.0\%} & \textbf{0.0\%} & \textbf{0.0\%} & \textbf{0.0\%} & 0.2\% & \textbf{0.0\%} & \textbf{0.0\%} & \textbf{0.0\%} & \textbf{0.0\%} & \textbf{0.0\%} \\
\cmidrule{2-13}    \multicolumn{1}{r|}{} & \multicolumn{1}{r|}{\multirow{3}[2]{*}{20}} & 1     & 0.1\% & 0.1\% & 0.2\% & 0.1\% & 0.3\% & 0.1\% & 0.2\% & \textbf{0.0\%} & \textbf{0.0\%} & \textbf{0.0\%} \\
    \multicolumn{1}{r|}{} & \multicolumn{1}{r|}{} & 3     & 1.3\% & 1.3\% & 1.0\% & 0.7\% & 1.6\% & 0.9\% & 0.6\% & 0.1\% & \textbf{0.0\%} & \textbf{0.0\%} \\
    \multicolumn{1}{r|}{} & \multicolumn{1}{r|}{} & 5     & 2.1\% & 2.3\% & 1.8\% & 1.6\% & 2.7\% & 1.6\% & 1.5\% & \textbf{0.0\%} & \textbf{0.0\%} & \textbf{0.0\%} \\
    \midrule
    \multicolumn{3}{r|}{Average} & 0.4\% & 0.4\% & 0.3\% & 0.3\% & 0.5\% & 0.3\% & 0.2\% & \textbf{0.0\%} & \textbf{0.0\%} & \textbf{0.0\%} \\
    \bottomrule
    \end{tabular}%
  
 \label{tab:sock_gap}
\end{table}

\begin{table}
	\centering
	\tiny
	\setlength{\tabcolsep}{2pt}
	\caption{Nodes explored comparison across all techniques for SOC-K benchmark.}
	\begin{tabular}{r|r|r|rrrrrrrrrr}
    \toprule
    \multicolumn{3}{c}{}  & \multicolumn{10}{c}{\textbf{\# Nodes Explored}} \\
    \midrule
    $n$     & $m$ & $\Omega$     & \cplex & \cover     & \coverLift    & \bddCover   & \bddFlow    & \bddFlowLift   & \bddGeneral    & \bddGeneralLift   & \bddTarget    & \bddTargetFlow \\
    \midrule
    \multicolumn{1}{r|}{\multirow{6}[4]{*}{100}} & \multicolumn{1}{r|}{\multirow{3}[2]{*}{10}} & 1     &         2,187  &             9,493  &             3,407  &             3,808  &          13,279  &             5,055  &             191  &             170  &             173  & \textbf{            159} \\
    \multicolumn{1}{r|}{} & \multicolumn{1}{r|}{} & 3     &      15,911  &          38,018  &          18,616  &          13,374  &          77,135  &          11,973  &             127  &             100  & \textbf{               92} &             167  \\
    \multicolumn{1}{r|}{} & \multicolumn{1}{r|}{} & 5     &      56,024  &          73,318  &          34,888  &          11,065  &       197,240  &          24,390  &             182  &             171  &             214  & \textbf{            165} \\
\cmidrule{2-13}    \multicolumn{1}{r|}{} & \multicolumn{1}{r|}{\multirow{3}[2]{*}{20}} & 1     &      12,093  &          84,852  &          36,813  &          37,479  &       152,371  &          56,573  &         1,929  &         1,128  &         1,088  & \textbf{        1,015} \\
    \multicolumn{1}{r|}{} & \multicolumn{1}{r|}{} & 3     &   275,046  &       498,671  &       367,255  &       326,495  &   1,061,138  &       271,931  &         3,062  &         1,465  &         1,591  & \textbf{        1,181} \\
    \multicolumn{1}{r|}{} & \multicolumn{1}{r|}{} & 5     &   796,298  &   1,030,628  &   1,073,455  &       431,804  &   1,648,577  &       655,462  &         1,461  &             607  &             578  & \textbf{            559} \\
    \midrule
    \multicolumn{1}{r|}{\multirow{6}[4]{*}{125}} & \multicolumn{1}{r|}{\multirow{3}[2]{*}{10}} & 1     & \textbf{        2,557} &          26,109  &             5,254  &             6,707  &          59,136  &          22,296  &         3,864  &         4,875  &         2,747  &         3,476  \\
    \multicolumn{1}{r|}{} & \multicolumn{1}{r|}{} & 3     &      16,810  &          89,336  &          35,721  &          28,278  &       190,740  &          39,481  &         8,114  & \textbf{        4,436} &         4,720  &         4,751  \\
    \multicolumn{1}{r|}{} & \multicolumn{1}{r|}{} & 5     &   157,561  &       301,041  &       253,128  &          87,198  &   1,290,400  &          89,930  &         6,137  &         4,415  &         3,665  & \textbf{        3,521} \\
\cmidrule{2-13}    \multicolumn{1}{r|}{} & \multicolumn{1}{r|}{\multirow{3}[2]{*}{20}} & 1     &   243,540  &   1,134,038  &       628,234  &       584,492  &   2,016,914  &       568,773  &   120,849  &      90,177  &      83,596  & \textbf{     68,425} \\
    \multicolumn{1}{r|}{} & \multicolumn{1}{r|}{} & 3     &   445,816  &       902,386  &       476,674  &       513,874  &   2,206,133  &   1,047,598  &      86,749  &      30,173  & \textbf{     24,841} &      26,485  \\
    \multicolumn{1}{r|}{} & \multicolumn{1}{r|}{} & 5     &  -    &  -    &  -    &  -    &  -    &  -    &  -    &  -    &  -    &  -  \\
    \midrule
    \multicolumn{1}{r|}{\multirow{6}[4]{*}{150}} & \multicolumn{1}{r|}{\multirow{3}[2]{*}{10}} & 1     & \textbf{        1,835} &          39,165  &          14,878  &          18,218  &          61,289  &          33,738  &      17,673  &         8,190  &      16,377  &      10,179  \\
    \multicolumn{1}{r|}{} & \multicolumn{1}{r|}{} & 3     &   267,752  &   1,156,661  &   1,084,202  &       400,829  &   1,900,415  &       399,942  &      56,718  &      36,842  & \textbf{     29,805} &      31,726  \\
    \multicolumn{1}{r|}{} & \multicolumn{1}{r|}{} & 5     &   394,755  &       877,686  &       423,235  &       337,258  &   2,718,140  &       658,492  &      88,738  &   111,968  & \textbf{     63,792} &      64,269  \\
\cmidrule{2-13}    \multicolumn{1}{r|}{} & \multicolumn{1}{r|}{\multirow{3}[2]{*}{20}} & 1     &   453,059  &   2,707,252  &   1,202,011  &   1,186,773  &   3,237,491  &   2,514,477  &   541,442  &   482,927  &   442,424  & \textbf{  377,666} \\
    \multicolumn{1}{r|}{} & \multicolumn{1}{r|}{} & 3     &  -    &  -    &  -    &  -    &  -    &  -    &  -    &  -    &  -    &  -  \\
    \multicolumn{1}{r|}{} & \multicolumn{1}{r|}{} & 5     &  -    &  -    &  -    &  -    &  -    &  -    &  -    &  -    &  -    &  -  \\
    \midrule
    \multicolumn{3}{r|}{Average} &   209,416  &       597,910  &       377,185  &       265,843  &   1,122,027  &       426,674  &      62,482  &      51,843  &      45,047  & \textbf{     39,583} \\
    \bottomrule
    \end{tabular}%
  
 \label{tab:sock_nodes}
\end{table}

\begin{table}
	\centering
	\scriptsize
	\setlength{\tabcolsep}{1.5pt}
	\caption{Solving time comparison across all techniques for SOC-K benchmark.}
	\begin{tabular}{r|r|r|rrrrrrrrrr}
    \toprule
    \multicolumn{3}{c|}{} & \multicolumn{10}{c}{\textbf{Solving Time (sec)}} \\
    \midrule
    $n$     & $m$ & $\Omega$     & \cplex & \cover     & \coverLift    & \bddCover   & \bddFlow    & \bddFlowLift   & \bddGeneral    & \bddGeneralLift   & \bddTarget    & \bddTargetFlow \\
    \midrule
    \multicolumn{1}{r|}{\multirow{6}[4]{*}{100}} & \multicolumn{1}{r|}{\multirow{3}[2]{*}{10}} & 1     & 3.3   & 2.7   & \textbf{1.4} & 7.2   & 13.5  & 7.6   & 138.5 & 31.5  & 37.2  & 24.4 \\
    \multicolumn{1}{r|}{} & \multicolumn{1}{r|}{} & 3     & 9.4   & 12.0  & \textbf{6.5} & 9.8   & 46.2  & 9.6   & 170.2 & 48.8  & 61.1  & 36.4 \\
    \multicolumn{1}{r|}{} & \multicolumn{1}{r|}{} & 5     & 32.3  & 24.6  & 13.0  & \textbf{9.0} & 97.2  & 13.1  & 391.8 & 55.2  & 72.5  & 33.0 \\
\cmidrule{2-13}    \multicolumn{1}{r|}{} & \multicolumn{1}{r|}{\multirow{3}[2]{*}{20}} & 1     & 19.1  & 38.5  & \textbf{18.8} & 30.0  & 121.3 & 39.7  & 457.6 & 149.4 & 135.5 & 84.2 \\
    \multicolumn{1}{r|}{} & \multicolumn{1}{r|}{} & 3     & 317.8 & 273.7 & 217.1 & 190.3 & 762.6 & \textbf{162.8} & 753.6 & 260.4 & 235.7 & 179.6 \\
    \multicolumn{1}{r|}{} & \multicolumn{1}{r|}{} & 5     & 970.4 & 718.5 & 744.6 & 331.2 & 1330.4 & 473.2 & 1232.0 & 374.7 & 373.7 & \textbf{220.1} \\
    \midrule
    \multicolumn{1}{r|}{\multirow{6}[4]{*}{125}} & \multicolumn{1}{r|}{\multirow{3}[2]{*}{10}} & 1     & 3.3   & 6.8   & \textbf{2.4} & 12.1  & 36.2  & 16.1  & 100.3 & 29.1  & 44.2  & 27.0 \\
    \multicolumn{1}{r|}{} & \multicolumn{1}{r|}{} & 3     & 11.8  & 23.7  & \textbf{12.7} & 18.4  & 80.6  & 20.9  & 415.5 & 67.2  & 88.9  & 46.6 \\
    \multicolumn{1}{r|}{} & \multicolumn{1}{r|}{} & 5     & 105.5 & 101.0 & 83.4  & 41.9  & 400.6 & \textbf{39.6} & 826.3 & 80.9  & 113.5 & 55.3 \\
\cmidrule{2-13}    \multicolumn{1}{r|}{} & \multicolumn{1}{r|}{\multirow{3}[2]{*}{20}} & 1     & 357.6 & 567.0 & 322.3 & 323.5 & 1031.0 & 279.5 & 409.6 & 227.1 & 220.0 & \textbf{146.7} \\
    \multicolumn{1}{r|}{} & \multicolumn{1}{r|}{} & 3     & 645.0 & 668.3 & 338.0 & 425.5 & 1716.2 & 652.4 & 946.3 & 518.6 & 523.2 & \textbf{275.1} \\
    \multicolumn{1}{r|}{} & \multicolumn{1}{r|}{} & 5     & -     & -     & -     & -     & -     & -     & -     & -     & -     & - \\
\cmidrule{2-13}    \multicolumn{1}{r|}{\multirow{6}[4]{*}{150}} & \multicolumn{1}{r|}{\multirow{3}[2]{*}{10}} & 1     & 2.9   & 10.0  & \textbf{5.2} & 18.6  & 36.1  & 22.8  & 85.2  & 32.1  & 41.7  & 29.1 \\
    \multicolumn{1}{r|}{} & \multicolumn{1}{r|}{} & 3     & 183.1 & 387.3 & 345.0 & 140.5 & 759.0 & 129.2 & 323.6 & 94.8  & 98.9  & \textbf{61.2} \\
    \multicolumn{1}{r|}{} & \multicolumn{1}{r|}{} & 5     & 252.5 & 260.5 & 156.5 & 127.5 & 1080.5 & 213.2 & 559.7 & 182.4 & 148.7 & \textbf{100.9} \\
\cmidrule{2-13}    \multicolumn{1}{r|}{} & \multicolumn{1}{r|}{\multirow{3}[2]{*}{20}} & 1     & 766.9 & 1506.0 & 808.8 & 750.1 & 2167.0 & 1341.7 & 914.8 & 391.4 & 375.5 & \textbf{324.3} \\
    \multicolumn{1}{r|}{} & \multicolumn{1}{r|}{} & 3     & -     & -     & -     & -     & -     & -     & -     & -     & -     & - \\
    \multicolumn{1}{r|}{} & \multicolumn{1}{r|}{} & 5     & -     & -     & -     & -     & -     & -     & -     & -     & -     & - \\
    \midrule
    \multicolumn{3}{r|}{Average} &   245.4  &   306.7  &   205.0  &   162.4  &   645.2  &   228.1  &   515.0  &   169.6  &   171.4  & \textbf{  109.6 } \\
    \bottomrule
    \end{tabular}%
   \label{tab:sock_time}
\end{table}


%

\newpage
\section{Average Performance Comparison for General Chance Constraints } \label{appendix:allresutls_general}

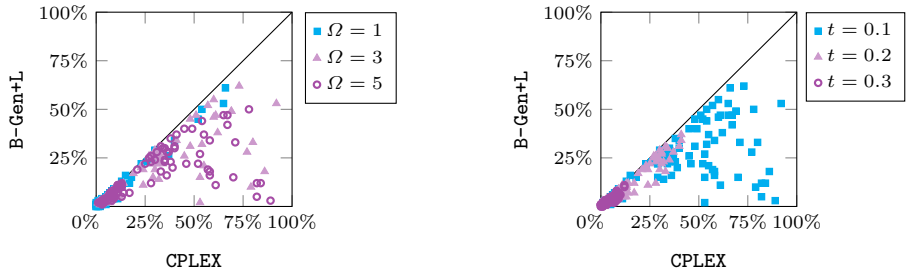
\begin{figure}[htb]
	\begin{tikzpicture}
		\begin{axis}[ legend style={at={(1.3,1)},anchor=north,font=\scriptsize},
			xlabel={\cplex}, ylabel={\bddGeneralLift},
			width=0.35\textwidth,
			height=0.35\textwidth,
			ymin=0, ymax=1, xmin=0, xmax=1,
			label style={font=\small},
			tick label style={font=\scriptsize},
			xtick={0,  0.25, 0.5, 0.75, 1},
			xticklabels={$0\%\quad$,  25\%, 50\%, 75\%, 100\%},
			ytick={0.25, 0.5, 0.75, 1},
			yticklabels={ 25\%, 50\%, 75\%, 100\%},
			ylabel near ticks ]
			
			\addplot[only marks, mark =square*, mark options={scale=0.6}, color= cyan] table {Figures/plots/rgap_bddvsCplex_o1_4000.txt};
			\addlegendentry{$\Omega = 1$}
			\addplot[only marks, mark =triangle*, mark options={scale=0.8}, color= violet!40!white] table {Figures/plots/rgap_bddvsCplex_o3_4000.txt};
			\addlegendentry{$\Omega = 3$}
			\addplot[only marks, mark =o, mark options={scale=0.6, thick}, color=violet!70!white ] table {Figures/plots/rgap_bddvsCplex_o5_4000.txt};
			\addlegendentry{$\Omega = 5$}
			\addplot[black] coordinates {(0,0) (1,1)};
		\end{axis}
	\end{tikzpicture}
	\hfill
	\begin{tikzpicture}
		\begin{axis}[ legend style={at={(1.3,1)},anchor=north,font=\scriptsize},
			xlabel={\cplex}, ylabel={\bddGeneralLift},
			width=0.35\textwidth,
			height=0.35\textwidth,
			ymin=0, ymax=1, xmin=0, xmax=1,
			label style={font=\small},
			tick label style={font=\scriptsize},
			xtick={0,  0.25, 0.5, 0.75, 1},
			xticklabels={$0\%\quad$,  25\%, 50\%, 75\%, 100\%},
			ytick={0.25, 0.5, 0.75, 1},
			yticklabels={ 25\%, 50\%, 75\%, 100\%},
			ylabel near ticks  ]
			
			\addplot[only marks, mark =square*, mark options={scale=0.6}, color= cyan] table {Figures/plots/rgap_bddvsCplex_b1_4000.txt};
			\addlegendentry{$t=0.1$}
			\addplot[only marks, mark =triangle*, mark options={scale=0.8}, color= violet!40!white] table {Figures/plots/rgap_bddvsCplex_b2_4000.txt};
			\addlegendentry{$t=0.2$}
			\addplot[only marks, mark =o, mark options={scale=0.6, thick}, color=violet!70!white ] table {Figures/plots/rgap_bddvsCplex_b3_4000.txt};
			\addlegendentry{$t = 0.3$}
			\addplot[black] coordinates {(0,0) (1,1)};
		\end{axis}
	\end{tikzpicture}
	\caption{Root node gap comparison with \cplex\ for the SOC-CC instances. The plot on the left considers different values of $\Omega$, and the plot on the right  different values of $t$. } \label{fig:gap_config}
\end{figure}

We now present additional results for the SOC-CC dataset. Figure \ref{fig:gap_config} shows two plots comparing the root gap of \bddGeneralLift\ and \cplex\ for the SOC-CC instances and different values of $\Omega$ and $t$. In each plot, an $(x,y)$ point represents the root gap for an instance given by the $x$-axis and the $y$-axis technique, respectively. 
Overall, we can see that \bddGeneralLift\ achieves a smaller or equal root gap to \cplex, however, the difference is considerably larger when $\Omega \geq 3$ and $t=0.1$.

The problems become more challenging with a larger $\Omega$ (i.e., a predominant quadratic term) due to a weak SOC relaxation and linearization. Thus, our procedure can potentially generate stronger cuts than \cplex.  In fact, the left plot in Figure \ref{fig:gap_config} shows all instances with $\Omega =1$ close to the diagonal, while problems with $\Omega \in \{3,5\}$ have larger gap reductions. Lastly, the right plot of Figure \ref{fig:gap_config} shows that \bddGeneralLift\ has significantly smaller gaps than \cplex\ over instances with a small $t$ (i.e., small solution sets). Our relaxed BDDs are close to exact BDDs in these cases, thus, making our cuts more effective.

The following tables show average results for each parameter configuration over the SOC-CC benchmark. We present results for our four variants (i.e., \bddFlow, \bddFlowLift, \bddGeneral, and \bddGeneralLift), \cplex, and best performing BDD-based cuts (i.e., \bddTarget\ and \bddTargetFlow). All techniques add cuts only at the root node of the tree search. Tables \ref{tab:cc_gap}, \ref{tab:cc_root_gap}, and \ref{tab:cc_gap} show the number of instances solved, average root gap, and average final gap for each $n$, $m$, $\Omega$, and $t$ combination, with $\maxwidth=4000$. Similarly,  Tables \ref{tab:cc_nodes} and \ref{tab:cc_time} show the average number of nodes in the branch-and-bound search and the average run time for the instances that all techniques solved to optimality. 

\aboverulesep=0ex
\belowrulesep=0.1ex
\renewcommand{\arraystretch}{1.15}

\begin{table}[htb]
	\centering
	\scriptsize
	\setlength{\tabcolsep}{4pt}
	\caption{Instances solved to Optimality comparison across all techniques for SOC-CC benchmark.}
	\begin{tabular}{rrrr|ccccccc}
    \toprule
          &       &       & \multicolumn{1}{r}{} & \multicolumn{7}{c}{\textbf{\# Instances Solved}} \\
    \midrule
    \multicolumn{1}{r|}{$t$} & \multicolumn{1}{r|}{$n$} & \multicolumn{1}{r|}{$m$} & $\Omega$     & \cplex & \bddFlow    & \bddFlowLift   & \bddGeneral    & \bddGeneralLift   & \bddTarget    & \bddTargetFlow \\
    \midrule
    \multicolumn{1}{r|}{\multirow{18}[6]{*}{0.1}} & \multicolumn{1}{r|}{\multirow{6}[2]{*}{75}} & \multicolumn{1}{r|}{\multirow{3}[1]{*}{10}} & 1     & \textbf{5} & \textbf{5} & \textbf{5} & \textbf{5} & \textbf{5} & \textbf{5} & \textbf{5} \\
    \multicolumn{1}{r|}{} & \multicolumn{1}{r|}{} & \multicolumn{1}{r|}{} & 3     & \textbf{5} & \textbf{5} & \textbf{5} & \textbf{5} & \textbf{5} & \textbf{5} & \textbf{5} \\
    \multicolumn{1}{r|}{} & \multicolumn{1}{r|}{} & \multicolumn{1}{r|}{} & 5     & \textbf{5} & \textbf{5} & \textbf{5} & \textbf{5} & \textbf{5} & \textbf{5} & \textbf{5} \\
    \cmidrule{3-11}\multicolumn{1}{r|}{} & \multicolumn{1}{r|}{} & \multicolumn{1}{r|}{\multirow{3}[1]{*}{20}} & 1     & \textbf{5} & \textbf{5} & \textbf{5} & \textbf{5} & \textbf{5} & \textbf{5} & \textbf{5} \\
    \multicolumn{1}{r|}{} & \multicolumn{1}{r|}{} & \multicolumn{1}{r|}{} & 3     & \textbf{5} & \textbf{5} & \textbf{5} & \textbf{5} & \textbf{5} & \textbf{5} & \textbf{5} \\
    \multicolumn{1}{r|}{} & \multicolumn{1}{r|}{} & \multicolumn{1}{r|}{} & 5     & \textbf{5} & \textbf{5} & \textbf{5} & \textbf{5} & \textbf{5} & \textbf{5} & \textbf{5} \\
\cmidrule{2-11}    \multicolumn{1}{r|}{} & \multicolumn{1}{r|}{\multirow{6}[2]{*}{100}} & \multicolumn{1}{r|}{\multirow{3}[1]{*}{10}} & 1     & \textbf{5} & \textbf{5} & \textbf{5} & \textbf{5} & \textbf{5} & \textbf{5} & \textbf{5} \\
    \multicolumn{1}{r|}{} & \multicolumn{1}{r|}{} & \multicolumn{1}{r|}{} & 3     & \textbf{5} & \textbf{5} & \textbf{5} & \textbf{5} & \textbf{5} & \textbf{5} & \textbf{5} \\
    \multicolumn{1}{r|}{} & \multicolumn{1}{r|}{} & \multicolumn{1}{r|}{} & 5     & \textbf{5} & \textbf{5} & \textbf{5} & \textbf{5} & \textbf{5} & \textbf{5} & \textbf{5} \\
    \cmidrule{3-11}\multicolumn{1}{r|}{} & \multicolumn{1}{r|}{} & \multicolumn{1}{r|}{\multirow{3}[1]{*}{20}} & 1     & 0     & 0     & 0     & 1     & 1     & \textbf{2} & \textbf{2} \\
    \multicolumn{1}{r|}{} & \multicolumn{1}{r|}{} & \multicolumn{1}{r|}{} & 3     & \textbf{5} & \textbf{5} & \textbf{5} & \textbf{5} & \textbf{5} & \textbf{5} & \textbf{5} \\
    \multicolumn{1}{r|}{} & \multicolumn{1}{r|}{} & \multicolumn{1}{r|}{} & 5     & \textbf{5} & \textbf{5} & \textbf{5} & \textbf{5} & \textbf{5} & \textbf{5} & \textbf{5} \\
\cmidrule{2-11}    \multicolumn{1}{r|}{} & \multicolumn{1}{r|}{\multirow{6}[2]{*}{125}} & \multicolumn{1}{r|}{\multirow{3}[1]{*}{10}} & 1     & 4     & 3     & 3     & 4     & 4     & \textbf{5} & 4 \\
    \multicolumn{1}{r|}{} & \multicolumn{1}{r|}{} & \multicolumn{1}{r|}{} & 3     & 0     & 0     & 0     & 0     & 0     & \textbf{1} & 0 \\
    \multicolumn{1}{r|}{} & \multicolumn{1}{r|}{} & \multicolumn{1}{r|}{} & 5     & 0     & 0     & \textbf{2} & 1     & 1     & 1     & 1 \\
    \cmidrule{3-11}\multicolumn{1}{r|}{} & \multicolumn{1}{r|}{} & \multicolumn{1}{r|}{\multirow{3}[1]{*}{20}} & 1     & 0     & 0     & 0     & 0     & 0     & 0     & 0 \\
    \multicolumn{1}{r|}{} & \multicolumn{1}{r|}{} & \multicolumn{1}{r|}{} & 3     & 0     & 0     & 1     & 1     & 1     & \textbf{2} & 1 \\
    \multicolumn{1}{r|}{} & \multicolumn{1}{r|}{} & \multicolumn{1}{r|}{} & 5     & 0     & 0     & 1     & 2     & 2     & \textbf{3} & 2 \\
    \midrule
    \multicolumn{1}{r|}{\multirow{18}[6]{*}{0.2}} & \multicolumn{1}{r|}{\multirow{6}[2]{*}{75}} & \multicolumn{1}{r|}{\multirow{3}[1]{*}{10}} & 1     & \textbf{5} & \textbf{5} & \textbf{5} & \textbf{5} & \textbf{5} & \textbf{5} & \textbf{5} \\
    \multicolumn{1}{r|}{} & \multicolumn{1}{r|}{} & \multicolumn{1}{r|}{} & 3     & 2     & 3     & 4     & \textbf{5} & \textbf{5} & \textbf{5} & \textbf{5} \\
    \multicolumn{1}{r|}{} & \multicolumn{1}{r|}{} & \multicolumn{1}{r|}{} & 5     & 1     & 1     & 3     & \textbf{5} & \textbf{5} & \textbf{5} & \textbf{5} \\
    \cmidrule{3-11}\multicolumn{1}{r|}{} & \multicolumn{1}{r|}{} & \multicolumn{1}{r|}{\multirow{3}[1]{*}{20}} & 1     & 2     & 2     & 3     & 3     & \textbf{4} & \textbf{4} & 3 \\
    \multicolumn{1}{r|}{} & \multicolumn{1}{r|}{} & \multicolumn{1}{r|}{} & 3     & 0     & 0     & 0     & 0     & \textbf{1} & 0     & \textbf{1} \\
    \multicolumn{1}{r|}{} & \multicolumn{1}{r|}{} & \multicolumn{1}{r|}{} & 5     & 0     & 0     & 0     & 0     & \textbf{1} & \textbf{1} & \textbf{1} \\
\cmidrule{2-11}    \multicolumn{1}{r|}{} & \multicolumn{1}{r|}{\multirow{6}[2]{*}{100}} & \multicolumn{1}{r|}{\multirow{3}[1]{*}{10}} & 1     & \textbf{5} & \textbf{5} & \textbf{5} & \textbf{5} & \textbf{5} & \textbf{5} & \textbf{5} \\
    \multicolumn{1}{r|}{} & \multicolumn{1}{r|}{} & \multicolumn{1}{r|}{} & 3     & 0     & 0     & 1     & 2     & 2     & \textbf{3} & 2 \\
    \multicolumn{1}{r|}{} & \multicolumn{1}{r|}{} & \multicolumn{1}{r|}{} & 5     & 0     & 0     & 0     & 0     & 0     & 0     & 0 \\
    \cmidrule{3-11}\multicolumn{1}{r|}{} & \multicolumn{1}{r|}{} & \multicolumn{1}{r|}{\multirow{3}[1]{*}{20}} & 1     & 0     & 0     & 1     & 1     & \textbf{2} & \textbf{2} & \textbf{2} \\
    \multicolumn{1}{r|}{} & \multicolumn{1}{r|}{} & \multicolumn{1}{r|}{} & 3     & 0     & 0     & 0     & 0     & 0     & 0     & 0 \\
    \multicolumn{1}{r|}{} & \multicolumn{1}{r|}{} & \multicolumn{1}{r|}{} & 5     & 0     & 0     & 0     & 0     & 0     & 0     & 0 \\
\cmidrule{2-11}    \multicolumn{1}{r|}{} & \multicolumn{1}{r|}{\multirow{6}[2]{*}{125}} & \multicolumn{1}{r|}{\multirow{3}[1]{*}{10}} & 1     & \textbf{5} & \textbf{5} & \textbf{5} & \textbf{5} & \textbf{5} & \textbf{5} & \textbf{5} \\
    \multicolumn{1}{r|}{} & \multicolumn{1}{r|}{} & \multicolumn{1}{r|}{} & 3     & 0     & 0     & 0     & 0     & 0     & 0     & 0 \\
    \multicolumn{1}{r|}{} & \multicolumn{1}{r|}{} & \multicolumn{1}{r|}{} & 5     & 0     & 0     & 0     & 0     & 0     & 0     & 0 \\
    \cmidrule{3-11}\multicolumn{1}{r|}{} & \multicolumn{1}{r|}{} & \multicolumn{1}{r|}{\multirow{3}[1]{*}{20}} & 1     & 0     & 0     & 0     & 0     & 0     & 0     & 0 \\
    \multicolumn{1}{r|}{} & \multicolumn{1}{r|}{} & \multicolumn{1}{r|}{} & 3     & 0     & 0     & 0     & 0     & 0     & 0     & 0 \\
    \multicolumn{1}{r|}{} & \multicolumn{1}{r|}{} & \multicolumn{1}{r|}{} & 5     & 0     & 0     & 0     & 0     & 0     & 0     & 0 \\
    \midrule
    \multicolumn{1}{r|}{\multirow{18}[6]{*}{0.3}} & \multicolumn{1}{r|}{\multirow{6}[2]{*}{75}} & \multicolumn{1}{r|}{\multirow{3}[1]{*}{10}} & 1     & \textbf{5} & \textbf{5} & \textbf{5} & \textbf{5} & \textbf{5} & \textbf{5} & \textbf{5} \\
    \multicolumn{1}{r|}{} & \multicolumn{1}{r|}{} & \multicolumn{1}{r|}{} & 3     & \textbf{5} & \textbf{5} & \textbf{5} & \textbf{5} & \textbf{5} & \textbf{5} & \textbf{5} \\
    \multicolumn{1}{r|}{} & \multicolumn{1}{r|}{} & \multicolumn{1}{r|}{} & 5     & 4     & 4     & \textbf{5} & \textbf{5} & \textbf{5} & \textbf{5} & \textbf{5} \\
    \cmidrule{3-11}\multicolumn{1}{r|}{} & \multicolumn{1}{r|}{} & \multicolumn{1}{r|}{\multirow{3}[1]{*}{20}} & 1     & \textbf{5} & \textbf{5} & \textbf{5} & \textbf{5} & \textbf{5} & \textbf{5} & \textbf{5} \\
    \multicolumn{1}{r|}{} & \multicolumn{1}{r|}{} & \multicolumn{1}{r|}{} & 3     & 0     & 0     & 1     & \textbf{4} & \textbf{4} & \textbf{4} & \textbf{4} \\
    \multicolumn{1}{r|}{} & \multicolumn{1}{r|}{} & \multicolumn{1}{r|}{} & 5     & 0     & 0     & 0     & 2     & 3     & \textbf{4} & \textbf{4} \\
\cmidrule{2-11}    \multicolumn{1}{r|}{} & \multicolumn{1}{r|}{\multirow{6}[2]{*}{100}} & \multicolumn{1}{r|}{\multirow{3}[1]{*}{10}} & 1     & \textbf{5} & \textbf{5} & \textbf{5} & \textbf{5} & \textbf{5} & \textbf{5} & \textbf{5} \\
    \multicolumn{1}{r|}{} & \multicolumn{1}{r|}{} & \multicolumn{1}{r|}{} & 3     & \textbf{5} & \textbf{5} & \textbf{5} & \textbf{5} & \textbf{5} & \textbf{5} & \textbf{5} \\
    \multicolumn{1}{r|}{} & \multicolumn{1}{r|}{} & \multicolumn{1}{r|}{} & 5     & \textbf{5} & \textbf{5} & \textbf{5} & \textbf{5} & \textbf{5} & \textbf{5} & \textbf{5} \\
    \cmidrule{3-11}\multicolumn{1}{r|}{} & \multicolumn{1}{r|}{} & \multicolumn{1}{r|}{\multirow{3}[1]{*}{20}} & 1     & 4     & 3     & 4     & \textbf{5} & \textbf{5} & \textbf{5} & \textbf{5} \\
    \multicolumn{1}{r|}{} & \multicolumn{1}{r|}{} & \multicolumn{1}{r|}{} & 3     & 0     & 0     & 0     & \textbf{1} & \textbf{1} & \textbf{1} & \textbf{1} \\
    \multicolumn{1}{r|}{} & \multicolumn{1}{r|}{} & \multicolumn{1}{r|}{} & 5     & 0     & 0     & 0     & 0     & 0     & 0     & 0 \\
\cmidrule{2-11}    \multicolumn{1}{r|}{} & \multicolumn{1}{r|}{\multirow{6}[2]{*}{125}} & \multicolumn{1}{r|}{\multirow{3}[1]{*}{10}} & 1     & \textbf{5} & \textbf{5} & \textbf{5} & \textbf{5} & \textbf{5} & \textbf{5} & \textbf{5} \\
    \multicolumn{1}{r|}{} & \multicolumn{1}{r|}{} & \multicolumn{1}{r|}{} & 3     & \textbf{5} & \textbf{5} & \textbf{5} & \textbf{5} & \textbf{5} & \textbf{5} & \textbf{5} \\
    \multicolumn{1}{r|}{} & \multicolumn{1}{r|}{} & \multicolumn{1}{r|}{} & 5     & \textbf{5} & \textbf{5} & \textbf{5} & \textbf{5} & \textbf{5} & \textbf{5} & \textbf{5} \\
    \cmidrule{3-11}\multicolumn{1}{r|}{} & \multicolumn{1}{r|}{} & \multicolumn{1}{r|}{\multirow{3}[1]{*}{20}} & 1     & \textbf{5} & 4     & \textbf{5} & \textbf{5} & \textbf{5} & 4     & \textbf{5} \\
    \multicolumn{1}{r|}{} & \multicolumn{1}{r|}{} & \multicolumn{1}{r|}{} & 3     & 0     & 0     & 0     & 0     & 0     & 0     & 0 \\
    \multicolumn{1}{r|}{} & \multicolumn{1}{r|}{} & \multicolumn{1}{r|}{} & 5     & 0     & 0     & 0     & 0     & 0     & 0     & 0 \\
    \midrule
    \multicolumn{4}{r|}{Total}    & 137   & 135   & 149   & 162   & 167   & \textbf{172} & 168 \\
    \bottomrule
    \end{tabular}%
 \label{tab:cc_solve}
\end{table}

\begin{table}[htb]
	\centering
	\scriptsize
	\setlength{\tabcolsep}{4pt}
	\caption{Root Gap comparison across all techniques for SOC-CC benchmark.}

    \begin{tabular}{rrrr|rrrrrrr}
    \toprule
          &       &       & \multicolumn{1}{r}{} & \multicolumn{7}{c}{\textbf{Root Gap (\%)}} \\
    \midrule
    \multicolumn{1}{r|}{$t$} & \multicolumn{1}{r|}{$n$} & \multicolumn{1}{r|}{$m$} & $\Omega$     & \cplex & \bddFlow    & \bddFlowLift   & \bddGeneral    & \bddGeneralLift   & \bddTarget    & \bddTargetFlow \\
    \midrule
    \multicolumn{1}{r|}{\multirow{18}[6]{*}{0.1}} & \multicolumn{1}{r|}{\multirow{6}[2]{*}{75}} & \multicolumn{1}{r|}{\multirow{3}[1]{*}{10}} & 1     & 7.6\% & 10.0\% & 6.9\% & 3.8\% & 3.8\% & \textbf{3.5\%} & 3.7\% \\
    \multicolumn{1}{r|}{} & \multicolumn{1}{r|}{} & \multicolumn{1}{r|}{} & 3     & 47.6\% & 45.4\% & 20.9\% & 16.4\% & 16.0\% & \textbf{15.6\%} & 16.0\% \\
    \multicolumn{1}{r|}{} & \multicolumn{1}{r|}{} & \multicolumn{1}{r|}{} & 5     & 56.1\% & 47.3\% & 22.0\% & 17.8\% & 17.4\% & 17.9\% & \textbf{17.3\%} \\
    \cmidrule{3-11} \multicolumn{1}{r|}{} & \multicolumn{1}{r|}{} & \multicolumn{1}{r|}{\multirow{3}[1]{*}{20}} & 1     & 54.7\% & 56.5\% & 54.9\% & \textbf{47.0\%} & 47.1\% & 47.1\% & 47.1\% \\
    \multicolumn{1}{r|}{} & \multicolumn{1}{r|}{} & \multicolumn{1}{r|}{} & 3     & 83.0\% & 75.7\% & 31.1\% & 25.3\% & 23.2\% & 30.0\% & \textbf{23.1\%} \\
    \multicolumn{1}{r|}{} & \multicolumn{1}{r|}{} & \multicolumn{1}{r|}{} & 5     & 81.3\% & 26.0\% & 8.2\% & \textbf{5.9\%} & 8.2\% & 16.6\% & 8.2\% \\
\cmidrule{2-11}    \multicolumn{1}{r|}{} & \multicolumn{1}{r|}{\multirow{6}[2]{*}{100}} & \multicolumn{1}{r|}{\multirow{3}[1]{*}{10}} & 1     & 6.1\% & 7.5\% & 6.5\% & 5.4\% & \textbf{5.3\%} & \textbf{5.3\%} & \textbf{5.3\%} \\
    \multicolumn{1}{r|}{} & \multicolumn{1}{r|}{} & \multicolumn{1}{r|}{} & 3     & 34.3\% & 34.0\% & 27.3\% & 23.3\% & \textbf{23.2\%} & 23.3\% & \textbf{23.2\%} \\
    \multicolumn{1}{r|}{} & \multicolumn{1}{r|}{} & \multicolumn{1}{r|}{} & 5     & 40.8\% & 39.7\% & 27.1\% & \textbf{22.4\%} & 22.6\% & \textbf{22.4\%} & \textbf{22.4\%} \\
    \cmidrule{3-11} \multicolumn{1}{r|}{} & \multicolumn{1}{r|}{} & \multicolumn{1}{r|}{\multirow{3}[1]{*}{20}} & 1     & 26.9\% & 28.8\% & 26.7\% & \textbf{23.2\%} & \textbf{23.2\%} & \textbf{23.2\%} & \textbf{23.2\%} \\
    \multicolumn{1}{r|}{} & \multicolumn{1}{r|}{} & \multicolumn{1}{r|}{} & 3     & 65.3\% & 62.9\% & 48.6\% & \textbf{46.0\%} & \textbf{46.0\%} & \textbf{46.0\%} & \textbf{46.0\%} \\
    \multicolumn{1}{r|}{} & \multicolumn{1}{r|}{} & \multicolumn{1}{r|}{} & 5     & 67.5\% & 59.6\% & 42.5\% & \textbf{39.7\%} & 40.0\% & \textbf{39.7\%} & 39.9\% \\
\cmidrule{2-11}    \multicolumn{1}{r|}{} & \multicolumn{1}{r|}{\multirow{6}[2]{*}{125}} & \multicolumn{1}{r|}{\multirow{3}[1]{*}{10}} & 1     & 4.4\% & 5.5\% & 4.7\% & \textbf{4.1\%} & \textbf{4.1\%} & \textbf{4.1\%} & \textbf{4.1\%} \\
    \multicolumn{1}{r|}{} & \multicolumn{1}{r|}{} & \multicolumn{1}{r|}{} & 3     & 33.5\% & 33.3\% & 30.2\% & \textbf{29.2\%} & \textbf{29.2\%} & \textbf{29.2\%} & \textbf{29.2\%} \\
    \multicolumn{1}{r|}{} & \multicolumn{1}{r|}{} & \multicolumn{1}{r|}{} & 5     & 38.0\% & 36.4\% & 32.1\% & 31.2\% & 31.2\% & \textbf{31.1\%} & 31.2\% \\
    \cmidrule{3-11} \multicolumn{1}{r|}{} & \multicolumn{1}{r|}{} & \multicolumn{1}{r|}{\multirow{3}[1]{*}{20}} & 1     & 18.9\% & 19.8\% & 19.3\% & \textbf{18.0\%} & \textbf{18.0\%} & \textbf{18.0\%} & \textbf{18.0\%} \\
    \multicolumn{1}{r|}{} & \multicolumn{1}{r|}{} & \multicolumn{1}{r|}{} & 3     & 56.4\% & 56.4\% & 53.0\% & 47.8\% & \textbf{47.7\%} & 47.8\% & \textbf{47.7\%} \\
    \multicolumn{1}{r|}{} & \multicolumn{1}{r|}{} & \multicolumn{1}{r|}{} & 5     & 56.2\% & 55.0\% & 43.0\% & \textbf{40.4\%} & \textbf{40.4\%} & \textbf{40.4\%} & \textbf{40.4\%} \\
    \midrule
    \cmidrule{3-11} \multicolumn{1}{r|}{\multirow{18}[6]{*}{0.2}} & \multicolumn{1}{r|}{\multirow{6}[2]{*}{75}} & \multicolumn{1}{r|}{\multirow{3}[1]{*}{10}} & 1     & 2.9\% & 4.5\% & 3.6\% & \textbf{2.0\%} & \textbf{2.0\%} & \textbf{2.0\%} & \textbf{2.0\%} \\
    \multicolumn{1}{r|}{} & \multicolumn{1}{r|}{} & \multicolumn{1}{r|}{} & 3     & 10.1\% & 10.7\% & 9.0\% & \textbf{5.4\%} & \textbf{5.4\%} & \textbf{5.4\%} & \textbf{5.4\%} \\
    \multicolumn{1}{r|}{} & \multicolumn{1}{r|}{} & \multicolumn{1}{r|}{} & 5     & 13.1\% & 13.8\% & 11.1\% & \textbf{7.0\%} & \textbf{7.0\%} & \textbf{7.0\%} & \textbf{7.0\%} \\
    \cmidrule{3-11} \multicolumn{1}{r|}{} & \multicolumn{1}{r|}{} & \multicolumn{1}{r|}{\multirow{3}[1]{*}{20}} & 1     & 11.6\% & 13.9\% & 11.4\% & \textbf{7.8\%} & \textbf{7.8\%} & \textbf{7.8\%} & \textbf{7.8\%} \\
    \multicolumn{1}{r|}{} & \multicolumn{1}{r|}{} & \multicolumn{1}{r|}{} & 3     & 30.6\% & 31.4\% & 30.2\% & \textbf{21.0\%} & \textbf{21.0\%} & \textbf{21.0\%} & \textbf{21.0\%} \\
    \multicolumn{1}{r|}{} & \multicolumn{1}{r|}{} & \multicolumn{1}{r|}{} & 5     & 32.5\% & 32.9\% & 29.2\% & \textbf{20.5\%} & \textbf{20.5\%} & \textbf{20.5\%} & \textbf{20.5\%} \\
\cmidrule{2-11}    \multicolumn{1}{r|}{} & \multicolumn{1}{r|}{\multirow{6}[2]{*}{100}} & \multicolumn{1}{r|}{\multirow{3}[1]{*}{10}} & 1     & 2.6\% & 3.8\% & 3.1\% & \textbf{2.4\%} & \textbf{2.4\%} & \textbf{2.4\%} & \textbf{2.4\%} \\
    \multicolumn{1}{r|}{} & \multicolumn{1}{r|}{} & \multicolumn{1}{r|}{} & 3     & 8.3\% & 8.8\% & 8.0\% & \textbf{6.2\%} & \textbf{6.2\%} & 6.3\% & \textbf{6.2\%} \\
    \multicolumn{1}{r|}{} & \multicolumn{1}{r|}{} & \multicolumn{1}{r|}{} & 5     & 11.8\% & 12.1\% & 11.5\% & \textbf{9.8\%} & \textbf{9.8\%} & \textbf{9.8\%} & \textbf{9.8\%} \\
    \cmidrule{3-11} \multicolumn{1}{r|}{} & \multicolumn{1}{r|}{} & \multicolumn{1}{r|}{\multirow{3}[1]{*}{20}} & 1     & 7.7\% & 9.2\% & 8.6\% & \textbf{6.5\%} & \textbf{6.5\%} & \textbf{6.5\%} & \textbf{6.5\%} \\
    \multicolumn{1}{r|}{} & \multicolumn{1}{r|}{} & \multicolumn{1}{r|}{} & 3     & 24.6\% & 24.9\% & 24.1\% & \textbf{20.7\%} & \textbf{20.7\%} & \textbf{20.7\%} & \textbf{20.7\%} \\
    \multicolumn{1}{r|}{} & \multicolumn{1}{r|}{} & \multicolumn{1}{r|}{} & 5     & 30.1\% & 30.2\% & 29.1\% & \textbf{25.6\%} & \textbf{25.6\%} & \textbf{25.6\%} & \textbf{25.6\%} \\
\cmidrule{2-11}    \multicolumn{1}{r|}{} & \multicolumn{1}{r|}{\multirow{6}[2]{*}{125}} & \multicolumn{1}{r|}{\multirow{3}[1]{*}{10}} & 1     & \textbf{1.6\%} & 2.8\% & 2.6\% & 2.4\% & 2.4\% & 2.4\% & 2.4\% \\
    \multicolumn{1}{r|}{} & \multicolumn{1}{r|}{} & \multicolumn{1}{r|}{} & 3     & 7.7\% & 7.9\% & 7.7\% & 7.4\% & 7.4\% & \textbf{7.3\%} & 7.4\% \\
    \multicolumn{1}{r|}{} & \multicolumn{1}{r|}{} & \multicolumn{1}{r|}{} & 5     & 9.4\% & 9.5\% & 9.2\% & \textbf{8.6\%} & \textbf{8.6\%} & \textbf{8.6\%} & \textbf{8.6\%} \\
    \cmidrule{3-11} \multicolumn{1}{r|}{} & \multicolumn{1}{r|}{} & \multicolumn{1}{r|}{\multirow{3}[1]{*}{20}} & 1     & 5.4\% & 6.3\% & 6.2\% & \textbf{5.5\%} & \textbf{5.5\%} & \textbf{5.5\%} & \textbf{5.5\%} \\
    \multicolumn{1}{r|}{} & \multicolumn{1}{r|}{} & \multicolumn{1}{r|}{} & 3     & 20.6\% & 20.7\% & 20.5\% & \textbf{19.0\%} & \textbf{19.0\%} & \textbf{19.0\%} & \textbf{19.0\%} \\
    \multicolumn{1}{r|}{} & \multicolumn{1}{r|}{} & \multicolumn{1}{r|}{} & 5     & 25.1\% & 25.2\% & 24.9\% & \textbf{22.9\%} & \textbf{22.9\%} & \textbf{22.9\%} & \textbf{22.9\%} \\
    \midrule
    \multicolumn{1}{r|}{\multirow{18}[6]{*}{0.3}} & \multicolumn{1}{r|}{\multirow{6}[2]{*}{75}} & \multicolumn{1}{r|}{\multirow{3}[1]{*}{10}} & 1     & 1.2\% & 2.7\% & 1.9\% & \textbf{0.9\%} & \textbf{0.9\%} & \textbf{0.9\%} & \textbf{0.9\%} \\
    \multicolumn{1}{r|}{} & \multicolumn{1}{r|}{} & \multicolumn{1}{r|}{} & 3     & 3.6\% & 4.4\% & 3.4\% & \textbf{1.8\%} & \textbf{1.8\%} & \textbf{1.8\%} & \textbf{1.8\%} \\
    \multicolumn{1}{r|}{} & \multicolumn{1}{r|}{} & \multicolumn{1}{r|}{} & 5     & 4.3\% & 5.0\% & 4.0\% & \textbf{2.1\%} & \textbf{2.1\%} & \textbf{2.1\%} & \textbf{2.1\%} \\
    \cmidrule{3-11} \multicolumn{1}{r|}{} & \multicolumn{1}{r|}{} & \multicolumn{1}{r|}{\multirow{3}[1]{*}{20}} & 1     & 3.7\% & 5.5\% & 4.4\% & 2.4\% & 2.4\% & \textbf{2.3\%} & 2.4\% \\
    \multicolumn{1}{r|}{} & \multicolumn{1}{r|}{} & \multicolumn{1}{r|}{} & 3     & 7.7\% & 9.0\% & 7.5\% & \textbf{4.3\%} & \textbf{4.3\%} & \textbf{4.3\%} & \textbf{4.3\%} \\
    \multicolumn{1}{r|}{} & \multicolumn{1}{r|}{} & \multicolumn{1}{r|}{} & 5     & 8.8\% & 9.8\% & 8.0\% & \textbf{4.8\%} & \textbf{4.8\%} & \textbf{4.8\%} & \textbf{4.8\%} \\
\cmidrule{2-11}    \multicolumn{1}{r|}{} & \multicolumn{1}{r|}{\multirow{6}[2]{*}{100}} & \multicolumn{1}{r|}{\multirow{3}[1]{*}{10}} & 1     & \textbf{1.0\%} & 2.1\% & 1.7\% & 1.2\% & 1.2\% & 1.2\% & 1.2\% \\
    \multicolumn{1}{r|}{} & \multicolumn{1}{r|}{} & \multicolumn{1}{r|}{} & 3     & 2.3\% & 3.1\% & 2.8\% & \textbf{2.1\%} & \textbf{2.1\%} & \textbf{2.1\%} & \textbf{2.1\%} \\
    \multicolumn{1}{r|}{} & \multicolumn{1}{r|}{} & \multicolumn{1}{r|}{} & 5     & 2.9\% & 3.6\% & 3.5\% & \textbf{2.6\%} & \textbf{2.6\%} & \textbf{2.6\%} & \textbf{2.6\%} \\
    \cmidrule{3-11} \multicolumn{1}{r|}{} & \multicolumn{1}{r|}{} & \multicolumn{1}{r|}{\multirow{3}[1]{*}{20}} & 1     & 3.3\% & 4.6\% & 3.8\% & \textbf{3.1\%} & \textbf{3.1\%} & \textbf{3.1\%} & \textbf{3.1\%} \\
    \multicolumn{1}{r|}{} & \multicolumn{1}{r|}{} & \multicolumn{1}{r|}{} & 3     & 6.0\% & 6.7\% & 6.0\% & \textbf{4.6\%} & \textbf{4.6\%} & \textbf{4.6\%} & \textbf{4.6\%} \\
    \multicolumn{1}{r|}{} & \multicolumn{1}{r|}{} & \multicolumn{1}{r|}{} & 5     & 7.9\% & 8.6\% & 7.9\% & \textbf{6.1\%} & \textbf{6.1\%} & \textbf{6.1\%} & \textbf{6.1\%} \\
\cmidrule{2-11}    \multicolumn{1}{r|}{} & \multicolumn{1}{r|}{\multirow{6}[2]{*}{125}} & \multicolumn{1}{r|}{\multirow{3}[1]{*}{10}} & 1     & \textbf{0.3\%} & 1.1\% & 1.0\% & 0.9\% & 0.9\% & 0.9\% & 0.9\% \\
    \multicolumn{1}{r|}{} & \multicolumn{1}{r|}{} & \multicolumn{1}{r|}{} & 3     & \textbf{1.6\%} & 2.2\% & 2.1\% & 1.9\% & 1.9\% & 1.9\% & 1.9\% \\
    \multicolumn{1}{r|}{} & \multicolumn{1}{r|}{} & \multicolumn{1}{r|}{} & 5     & \textbf{1.8\%} & 2.3\% & 2.0\% & 2.0\% & 2.0\% & 2.0\% & 2.0\% \\
    \cmidrule{3-11} \multicolumn{1}{r|}{} & \multicolumn{1}{r|}{} & \multicolumn{1}{r|}{\multirow{3}[1]{*}{20}} & 1     & \textbf{1.7\%} & 2.8\% & 2.6\% & 2.4\% & 2.4\% & 2.4\% & 2.4\% \\
    \multicolumn{1}{r|}{} & \multicolumn{1}{r|}{} & \multicolumn{1}{r|}{} & 3     & 4.8\% & 5.4\% & 5.2\% & \textbf{4.5\%} & \textbf{4.5\%} & \textbf{4.5\%} & \textbf{4.5\%} \\
    \multicolumn{1}{r|}{} & \multicolumn{1}{r|}{} & \multicolumn{1}{r|}{} & 5     & 6.7\% & 7.2\% & 7.0\% & \textbf{6.0\%} & 6.1\% & \textbf{6.0\%} & 6.1\% \\
    \midrule
    \multicolumn{4}{r|}{Average}  & 20.4\% & 19.5\% & 15.4\% & \textbf{13.0\%} & \textbf{13.0\%} & 13.3\% & \textbf{13.0\%} \\
    \bottomrule
    \end{tabular}%

 \label{tab:cc_root_gap}
\end{table}

\begin{table}[htb]
	\centering
	\scriptsize
	\setlength{\tabcolsep}{4pt}
	\caption{Final Gap comparison across all techniques for SOC-CC benchmark.}
	\begin{tabular}{rrrr|rrrrrrr}
    \toprule
          &       &       & \multicolumn{1}{r}{} & \multicolumn{7}{c}{\textbf{Final Gap (\%)}} \\
    \midrule
    \multicolumn{1}{r|}{$t$} & \multicolumn{1}{r|}{$n$} & \multicolumn{1}{r|}{$m$} & $\Omega$     & \cplex & \bddFlow    & \bddFlowLift   & \bddGeneral    & \bddGeneralLift   & \bddTarget    & \bddTargetFlow \\
    \midrule
    \multicolumn{1}{r|}{\multirow{18}[6]{*}{0.1}} & \multicolumn{1}{r|}{\multirow{6}[2]{*}{75}} & \multicolumn{1}{r|}{\multirow{3}[1]{*}{10}} & 1     & \textbf{0.0\%} & \textbf{0.0\%} & \textbf{0.0\%} & \textbf{0.0\%} & \textbf{0.0\%} & \textbf{0.0\%} & \textbf{0.0\%} \\
    \multicolumn{1}{r|}{} & \multicolumn{1}{r|}{} & \multicolumn{1}{r|}{} & 3     & \textbf{0.0\%} & \textbf{0.0\%} & \textbf{0.0\%} & \textbf{0.0\%} & \textbf{0.0\%} & \textbf{0.0\%} & \textbf{0.0\%} \\
    \multicolumn{1}{r|}{} & \multicolumn{1}{r|}{} & \multicolumn{1}{r|}{} & 5     & \textbf{0.0\%} & \textbf{0.0\%} & \textbf{0.0\%} & \textbf{0.0\%} & \textbf{0.0\%} & \textbf{0.0\%} & \textbf{0.0\%} \\
    \cmidrule{3-11} \multicolumn{1}{r|}{} & \multicolumn{1}{r|}{} & \multicolumn{1}{r|}{\multirow{3}[1]{*}{20}} & 1     & \textbf{0.0\%} & \textbf{0.0\%} & \textbf{0.0\%} & \textbf{0.0\%} & \textbf{0.0\%} & \textbf{0.0\%} & \textbf{0.0\%} \\
    \multicolumn{1}{r|}{} & \multicolumn{1}{r|}{} & \multicolumn{1}{r|}{} & 3     & \textbf{0.0\%} & \textbf{0.0\%} & \textbf{0.0\%} & \textbf{0.0\%} & \textbf{0.0\%} & \textbf{0.0\%} & \textbf{0.0\%} \\
    \multicolumn{1}{r|}{} & \multicolumn{1}{r|}{} & \multicolumn{1}{r|}{} & 5     & \textbf{0.0\%} & \textbf{0.0\%} & \textbf{0.0\%} & \textbf{0.0\%} & \textbf{0.0\%} & \textbf{0.0\%} & \textbf{0.0\%} \\
\cmidrule{2-11}    \multicolumn{1}{r|}{} & \multicolumn{1}{r|}{\multirow{6}[2]{*}{100}} & \multicolumn{1}{r|}{\multirow{3}[1]{*}{10}} & 1     & \textbf{0.0\%} & \textbf{0.0\%} & \textbf{0.0\%} & \textbf{0.0\%} & \textbf{0.0\%} & \textbf{0.0\%} & \textbf{0.0\%} \\
    \multicolumn{1}{r|}{} & \multicolumn{1}{r|}{} & \multicolumn{1}{r|}{} & 3     & \textbf{0.0\%} & \textbf{0.0\%} & \textbf{0.0\%} & \textbf{0.0\%} & \textbf{0.0\%} & \textbf{0.0\%} & \textbf{0.0\%} \\
    \multicolumn{1}{r|}{} & \multicolumn{1}{r|}{} & \multicolumn{1}{r|}{} & 5     & \textbf{0.0\%} & \textbf{0.0\%} & \textbf{0.0\%} & \textbf{0.0\%} & \textbf{0.0\%} & \textbf{0.0\%} & \textbf{0.0\%} \\
    \cmidrule{3-11} \multicolumn{1}{r|}{} & \multicolumn{1}{r|}{} & \multicolumn{1}{r|}{\multirow{3}[1]{*}{20}} & 1     & 17.4\% & 18.4\% & 12.4\% & 7.5\% & 7.5\% & \textbf{5.1\%} & 5.5\% \\
    \multicolumn{1}{r|}{} & \multicolumn{1}{r|}{} & \multicolumn{1}{r|}{} & 3     & \textbf{0.0\%} & \textbf{0.0\%} & \textbf{0.0\%} & \textbf{0.0\%} & \textbf{0.0\%} & \textbf{0.0\%} & \textbf{0.0\%} \\
    \multicolumn{1}{r|}{} & \multicolumn{1}{r|}{} & \multicolumn{1}{r|}{} & 5     & \textbf{0.0\%} & \textbf{0.0\%} & \textbf{0.0\%} & \textbf{0.0\%} & \textbf{0.0\%} & \textbf{0.0\%} & \textbf{0.0\%} \\
\cmidrule{2-11}    \multicolumn{1}{r|}{} & \multicolumn{1}{r|}{\multirow{6}[2]{*}{125}} & \multicolumn{1}{r|}{\multirow{3}[1]{*}{10}} & 1     & 0.3\% & 0.6\% & 0.7\% & 0.3\% & 0.4\% & \textbf{0.0\%} & 0.2\% \\
    \multicolumn{1}{r|}{} & \multicolumn{1}{r|}{} & \multicolumn{1}{r|}{} & 3     & 25.1\% & 22.8\% & 17.6\% & 15.0\% & 16.1\% & \textbf{14.5\%} & 16.0\% \\
    \multicolumn{1}{r|}{} & \multicolumn{1}{r|}{} & \multicolumn{1}{r|}{} & 5     & 28.0\% & 23.7\% & 13.0\% & 11.3\% & \textbf{11.3\%} & 11.5\% & \textbf{11.3\%} \\
    \cmidrule{3-11} \multicolumn{1}{r|}{} & \multicolumn{1}{r|}{} & \multicolumn{1}{r|}{\multirow{3}[1]{*}{20}} & 1     & 16.0\% & 16.9\% & 16.5\% & 14.6\% & 14.4\% & 14.4\% & \textbf{14.2\%} \\
    \multicolumn{1}{r|}{} & \multicolumn{1}{r|}{} & \multicolumn{1}{r|}{} & 3     & 52.0\% & 50.5\% & 32.3\% & 22.5\% & 21.8\% & \textbf{19.8\%} & 22.1\% \\
    \multicolumn{1}{r|}{} & \multicolumn{1}{r|}{} & \multicolumn{1}{r|}{} & 5     & 49.4\% & 28.3\% & 10.4\% & 6.7\% & 4.1\% & \textbf{2.5\%} & 6.9\% \\
    \midrule
    \cmidrule{3-11} \multicolumn{1}{r|}{\multirow{18}[6]{*}{0.2}} & \multicolumn{1}{r|}{\multirow{6}[2]{*}{75}} & \multicolumn{1}{r|}{\multirow{3}[1]{*}{10}} & 1     & \textbf{0.0\%} & \textbf{0.0\%} & \textbf{0.0\%} & \textbf{0.0\%} & \textbf{0.0\%} & \textbf{0.0\%} & \textbf{0.0\%} \\
    \multicolumn{1}{r|}{} & \multicolumn{1}{r|}{} & \multicolumn{1}{r|}{} & 3     & 1.4\% & 0.6\% & 0.1\% & \textbf{0.0\%} & \textbf{0.0\%} & \textbf{0.0\%} & \textbf{0.0\%} \\
    \multicolumn{1}{r|}{} & \multicolumn{1}{r|}{} & \multicolumn{1}{r|}{} & 5     & 4.5\% & 3.6\% & 1.2\% & \textbf{0.0\%} & \textbf{0.0\%} & \textbf{0.0\%} & \textbf{0.0\%} \\
    \cmidrule{3-11} \multicolumn{1}{r|}{} & \multicolumn{1}{r|}{} & \multicolumn{1}{r|}{\multirow{3}[1]{*}{20}} & 1     & 3.2\% & 3.8\% & 2.5\% & 1.4\% & \textbf{0.5\%} & 0.6\% & 0.7\% \\
    \multicolumn{1}{r|}{} & \multicolumn{1}{r|}{} & \multicolumn{1}{r|}{} & 3     & 24.5\% & 23.7\% & 22.7\% & 13.6\% & 12.9\% & 13.6\% & \textbf{12.5\%} \\
    \multicolumn{1}{r|}{} & \multicolumn{1}{r|}{} & \multicolumn{1}{r|}{} & 5     & 25.1\% & 25.8\% & 21.9\% & 12.5\% & 12.3\% & 12.1\% & \textbf{12.0\%} \\
\cmidrule{2-11}    \multicolumn{1}{r|}{} & \multicolumn{1}{r|}{\multirow{6}[2]{*}{100}} & \multicolumn{1}{r|}{\multirow{3}[1]{*}{10}} & 1     & \textbf{0.0\%} & \textbf{0.0\%} & \textbf{0.0\%} & \textbf{0.0\%} & \textbf{0.0\%} & \textbf{0.0\%} & \textbf{0.0\%} \\
    \multicolumn{1}{r|}{} & \multicolumn{1}{r|}{} & \multicolumn{1}{r|}{} & 3     & 4.4\% & 4.5\% & 3.1\% & 1.8\% & 1.3\% & \textbf{1.3\%} & 1.5\% \\
    \multicolumn{1}{r|}{} & \multicolumn{1}{r|}{} & \multicolumn{1}{r|}{} & 5     & 8.4\% & 8.9\% & 8.0\% & 6.2\% & 6.1\% & \textbf{6.0\%} & \textbf{6.0\%} \\
    \cmidrule{3-11} \multicolumn{1}{r|}{} & \multicolumn{1}{r|}{} & \multicolumn{1}{r|}{\multirow{3}[1]{*}{20}} & 1     & 3.6\% & 4.6\% & 3.8\% & 2.1\% & 2.0\% & \textbf{1.7\%} & 1.9\% \\
    \multicolumn{1}{r|}{} & \multicolumn{1}{r|}{} & \multicolumn{1}{r|}{} & 3     & 21.3\% & 22.3\% & 21.6\% & 17.6\% & 17.4\% & \textbf{17.2\%} & 17.3\% \\
    \multicolumn{1}{r|}{} & \multicolumn{1}{r|}{} & \multicolumn{1}{r|}{} & 5     & 27.3\% & 28.0\% & 26.5\% & 22.7\% & 22.5\% & 22.5\% & \textbf{22.4\%} \\
\cmidrule{2-11}    \multicolumn{1}{r|}{} & \multicolumn{1}{r|}{\multirow{6}[2]{*}{125}} & \multicolumn{1}{r|}{\multirow{3}[1]{*}{10}} & 1     & \textbf{0.0\%} & \textbf{0.0\%} & \textbf{0.0\%} & \textbf{0.0\%} & \textbf{0.0\%} & \textbf{0.0\%} & \textbf{0.0\%} \\
    \multicolumn{1}{r|}{} & \multicolumn{1}{r|}{} & \multicolumn{1}{r|}{} & 3     & 5.4\% & 5.7\% & 5.4\% & 5.1\% & \textbf{5.0\%} & 5.1\% & 5.1\% \\
    \multicolumn{1}{r|}{} & \multicolumn{1}{r|}{} & \multicolumn{1}{r|}{} & 5     & 7.2\% & 7.7\% & 7.7\% & 7.2\% & 7.1\% & \textbf{7.0\%} & 7.1\% \\
    \cmidrule{3-11} \multicolumn{1}{r|}{} & \multicolumn{1}{r|}{} & \multicolumn{1}{r|}{\multirow{3}[1]{*}{20}} & 1     & 3.6\% & 4.3\% & 4.1\% & 3.5\% & \textbf{3.2\%} & 3.3\% & 3.5\% \\
    \multicolumn{1}{r|}{} & \multicolumn{1}{r|}{} & \multicolumn{1}{r|}{} & 3     & 19.2\% & 19.9\% & 19.7\% & 17.8\% & 17.8\% & \textbf{17.7\%} & 17.8\% \\
    \multicolumn{1}{r|}{} & \multicolumn{1}{r|}{} & \multicolumn{1}{r|}{} & 5     & 23.7\% & 24.5\% & 24.2\% & 21.8\% & \textbf{21.7\%} & \textbf{21.7\%} & \textbf{21.7\%} \\
    \midrule
    \multicolumn{1}{r|}{\multirow{18}[6]{*}{0.3}} & \multicolumn{1}{r|}{\multirow{6}[2]{*}{75}} & \multicolumn{1}{r|}{\multirow{3}[1]{*}{10}} & 1     & \textbf{0.0\%} & \textbf{0.0\%} & \textbf{0.0\%} & \textbf{0.0\%} & \textbf{0.0\%} & \textbf{0.0\%} & \textbf{0.0\%} \\
    \multicolumn{1}{r|}{} & \multicolumn{1}{r|}{} & \multicolumn{1}{r|}{} & 3     & \textbf{0.0\%} & \textbf{0.0\%} & \textbf{0.0\%} & \textbf{0.0\%} & \textbf{0.0\%} & \textbf{0.0\%} & \textbf{0.0\%} \\
    \multicolumn{1}{r|}{} & \multicolumn{1}{r|}{} & \multicolumn{1}{r|}{} & 5     & 0.3\% & 0.4\% & \textbf{0.0\%} & \textbf{0.0\%} & \textbf{0.0\%} & \textbf{0.0\%} & \textbf{0.0\%} \\
    \cmidrule{3-11} \multicolumn{1}{r|}{} & \multicolumn{1}{r|}{} & \multicolumn{1}{r|}{\multirow{3}[1]{*}{20}} & 1     & \textbf{0.0\%} & \textbf{0.0\%} & \textbf{0.0\%} & \textbf{0.0\%} & \textbf{0.0\%} & \textbf{0.0\%} & \textbf{0.0\%} \\
    \multicolumn{1}{r|}{} & \multicolumn{1}{r|}{} & \multicolumn{1}{r|}{} & 3     & 3.1\% & 3.5\% & 2.1\% & 0.6\% & 0.4\% & \textbf{0.2\%} & \textbf{0.2\%} \\
    \multicolumn{1}{r|}{} & \multicolumn{1}{r|}{} & \multicolumn{1}{r|}{} & 5     & 4.8\% & 5.4\% & 3.5\% & 1.0\% & 0.5\% & 0.5\% & \textbf{0.4\%} \\
\cmidrule{2-11}    \multicolumn{1}{r|}{} & \multicolumn{1}{r|}{\multirow{6}[2]{*}{100}} & \multicolumn{1}{r|}{\multirow{3}[1]{*}{10}} & 1     & \textbf{0.0\%} & \textbf{0.0\%} & \textbf{0.0\%} & \textbf{0.0\%} & \textbf{0.0\%} & \textbf{0.0\%} & \textbf{0.0\%} \\
    \multicolumn{1}{r|}{} & \multicolumn{1}{r|}{} & \multicolumn{1}{r|}{} & 3     & \textbf{0.0\%} & \textbf{0.0\%} & \textbf{0.0\%} & \textbf{0.0\%} & \textbf{0.0\%} & \textbf{0.0\%} & \textbf{0.0\%} \\
    \multicolumn{1}{r|}{} & \multicolumn{1}{r|}{} & \multicolumn{1}{r|}{} & 5     & \textbf{0.0\%} & \textbf{0.0\%} & \textbf{0.0\%} & \textbf{0.0\%} & \textbf{0.0\%} & \textbf{0.0\%} & \textbf{0.0\%} \\
    \cmidrule{3-11} \multicolumn{1}{r|}{} & \multicolumn{1}{r|}{} & \multicolumn{1}{r|}{\multirow{3}[1]{*}{20}} & 1     & 0.2\% & 0.4\% & 0.3\% & \textbf{0.0\%} & \textbf{0.0\%} & \textbf{0.0\%} & \textbf{0.0\%} \\
    \multicolumn{1}{r|}{} & \multicolumn{1}{r|}{} & \multicolumn{1}{r|}{} & 3     & 3.4\% & 3.8\% & 3.3\% & 1.9\% & \textbf{1.8\%} & \textbf{1.8\%} & 1.9\% \\
    \multicolumn{1}{r|}{} & \multicolumn{1}{r|}{} & \multicolumn{1}{r|}{} & 5     & 5.9\% & 6.2\% & 5.7\% & 4.0\% & \textbf{3.9\%} & \textbf{3.9\%} & \textbf{3.9\%} \\
\cmidrule{2-11}    \multicolumn{1}{r|}{} & \multicolumn{1}{r|}{\multirow{6}[2]{*}{125}} & \multicolumn{1}{r|}{\multirow{3}[1]{*}{10}} & 1     & \textbf{0.0\%} & \textbf{0.0\%} & \textbf{0.0\%} & \textbf{0.0\%} & \textbf{0.0\%} & \textbf{0.0\%} & \textbf{0.0\%} \\
    \multicolumn{1}{r|}{} & \multicolumn{1}{r|}{} & \multicolumn{1}{r|}{} & 3     & \textbf{0.0\%} & \textbf{0.0\%} & \textbf{0.0\%} & \textbf{0.0\%} & \textbf{0.0\%} & \textbf{0.0\%} & \textbf{0.0\%} \\
    \multicolumn{1}{r|}{} & \multicolumn{1}{r|}{} & \multicolumn{1}{r|}{} & 5     & \textbf{0.0\%} & \textbf{0.0\%} & \textbf{0.0\%} & \textbf{0.0\%} & \textbf{0.0\%} & \textbf{0.0\%} & \textbf{0.0\%} \\
    \cmidrule{3-11} \multicolumn{1}{r|}{} & \multicolumn{1}{r|}{} & \multicolumn{1}{r|}{\multirow{3}[1]{*}{20}} & 1     & \textbf{0.0\%} & 0.1\% & \textbf{0.0\%} & \textbf{0.0\%} & \textbf{0.0\%} & 0.1\% & \textbf{0.0\%} \\
    \multicolumn{1}{r|}{} & \multicolumn{1}{r|}{} & \multicolumn{1}{r|}{} & 3     & 3.3\% & 3.9\% & 3.4\% & 3.0\% & 2.9\% & \textbf{2.8\%} & \textbf{2.8\%} \\
    \multicolumn{1}{r|}{} & \multicolumn{1}{r|}{} & \multicolumn{1}{r|}{} & 5     & 5.5\% & 5.9\% & 5.6\% & 4.8\% & 4.8\% & \textbf{4.7\%} & \textbf{4.7\%} \\
    \midrule
    \multicolumn{4}{r|}{Average}  & 7.4\% & 7.0\% & 5.5\% & 4.2\% & 4.1\% & \textbf{3.9\%} & 4.1\% \\
    \bottomrule
    \end{tabular}%
  \label{tab:cc_gap}
\end{table}

\begin{table}[htb]
	\centering
	\scriptsize
	\setlength{\tabcolsep}{2pt}
	\caption{Nodes explored comparison across all techniques for SOC-CC benchmark.}
	\begin{tabular}{rrrr|rrrrrrr}
    \toprule
          &       &       & \multicolumn{1}{r}{} & \multicolumn{7}{c}{\textbf{\# Nodes Explored}} \\
    \midrule
    \multicolumn{1}{r|}{$t$} & \multicolumn{1}{r|}{$n$} & \multicolumn{1}{r|}{$m$} & $\Omega$     & \cplex & \bddFlow    & \bddFlowLift   & \bddGeneral    & \bddGeneralLift   & \bddTarget    & \bddTargetFlow \\
    \midrule
    \multicolumn{1}{r|}{\multirow{18}[12]{*}{0.1}} & \multicolumn{1}{r|}{\multirow{6}[4]{*}{75}} & \multicolumn{1}{r|}{\multirow{3}[2]{*}{10}} & 1     &       12,185  &           17,956  &          4,756  &              676  &              426  &              422  & \textbf{             399} \\
    \multicolumn{1}{r|}{} & \multicolumn{1}{r|}{} & \multicolumn{1}{r|}{} & 3     &       56,780  &           11,892  &              674  &              363  &              338  &              327  & \textbf{             292} \\
    \multicolumn{1}{r|}{} & \multicolumn{1}{r|}{} & \multicolumn{1}{r|}{} & 5     &       48,045  &              3,821  &              356  &              243  &              208  & \textbf{             178} &              214  \\
\cmidrule{3-11}    \multicolumn{1}{r|}{} & \multicolumn{1}{r|}{} & \multicolumn{1}{r|}{\multirow{3}[2]{*}{20}} & 1     &    202,904  &        215,447  &    142,030  &       26,079  & \textbf{      16,928} &       24,498  &       18,745  \\
    \multicolumn{1}{r|}{} & \multicolumn{1}{r|}{} & \multicolumn{1}{r|}{} & 3     &       13,220  &              1,767  &                 13  & \textbf{                   6} &                    7  &                    8  &                    7  \\
    \multicolumn{1}{r|}{} & \multicolumn{1}{r|}{} & \multicolumn{1}{r|}{} & 5     &       11,944  &                      24  & \textbf{                   4} &                    5  & \textbf{                   4} & \textbf{                   4} & \textbf{                   4} \\
\cmidrule{2-11}    \multicolumn{1}{r|}{} & \multicolumn{1}{r|}{\multirow{6}[4]{*}{100}} & \multicolumn{1}{r|}{\multirow{3}[2]{*}{10}} & 1     &    225,098  &        419,363  &    276,059  &       96,340  &    105,972  &    112,095  & \textbf{      89,448} \\
    \multicolumn{1}{r|}{} & \multicolumn{1}{r|}{} & \multicolumn{1}{r|}{} & 3     &    887,680  &        783,566  &    109,802  &       58,768  &       53,628  & \textbf{      46,936} &       53,654  \\
    \multicolumn{1}{r|}{} & \multicolumn{1}{r|}{} & \multicolumn{1}{r|}{} & 5     &    645,651  &        199,487  &       20,966  &       15,700  &       18,833  & \textbf{      11,710} &       12,749  \\
\cmidrule{3-11}    \multicolumn{1}{r|}{} & \multicolumn{1}{r|}{} & \multicolumn{1}{r|}{\multirow{3}[2]{*}{20}} & 1     &  -    &  -    &  -    &  -    &  -    &  -    &  -  \\
    \multicolumn{1}{r|}{} & \multicolumn{1}{r|}{} & \multicolumn{1}{r|}{} & 3     &    351,411  &           45,407  &          7,779  &          5,275  &          5,923  & \textbf{         5,054} &          5,124  \\
    \multicolumn{1}{r|}{} & \multicolumn{1}{r|}{} & \multicolumn{1}{r|}{} & 5     &    270,319  &           14,754  &          5,028  &          2,206  &          2,372  &          2,212  & \textbf{         2,192} \\
\cmidrule{2-11}    \multicolumn{1}{r|}{} & \multicolumn{1}{r|}{\multirow{6}[4]{*}{125}} & \multicolumn{1}{r|}{\multirow{3}[2]{*}{10}} & 1     &    247,942  &        448,522  &    193,963  &       92,868  &       78,944  &       71,829  & \textbf{      53,197} \\
    \multicolumn{1}{r|}{} & \multicolumn{1}{r|}{} & \multicolumn{1}{r|}{} & 3     &  -    &  -    &  -    &  -    &  -    &  -    &  -  \\
    \multicolumn{1}{r|}{} & \multicolumn{1}{r|}{} & \multicolumn{1}{r|}{} & 5     &  -    &  -    &  -    &  -    &  -    &  -    &  -  \\
\cmidrule{3-11}    \multicolumn{1}{r|}{} & \multicolumn{1}{r|}{} & \multicolumn{1}{r|}{\multirow{3}[2]{*}{20}} & 1     &  -    &  -    &  -    &  -    &  -    &  -    &  -  \\
    \multicolumn{1}{r|}{} & \multicolumn{1}{r|}{} & \multicolumn{1}{r|}{} & 3     &  -    &  -    &  -    &  -    &  -    &  -    &  -  \\
    \multicolumn{1}{r|}{} & \multicolumn{1}{r|}{} & \multicolumn{1}{r|}{} & 5     &  -    &  -    &  -    &  -    &  -    &  -    &  -  \\
    \midrule
    \multicolumn{1}{r|}{\multirow{18}[12]{*}{0.2}} & \multicolumn{1}{r|}{\multirow{6}[4]{*}{75}} & \multicolumn{1}{r|}{\multirow{3}[2]{*}{10}} & 1     &       12,121  &           20,715  &          7,593  &          1,244  & \textbf{         1,094} &          1,128  &          1,264  \\
    \multicolumn{1}{r|}{} & \multicolumn{1}{r|}{} & \multicolumn{1}{r|}{} & 3     &    467,593  &    1,009,053  &    318,850  &       65,594  & \textbf{      26,086} &       36,381  &       26,449  \\
    \multicolumn{1}{r|}{} & \multicolumn{1}{r|}{} & \multicolumn{1}{r|}{} & 5     &    465,102  &        501,715  &    260,340  &    141,485  &    105,020  & \textbf{      58,315} &    117,946  \\
\cmidrule{3-11}    \multicolumn{1}{r|}{} & \multicolumn{1}{r|}{} & \multicolumn{1}{r|}{\multirow{3}[2]{*}{20}} & 1     &    188,806  &        276,917  &    131,661  &       14,144  &       12,410  &       14,528  & \textbf{         9,460} \\
    \multicolumn{1}{r|}{} & \multicolumn{1}{r|}{} & \multicolumn{1}{r|}{} & 3     &  -    &  -    &  -    &  -    &  -    &  -    &  -  \\
    \multicolumn{1}{r|}{} & \multicolumn{1}{r|}{} & \multicolumn{1}{r|}{} & 5     &  -    &  -    &  -    &  -    &  -    &  -    &  -  \\
\cmidrule{2-11}    \multicolumn{1}{r|}{} & \multicolumn{1}{r|}{\multirow{6}[4]{*}{100}} & \multicolumn{1}{r|}{\multirow{3}[2]{*}{10}} & 1     &       95,393  &        248,309  &    151,732  &       52,426  &       33,054  &       51,027  & \textbf{      29,836} \\
    \multicolumn{1}{r|}{} & \multicolumn{1}{r|}{} & \multicolumn{1}{r|}{} & 3     &  -    &  -    &  -    &  -    &  -    &  -    &  -  \\
    \multicolumn{1}{r|}{} & \multicolumn{1}{r|}{} & \multicolumn{1}{r|}{} & 5     &  -    &  -    &  -    &  -    &  -    &  -    &  -  \\
\cmidrule{3-11}    \multicolumn{1}{r|}{} & \multicolumn{1}{r|}{} & \multicolumn{1}{r|}{\multirow{3}[2]{*}{20}} & 1     &  -    &  -    &  -    &  -    &  -    &  -    &  -  \\
    \multicolumn{1}{r|}{} & \multicolumn{1}{r|}{} & \multicolumn{1}{r|}{} & 3     &  -    &  -    &  -    &  -    &  -    &  -    &  -  \\
    \multicolumn{1}{r|}{} & \multicolumn{1}{r|}{} & \multicolumn{1}{r|}{} & 5     &  -    &  -    &  -    &  -    &  -    &  -    &  -  \\
\cmidrule{2-11}    \multicolumn{1}{r|}{} & \multicolumn{1}{r|}{\multirow{6}[4]{*}{125}} & \multicolumn{1}{r|}{\multirow{3}[2]{*}{10}} & 1     & \textbf{      49,518} &        312,285  &    255,539  &    312,597  &    204,109  &    213,868  &    230,317  \\
    \multicolumn{1}{r|}{} & \multicolumn{1}{r|}{} & \multicolumn{1}{r|}{} & 3     &  -    &  -    &  -    &  -    &  -    &  -    &  -  \\
    \multicolumn{1}{r|}{} & \multicolumn{1}{r|}{} & \multicolumn{1}{r|}{} & 5     &  -    &  -    &  -    &  -    &  -    &  -    &  -  \\
\cmidrule{3-11}    \multicolumn{1}{r|}{} & \multicolumn{1}{r|}{} & \multicolumn{1}{r|}{\multirow{3}[2]{*}{20}} & 1     &  -    &  -    &  -    &  -    &  -    &  -    &  -  \\
    \multicolumn{1}{r|}{} & \multicolumn{1}{r|}{} & \multicolumn{1}{r|}{} & 3     &  -    &  -    &  -    &  -    &  -    &  -    &  -  \\
    \multicolumn{1}{r|}{} & \multicolumn{1}{r|}{} & \multicolumn{1}{r|}{} & 5     &  -    &  -    &  -    &  -    &  -    &  -    &  -  \\
    \midrule
    \multicolumn{1}{r|}{\multirow{18}[12]{*}{0.3}} & \multicolumn{1}{r|}{\multirow{6}[4]{*}{75}} & \multicolumn{1}{r|}{\multirow{3}[2]{*}{10}} & 1     &          1,749  &              4,503  &          1,480  &              617  &              628  & \textbf{             471} &              693  \\
    \multicolumn{1}{r|}{} & \multicolumn{1}{r|}{} & \multicolumn{1}{r|}{} & 3     &    105,042  &        289,610  &       88,043  &          4,729  &          3,330  &          3,172  & \textbf{         2,798} \\
    \multicolumn{1}{r|}{} & \multicolumn{1}{r|}{} & \multicolumn{1}{r|}{} & 5     &    106,290  &        166,623  &       40,851  & \textbf{         2,344} &          2,990  &          3,066  &          3,140  \\
\cmidrule{3-11}    \multicolumn{1}{r|}{} & \multicolumn{1}{r|}{} & \multicolumn{1}{r|}{\multirow{3}[2]{*}{20}} & 1     &       46,165  &           92,516  &       64,630  &          7,094  &          3,820  & \textbf{         3,340} &          3,928  \\
    \multicolumn{1}{r|}{} & \multicolumn{1}{r|}{} & \multicolumn{1}{r|}{} & 3     &  -    &  -    &  -    &  -    &  -    &  -    &  -  \\
    \multicolumn{1}{r|}{} & \multicolumn{1}{r|}{} & \multicolumn{1}{r|}{} & 5     &  -    &  -    &  -    &  -    &  -    &  -    &  -  \\
\cmidrule{2-11}    \multicolumn{1}{r|}{} & \multicolumn{1}{r|}{\multirow{6}[4]{*}{100}} & \multicolumn{1}{r|}{\multirow{3}[2]{*}{10}} & 1     &          4,668  &           13,241  &       12,104  &          2,445  &          4,007  & \textbf{         2,414} &          3,902  \\
    \multicolumn{1}{r|}{} & \multicolumn{1}{r|}{} & \multicolumn{1}{r|}{} & 3     &       98,320  &        232,323  &    215,941  &       41,555  &       54,620  & \textbf{      41,508} &       48,554  \\
    \multicolumn{1}{r|}{} & \multicolumn{1}{r|}{} & \multicolumn{1}{r|}{} & 5     &    325,561  &        779,990  &    647,941  & \textbf{   130,800} &    160,442  &    160,147  &    136,534  \\
\cmidrule{3-11}    \multicolumn{1}{r|}{} & \multicolumn{1}{r|}{} & \multicolumn{1}{r|}{\multirow{3}[2]{*}{20}} & 1     &       73,317  &        104,970  &    110,269  & \textbf{      37,389} &       39,225  &       37,523  &       40,692  \\
    \multicolumn{1}{r|}{} & \multicolumn{1}{r|}{} & \multicolumn{1}{r|}{} & 3     &  -    &  -    &  -    &  -    &  -    &  -    &  -  \\
    \multicolumn{1}{r|}{} & \multicolumn{1}{r|}{} & \multicolumn{1}{r|}{} & 5     &  -    &  -    &  -    &  -    &  -    &  -    &  -  \\
\cmidrule{2-11}    \multicolumn{1}{r|}{} & \multicolumn{1}{r|}{\multirow{6}[4]{*}{125}} & \multicolumn{1}{r|}{\multirow{3}[2]{*}{10}} & 1     & \textbf{             621} &              7,117  &          5,397  &          6,667  &          3,226  &          7,361  &          4,736  \\
    \multicolumn{1}{r|}{} & \multicolumn{1}{r|}{} & \multicolumn{1}{r|}{} & 3     &    115,889  &        289,327  &    310,125  &    140,877  &    117,999  & \textbf{   105,835} &    108,763  \\
    \multicolumn{1}{r|}{} & \multicolumn{1}{r|}{} & \multicolumn{1}{r|}{} & 5     & \textbf{   225,090} &        447,631  &    374,133  &    292,988  &    495,361  &    332,646  &    334,653  \\
\cmidrule{3-11}    \multicolumn{1}{r|}{} & \multicolumn{1}{r|}{} & \multicolumn{1}{r|}{\multirow{3}[2]{*}{20}} & 1     & \textbf{   130,174} &        374,273  &    234,563  &    304,234  &    188,323  &    319,374  &    155,635  \\
    \multicolumn{1}{r|}{} & \multicolumn{1}{r|}{} & \multicolumn{1}{r|}{} & 3     &  -    &  -    &  -    &  -    &  -    &  -    &  -  \\
    \multicolumn{1}{r|}{} & \multicolumn{1}{r|}{} & \multicolumn{1}{r|}{} & 5     &  -    &  -    &  -    &  -    &  -    &  -    &  -  \\
    \midrule
    \multicolumn{4}{r|}{Average}  &    182,820  &        244,438  &    133,087  &       61,925  &       57,977  &       55,579  & \textbf{      49,844} \\
    \bottomrule
    \end{tabular}%
 \label{tab:cc_nodes}
\end{table}

\begin{table}[htb]
	\centering
	\scriptsize
	\setlength{\tabcolsep}{3pt}
	\caption{Solving time comparison across all techniques for SOC-CC benchmark.}
	\begin{tabular}{rrrr|rrrrrrr}
    \toprule
          &       &       & \multicolumn{1}{r}{} & \multicolumn{7}{c}{\textbf{Solving Time (s)}} \\
    \midrule
    \multicolumn{1}{r|}{$t$} & \multicolumn{1}{r|}{$n$} & \multicolumn{1}{r|}{$m$} & $\Omega$     & \cplex & \bddFlow    & \bddFlowLift   & \bddGeneral    & \bddGeneralLift   & \bddTarget    & \bddTargetFlow \\
    \midrule
    \multicolumn{1}{r|}{\multirow{18}[12]{*}{0.1}} & \multicolumn{1}{r|}{\multirow{6}[4]{*}{75}} & \multicolumn{1}{r|}{\multirow{3}[2]{*}{10}} & 1     &           22.8  &           18.5  & \textbf{             9.5} &           50.1  &          17.9  &           18.1  &          12.1  \\
    \multicolumn{1}{r|}{} & \multicolumn{1}{r|}{} & \multicolumn{1}{r|}{} & 3     &           54.2  &           14.4  & \textbf{             4.3} &           17.8  &             5.1  &           10.4  &             5.1  \\
    \multicolumn{1}{r|}{} & \multicolumn{1}{r|}{} & \multicolumn{1}{r|}{} & 5     &           52.5  &              8.2  & \textbf{             3.4} &           10.5  &             3.5  &              6.7  &             3.7  \\
\cmidrule{3-11}    \multicolumn{1}{r|}{} & \multicolumn{1}{r|}{} & \multicolumn{1}{r|}{\multirow{3}[2]{*}{20}} & 1     &    1,184.5  &        641.2  &        531.6  &        438.8  &       244.4  &        178.2  & \textbf{      165.2} \\
    \multicolumn{1}{r|}{} & \multicolumn{1}{r|}{} & \multicolumn{1}{r|}{} & 3     &           58.3  &           24.9  & \textbf{          10.3} &           32.3  &          11.1  &           29.5  &          11.7  \\
    \multicolumn{1}{r|}{} & \multicolumn{1}{r|}{} & \multicolumn{1}{r|}{} & 5     &           59.8  &           15.2  & \textbf{             8.5} &           15.7  &             8.9  &           13.8  &             9.0  \\
\cmidrule{2-11}    \multicolumn{1}{r|}{} & \multicolumn{1}{r|}{\multirow{6}[4]{*}{100}} & \multicolumn{1}{r|}{\multirow{3}[2]{*}{10}} & 1     &        557.1  &        494.3  &        368.6  &        218.8  &       176.5  &        173.3  & \textbf{      141.8} \\
    \multicolumn{1}{r|}{} & \multicolumn{1}{r|}{} & \multicolumn{1}{r|}{} & 3     &    1,648.1  &        948.1  &        216.3  &        194.8  & \textbf{      122.0} &        135.2  &       132.3  \\
    \multicolumn{1}{r|}{} & \multicolumn{1}{r|}{} & \multicolumn{1}{r|}{} & 5     &    1,039.0  &        327.3  &           46.1  &           87.1  &          55.3  &           54.9  & \textbf{         41.7} \\
\cmidrule{3-11}    \multicolumn{1}{r|}{} & \multicolumn{1}{r|}{} & \multicolumn{1}{r|}{\multirow{3}[2]{*}{20}} & 1     &  -    &  -    &  -    &  -    &  -    &  -    &  -  \\
    \multicolumn{1}{r|}{} & \multicolumn{1}{r|}{} & \multicolumn{1}{r|}{} & 3     &    1,733.5  &        234.3  & \textbf{          49.2} &           99.3  &          50.1  &        108.3  &          52.4  \\
    \multicolumn{1}{r|}{} & \multicolumn{1}{r|}{} & \multicolumn{1}{r|}{} & 5     &    1,177.4  &           87.7  &           36.9  &           53.1  & \textbf{         31.2} &           81.6  &          31.2  \\
\cmidrule{2-11}    \multicolumn{1}{r|}{} & \multicolumn{1}{r|}{\multirow{6}[4]{*}{125}} & \multicolumn{1}{r|}{\multirow{3}[2]{*}{10}} & 1     &        547.3  &        469.8  &        366.5  &        154.1  &       140.2  &           95.1  & \textbf{         83.3} \\
    \multicolumn{1}{r|}{} & \multicolumn{1}{r|}{} & \multicolumn{1}{r|}{} & 3     &  -    &  -    &  -    &  -    &  -    &  -    &  -  \\
    \multicolumn{1}{r|}{} & \multicolumn{1}{r|}{} & \multicolumn{1}{r|}{} & 5     &  -    &  -    &  -    &  -    &  -    &  -    &  -  \\
\cmidrule{3-11}    \multicolumn{1}{r|}{} & \multicolumn{1}{r|}{} & \multicolumn{1}{r|}{\multirow{3}[2]{*}{20}} & 1     &  -    &  -    &  -    &  -    &  -    &  -    &  -  \\
    \multicolumn{1}{r|}{} & \multicolumn{1}{r|}{} & \multicolumn{1}{r|}{} & 3     &  -    &  -    &  -    &  -    &  -    &  -    &  -  \\
    \multicolumn{1}{r|}{} & \multicolumn{1}{r|}{} & \multicolumn{1}{r|}{} & 5     &  -    &  -    &  -    &  -    &  -    &  -    &  -  \\
    \midrule
    \multicolumn{1}{r|}{\multirow{18}[12]{*}{0.2}} & \multicolumn{1}{r|}{\multirow{6}[4]{*}{75}} & \multicolumn{1}{r|}{\multirow{3}[2]{*}{10}} & 1     &           20.8  &           21.4  & \textbf{          13.6} &           39.6  &          29.8  &           25.7  &          18.6  \\
    \multicolumn{1}{r|}{} & \multicolumn{1}{r|}{} & \multicolumn{1}{r|}{} & 3     &    1,851.3  &    1,889.4  &        892.7  &        192.2  &          99.6  & \textbf{          74.5} &          84.4  \\
    \multicolumn{1}{r|}{} & \multicolumn{1}{r|}{} & \multicolumn{1}{r|}{} & 5     &    1,418.2  &        854.6  &        490.0  &        356.2  &       337.9  & \textbf{       164.8} &       199.9  \\
\cmidrule{3-11}    \multicolumn{1}{r|}{} & \multicolumn{1}{r|}{} & \multicolumn{1}{r|}{\multirow{3}[2]{*}{20}} & 1     &    1,361.3  &    1,216.9  &        776.6  &        307.7  &       192.4  &        176.1  & \textbf{      114.0} \\
    \multicolumn{1}{r|}{} & \multicolumn{1}{r|}{} & \multicolumn{1}{r|}{} & 3     &  -    &  -    &  -    &  -    &  -    &  -    &  -  \\
    \multicolumn{1}{r|}{} & \multicolumn{1}{r|}{} & \multicolumn{1}{r|}{} & 5     &  -    &  -    &  -    &  -    &  -    &  -    &  -  \\
\cmidrule{2-11}    \multicolumn{1}{r|}{} & \multicolumn{1}{r|}{\multirow{6}[4]{*}{100}} & \multicolumn{1}{r|}{\multirow{3}[2]{*}{10}} & 1     &        209.1  &        257.1  &        169.8  &           93.5  &          63.5  &           83.8  & \textbf{         47.5} \\
    \multicolumn{1}{r|}{} & \multicolumn{1}{r|}{} & \multicolumn{1}{r|}{} & 3     &  -    &  -    &  -    &  -    &  -    &  -    &  -  \\
    \multicolumn{1}{r|}{} & \multicolumn{1}{r|}{} & \multicolumn{1}{r|}{} & 5     &  -    &  -    &  -    &  -    &  -    &  -    &  -  \\
\cmidrule{3-11}    \multicolumn{1}{r|}{} & \multicolumn{1}{r|}{} & \multicolumn{1}{r|}{\multirow{3}[2]{*}{20}} & 1     &  -    &  -    &  -    &  -    &  -    &  -    &  -  \\
    \multicolumn{1}{r|}{} & \multicolumn{1}{r|}{} & \multicolumn{1}{r|}{} & 3     &  -    &  -    &  -    &  -    &  -    &  -    &  -  \\
    \multicolumn{1}{r|}{} & \multicolumn{1}{r|}{} & \multicolumn{1}{r|}{} & 5     &  -    &  -    &  -    &  -    &  -    &  -    &  -  \\
\cmidrule{2-11}    \multicolumn{1}{r|}{} & \multicolumn{1}{r|}{\multirow{6}[4]{*}{125}} & \multicolumn{1}{r|}{\multirow{3}[2]{*}{10}} & 1     & \textbf{       112.0} &        239.0  &        228.6  &        255.0  &       190.2  &        210.2  &       205.3  \\
    \multicolumn{1}{r|}{} & \multicolumn{1}{r|}{} & \multicolumn{1}{r|}{} & 3     &  -    &  -    &  -    &  -    &  -    &  -    &  -  \\
    \multicolumn{1}{r|}{} & \multicolumn{1}{r|}{} & \multicolumn{1}{r|}{} & 5     &  -    &  -    &  -    &  -    &  -    &  -    &  -  \\
\cmidrule{3-11}    \multicolumn{1}{r|}{} & \multicolumn{1}{r|}{} & \multicolumn{1}{r|}{\multirow{3}[2]{*}{20}} & 1     &  -    &  -    &  -    &  -    &  -    &  -    &  -  \\
    \multicolumn{1}{r|}{} & \multicolumn{1}{r|}{} & \multicolumn{1}{r|}{} & 3     &  -    &  -    &  -    &  -    &  -    &  -    &  -  \\
    \multicolumn{1}{r|}{} & \multicolumn{1}{r|}{} & \multicolumn{1}{r|}{} & 5     &  -    &  -    &  -    &  -    &  -    &  -    &  -  \\
    \midrule
    \multicolumn{1}{r|}{\multirow{18}[12]{*}{0.3}} & \multicolumn{1}{r|}{\multirow{6}[4]{*}{75}} & \multicolumn{1}{r|}{\multirow{3}[2]{*}{10}} & 1     & \textbf{             4.7} &           11.3  &              9.3  &           33.3  &          19.1  &           25.7  &          20.6  \\
    \multicolumn{1}{r|}{} & \multicolumn{1}{r|}{} & \multicolumn{1}{r|}{} & 3     &        249.8  &        392.1  &        109.6  &        137.0  &          93.4  &           63.0  & \textbf{         53.4} \\
    \multicolumn{1}{r|}{} & \multicolumn{1}{r|}{} & \multicolumn{1}{r|}{} & 5     &        214.8  &        233.3  &           78.6  &        130.3  &          66.1  &           63.5  & \textbf{         40.3} \\
\cmidrule{3-11}    \multicolumn{1}{r|}{} & \multicolumn{1}{r|}{} & \multicolumn{1}{r|}{\multirow{3}[2]{*}{20}} & 1     &        236.1  &        270.0  &        200.0  &        161.2  &          84.8  &           78.1  & \textbf{         64.7} \\
    \multicolumn{1}{r|}{} & \multicolumn{1}{r|}{} & \multicolumn{1}{r|}{} & 3     &  -    &  -    &  -    &  -    &  -    &  -    &  -  \\
    \multicolumn{1}{r|}{} & \multicolumn{1}{r|}{} & \multicolumn{1}{r|}{} & 5     &  -    &  -    &  -    &  -    &  -    &  -    &  -  \\
\cmidrule{2-11}    \multicolumn{1}{r|}{} & \multicolumn{1}{r|}{\multirow{6}[4]{*}{100}} & \multicolumn{1}{r|}{\multirow{3}[2]{*}{10}} & 1     & \textbf{          11.0} &           21.2  &           21.8  &           31.6  &          25.4  &           27.0  &          24.3  \\
    \multicolumn{1}{r|}{} & \multicolumn{1}{r|}{} & \multicolumn{1}{r|}{} & 3     &        245.3  &        374.1  &        319.2  &        118.7  &       144.6  & \textbf{       105.6} &       122.1  \\
    \multicolumn{1}{r|}{} & \multicolumn{1}{r|}{} & \multicolumn{1}{r|}{} & 5     &    1,081.2  &    1,282.5  &    1,448.4  & \textbf{       348.5} &       438.7  &        446.1  &       390.4  \\
\cmidrule{3-11}    \multicolumn{1}{r|}{} & \multicolumn{1}{r|}{} & \multicolumn{1}{r|}{\multirow{3}[2]{*}{20}} & 1     &        550.6  &        315.5  &        317.7  &        166.8  &       144.2  & \textbf{       137.7} &       176.7  \\
    \multicolumn{1}{r|}{} & \multicolumn{1}{r|}{} & \multicolumn{1}{r|}{} & 3     &  -    &  -    &  -    &  -    &  -    &  -    &  -  \\
    \multicolumn{1}{r|}{} & \multicolumn{1}{r|}{} & \multicolumn{1}{r|}{} & 5     &  -    &  -    &  -    &  -    &  -    &  -    &  -  \\
\cmidrule{2-11}    \multicolumn{1}{r|}{} & \multicolumn{1}{r|}{\multirow{6}[4]{*}{125}} & \multicolumn{1}{r|}{\multirow{3}[2]{*}{10}} & 1     & \textbf{             2.5} &           19.2  &           18.5  &           24.2  &          19.2  &           25.0  &          22.9  \\
    \multicolumn{1}{r|}{} & \multicolumn{1}{r|}{} & \multicolumn{1}{r|}{} & 3     &        301.2  &        468.9  &        445.3  &        238.9  & \textbf{      146.6} &        196.7  &       152.8  \\
    \multicolumn{1}{r|}{} & \multicolumn{1}{r|}{} & \multicolumn{1}{r|}{} & 5     &        655.2  &        703.1  &        566.7  &        439.0  &       817.1  & \textbf{       418.0} &       466.0  \\
\cmidrule{3-11}    \multicolumn{1}{r|}{} & \multicolumn{1}{r|}{} & \multicolumn{1}{r|}{\multirow{3}[2]{*}{20}} & 1     &    1,255.8  &    1,490.4  & \textbf{       639.7} &    1,251.0  &       962.6  &    1,211.9  &       752.7  \\
    \multicolumn{1}{r|}{} & \multicolumn{1}{r|}{} & \multicolumn{1}{r|}{} & 3     &  -    &  -    &  -    &  -    &  -    &  -    &  -  \\
    \multicolumn{1}{r|}{} & \multicolumn{1}{r|}{} & \multicolumn{1}{r|}{} & 5     &  -    &  -    &  -    &  -    &  -    &  -    &  -  \\
    \midrule
    \multicolumn{4}{r|}{Average}  &     597.18  &     444.80  &     279.91  &     189.90  &    158.05  &     147.95  & \textbf{   121.54} \\
    \bottomrule
    \end{tabular}%
   \label{tab:cc_time}
\end{table}

\end{document}